%% file: main.tex
\input{macros}

\title{Algorithms for adaptive and heteroskedastic linear regression at the computational threshold}
\author[]{Spencer Compton}
\author[]{Tselil Schramm}
\affil[]{Stanford University \authorcr
  \{\tt comptons, tselil\}@stanford.edu}

\begin{document}

\maketitle

\begin{abstract}

We study finite-sample linear regression in the presence of varied and unknown label noise, focusing on the \emph{heteroskedastic} and \emph{adaptive} linear regression models.

\textit{Heteroskedastic} linear regression models regression problems in which the labels or measurements are of varying quality. 
We receive $n$ covariate, label pairs $(X_i,Y_i)$ with labels $Y_i = X_i^\top \beta + \eps_i$, where each $\eps_i \sim N(0,\sigma_i^2)$ for a possibly distinct variance $\sigma_i^2$, and the variances $\sigma_i^2$ are unknown to the estimator. One natural measurement of the difficulty of this problem is the number of samples $m$ for which $\sigma_i^2 \le 1$ (larger $m$ is easier). Building on Compton and Valiant's results for the simpler problem of mean estimation (STOC 2024), we obtain a polynomial-time estimator with rate $\tilde{O}((nd^3/m^4)^{1/6})$ when $m \gg d^{3/4}n^{1/4}$, as well as nearly-matching lower bounds. For $d = O(1)$, our estimator achieves error $o(1)$ so long as $m \gg n^{1/4}$, whereas $L_1$ regression and other traditional approaches require $m \gg n^{1/2}$.

\textit{Adaptive} linear regression models regression problems in which we have limited information about the nature of the label noise. Here, the errors $\eps_i$ are drawn i.i.d. from an unknown distribution $p$, and our goal is to design a generic estimator that performs nearly as well as the best custom estimator that knows $p$. 
To appreciate the challenge of this task, note that different parametric families exhibit dramatically different behavior: for example, when $p$ is Normal the optimal rate is $\Theta(\sqrt{d/n})$, whereas if $p$ is Uniform$[-1,1]$ the optimal rate is $\tilde\Theta(d/n)$, and the maximum likelihood estimator in each case is quite different.
We introduce a (computationally inefficient) adaptive estimator that, so long as $p$ is a mixture of $k$ symmetric log-concave densities, achieves error comparable with the optimal estimator that knows $p$ and has $\tilde\Theta(n/k)$ samples. For the special case $k=1$, we show that $L_q$ regression (with data-dependent $q$) gives a polynomial-time estimator.

Finally, in order to study the computational limits of both problems, we introduce the \textit{planted linear regression} problem, where $X_i \sim N(0,I_d)$, $m$ unknown samples are noiseless, and the rest have error $\eps_i \sim N(0,1)$. 
We conjecture that recovering $\beta$ up to error $\ll \sqrt{d/n}$ (or exactly) may have an information-computation gap between $m=d+1$ and $m \sim d^{3/4}n^{1/4}$, as is suggested by our near-matching polynomial-time estimator and statistical query (SQ) lower bound.

\end{abstract}
\newpage
\setcounter{tocdepth}{2}
\tableofcontents
\newpage

\section{Introduction}
Linear regression is the bread and butter of statistical prediction and modeling. 
Developing linear regression estimators that best adapt to unknown noise distributions is a core challenge. There is a mature, insightful body of work for this task in the asymptotic setting, where typically the noise distribution is fixed, has finite Fisher information, and then $n \rightarrow \infty$.

In our work, we study linear regression with unknown noise in the \emph{finite-sample} setting, which is much less understood.  
We allow the noise distribution to be chosen after $n$, and we do not assume finite Fisher information.
A strikingly different picture emerges when we relax these assumptions: the optimal rates are different, the estimators achieving these rates are different, and the guarantees pertain to more problems of recent interest. Our study focuses on three representative tasks: {heteroskedastic}, {adaptive}, and {planted} linear regression.

\paragraph{Heteroskedastic linear regression.} 
How much signal can you extract from data when its labels have unknown, varying quality? 
This is a question that naturally must be answered when handling data from distinctive sources; for example, when data is scraped from multiple collections, or crowdsourced from users of unequal skill. Formally, in heteroskedastic linear regression we receive $n$ samples of the form $(X_i,Y_i)$, and each $Y_i = X_i^\top \beta + \eps_i$, where $\eps_i \sim N(0,\sigma_i^2)$ for different variances $\sigma_i^2$ with unknown values.

A comprehensive treatment of such problems was initiated by Chierichetti, Dasgupta, Kumar, and Lattanzi \cite{chierichetti2014learning} for heteroskedastic \emph{mean estimation}: where we aim to learn a one-dimensional mean $\mu$ from samples $X_i \sim N(\mu,\sigma_i^2)$ with unknown $\sigma_i^2$. In this setting, estimators such as the empirical mean are quite undesirable; even just one sample with extremely large variance can render the estimator useless. This has motivated the study of alternative estimators which are better-suited for this task, such as the empirical median, iterative trimming, and modal estimators \cite{chierichetti2014learning,xia2019non,yuan2020learning,pensia2022estimating,devroye2023mean,louati2025performance}. In the heteroskedastic setting, the rate achieved by each estimator may be a sophisticated function of $\sigma_1,\ldots,\sigma_n$, making it difficult to compare estimators. As a point of comparison, Liang and Yuan \cite{liang2020learning} suggested an interpretable ``Subset-of-Signals'' benchmark: ``What is the best error guaranteed by an estimator (as a function of $n,m$) if at least $m$ of the samples have $\sigma_i \le 1$, but the estimator is not told which samples?'' This perspective makes comparing the guarantees of different estimators much easier; estimators such as the median and other previously-discussed works require $m \ge \Omega(n^{1/2})$ samples to attain estimation error $\norm{\wh\mu - \mu}_2 \ll 1$, whereas a lower bound of \cite{liang2020learning} only requires $m \ge \Omega(n^{1/4})$. Compton and Valiant \cite{compton2024near} designed a new \textit{balance-finding} estimator for this setting, obtaining near-optimal errors for the Subset-of-Signals benchmark and demonstrating that error $\ll 1$ is possible even with only $m \ge n^{1/4 + o(1)}$ bounded-variance samples. 

Just as the empirical mean fails in the heteroskedastic setting, the OLS estimator degrades dramatically in the presence of even one high-variance label. 
There has been a similar study into alternative estimators for heteroskedastic regression \cite{yuan2020learning,pensia2022estimating,chaudhary2026linear}, and understanding the statistical and computational landscape has been highlighted as an open direction in the recent survey of \cite{maleki2026high}, where they remark that ``it is largely unclear whether polynomial-time algorithms can achieve optimal rates in these heterogeneous high-dimensional settings.'' 
A natural open question is whether there exists a polynomial-time estimator which similarly succeeds in the Subset-of-Signals setting for \emph{regression} when only $m \ge n^{1/4 +o(1)}$ out of the $n$ labels have bounded variance.
In our work, we resolve this positively: obtaining error $\ll 1$ when $m \gg d^{3/4}n^{1/4}$. More generally, we show this is statistically optimal, and obtain near-optimal Subset-of-Signals error in this regime of $m$.

\paragraph{Adaptive linear regression.} In the task of adaptive linear regression, errors are drawn i.i.d. from an unknown distribution $p$. The main challenge is to understand which conditions enable an adaptive estimator, which does not know $p$, to estimate nearly as well as a custom estimator with full knowledge of $p$.

Even if we are willing to restrict to $p$ which are symmetric and log-concave, this problem is quite nuanced. Consider three examples (with simulations in \cref{fig:plots}):
\begin{itemize}
    \item \textit{Gaussian errors: $N(0,1)$.} The optimal estimation error $\norm{\wh\beta - \beta}_2$ is on the order of $\sqrt{d/n}$; this is achieved by OLS.
    \item \textit{Uniform errors: $\operatorname{Unif}[-1,1]$.} Dramatically better estimation error of $\tilde{O}(d/n)$ is possible; this is achieved by minimizing the $L_\infty$ norm of the residuals \cite[Theorem 2.2, Example 2.10]{yi2024non}.
    \item \textit{Smoothed uniform errors: $\operatorname{Unif}[-1,1] * N(0,\sigma^2)$.} When $\sigma = n^{-a}$ for $a \in (0,1]$, then previously-known techniques do not determine the optimal error (to our knowledge); our later results will imply that the optimal error is on the order of $\max\left\{ \sqrt{\frac{d}{n^{1+a}}}, \frac{d}{n} \right\}$ (ignoring logarithmic factors; see \cref{sec:smoothed-uniform}), which is attained by neither OLS nor $L_\infty$ regression. 
\end{itemize}
\begin{figure}
\centering
\begin{subfigure}{.32\linewidth}
    \centering
    \includegraphics[width=\linewidth]{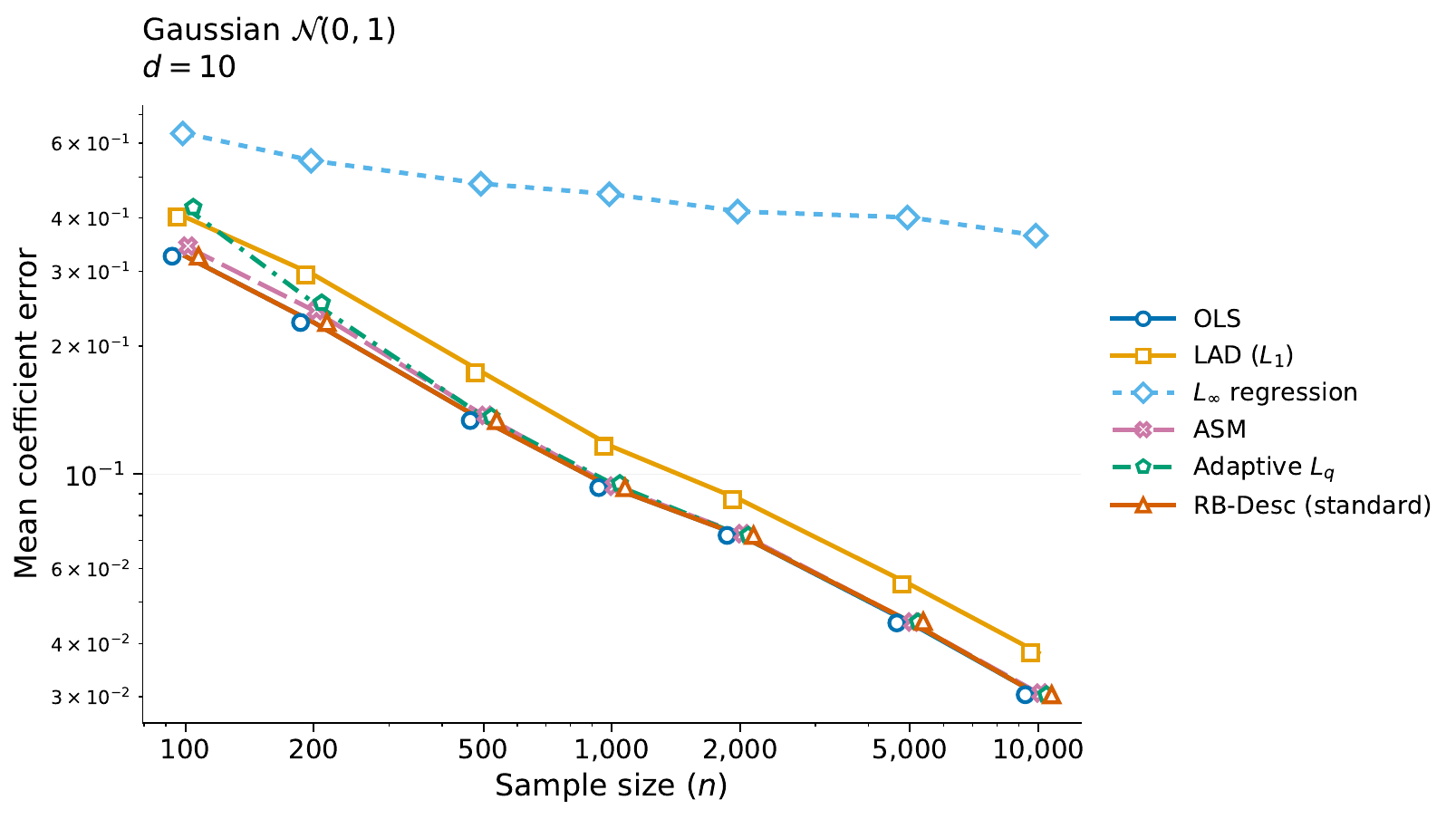}
    \caption{}\label{fig:plot-gaussian}
\end{subfigure}
    \hfill
\begin{subfigure}{.32\linewidth}
    \centering
    \includegraphics[width=\linewidth]{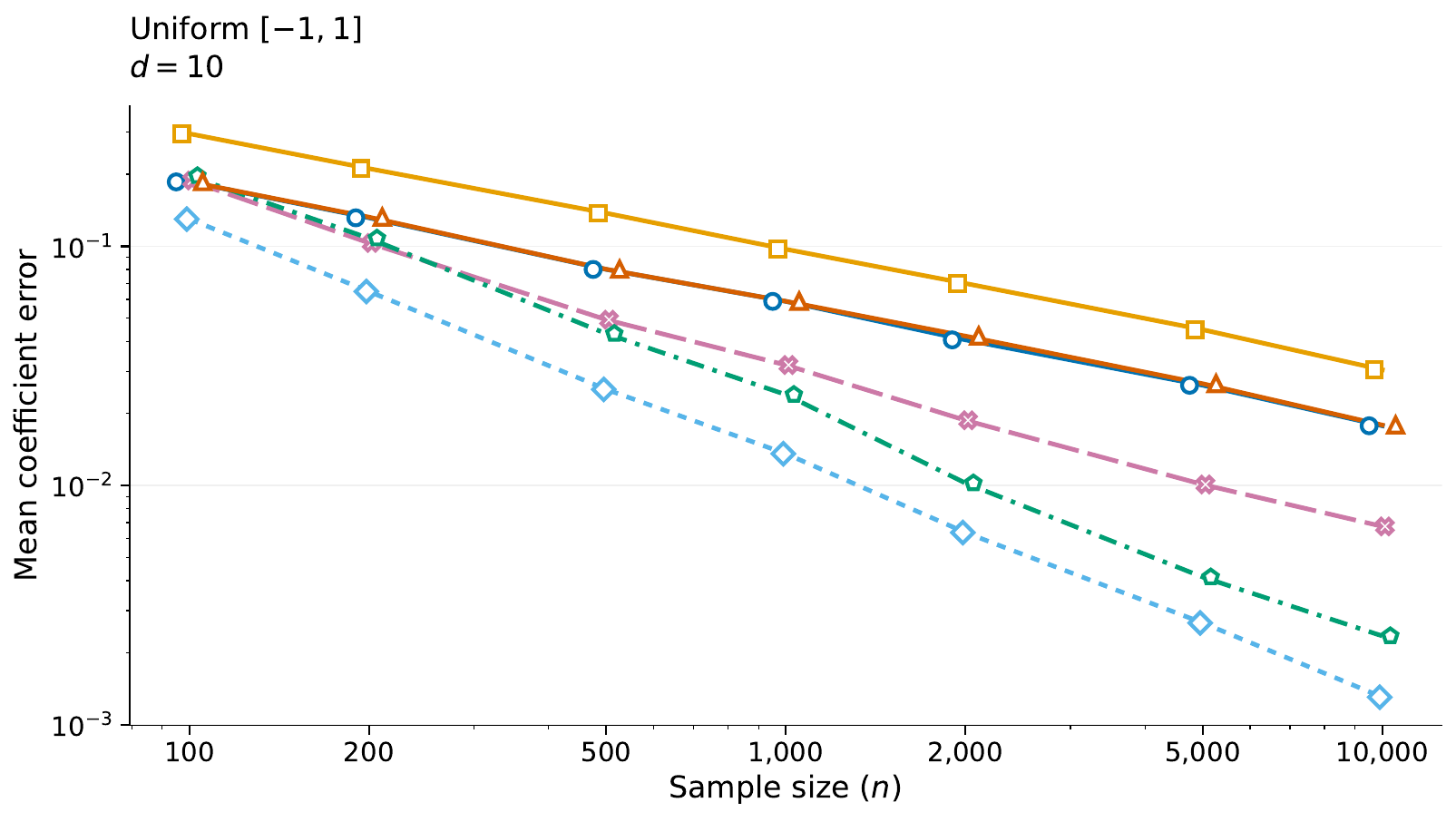}
    \caption{}\label{fig:plot-unif}
\end{subfigure}
   \hfill
\begin{subfigure}{.32\linewidth}
    \centering
    \includegraphics[width=\linewidth]{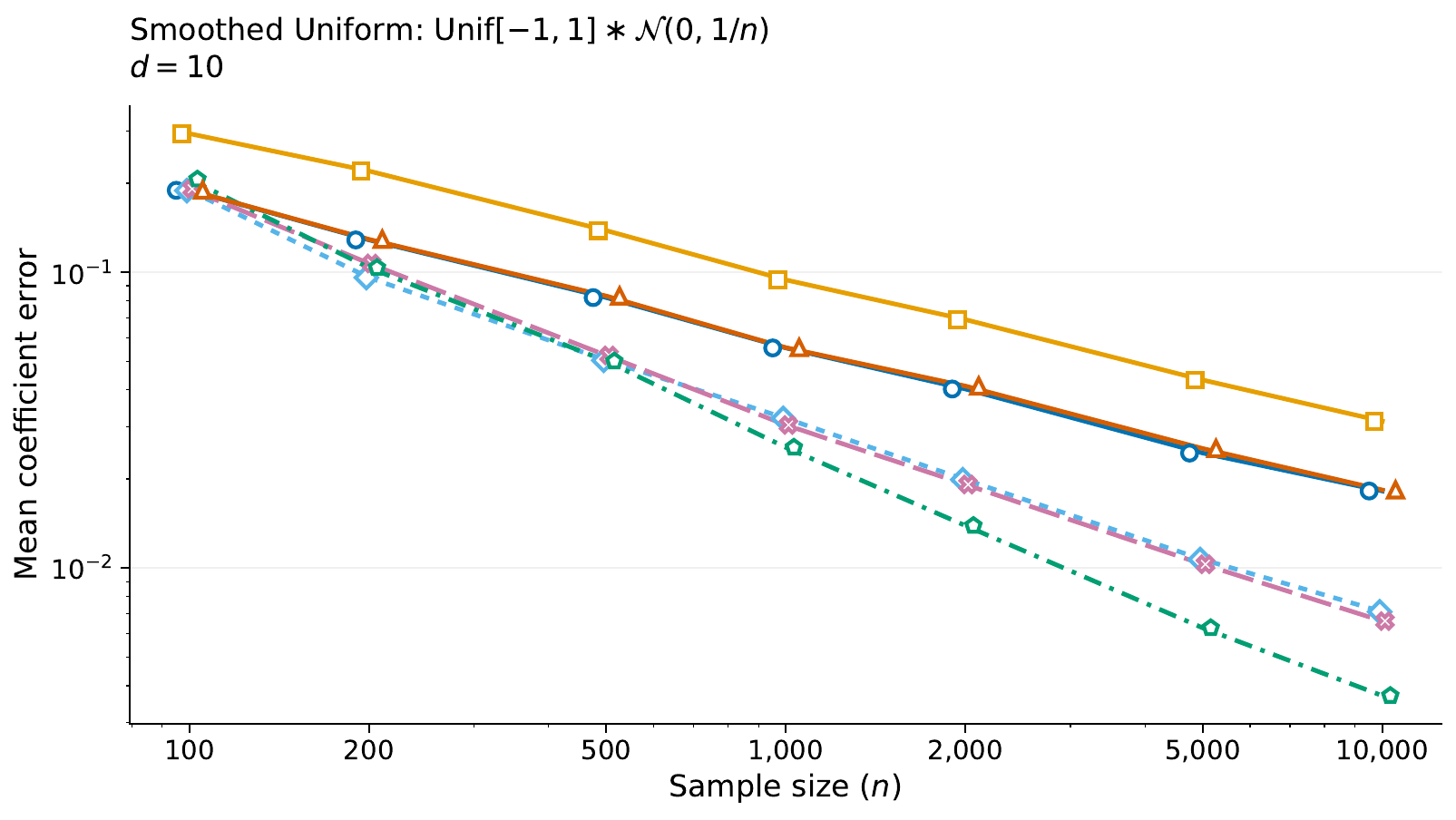}
    \caption{}\label{fig:plot-smoothed-unif}
\end{subfigure}

{\caption{Performance for Gaussian, uniform, and smoothed uniform errors (see \cref{sec:simulations}).}
\label{fig:plots}}
\end{figure}
In each case, the optimal rate is different, and is achieved by a different estimator. A key question is whether there exists a single estimator that may adapt to all such distributions. 

While analogous questions in the asymptotic setting are much better understood (we point to \cite{feng2026optimal} as an illuminating study), the limits of adaptivity with finite samples (and no bound on Fisher information) are comparatively uncharted. This concerns not only the simple example of uniform errors (with infinite Fisher information), but also more interesting errors such as: (i) smoothed uniform errors (where the smoothing may depend on $n$), or (ii) a mixture of two heteroskedastic mean-zero Gaussians (where the mixing weight and variances may depend on $n$). The heteroskedastic mixture is an instructive case for the importance of the finite-sample perspective; if we treated it asymptotically, both components would have a constant-fraction of the mixing weight, and we would not observe the interesting phase transition behaviors that appear in Subset-of-Signals heteroskedastic estimation.

In the case of symmetric, log-concave errors (which includes all examples in \cref{fig:plots}), we show that polynomial-time $L_q$ regression has the desired adaptive behavior when $q$ is chosen in a data-dependent manner. This style of estimator is comprehensively studied by Kao, Xu, and Zhang \cite{kao2024choosing}, but our result does not follow from existing analyses; we discuss the technical obstacles later. Roughly, our result proves that for any symmetric, log-concave $p$, this $L_q$ estimator given $n$ samples performs nearly as well as \textit{the best estimator that knows $p$} when it is given access to $\frac{n}{\polylog(n)}$ samples.

For more general errors, the most relevant prior work is a different paper of Compton and Valiant \cite{compton2026attainability}, where they study the analogous problems in the setting of univariate mean estimation. Their main result considers the setting where $p$ is a mixture of $k$ symmetric, log-concave distributions (which includes all examples mentioned so far); they design an adaptive estimator that, with $n$ samples, performs nearly as well as the best estimator that knows $p$ yet has $\frac{n}{k\polylog(n)}$ samples. They also prove that when $p$ is only symmetric and unimodal, it is impossible to attain the analogous result competing against an estimator with $\frac{n}{\polylog(n)}$ samples that knows $p$. Under the hood, both of these results leverage the \textit{Hellinger modulus of continuity} to characterize the optimal error (we discuss the Hellinger modulus in the preliminaries).

In our work, we design a computationally inefficient estimator that yields the analogous positive result for linear regression with symmetric log-concave mixture errors. Hence, the adaptive $L_q$ estimator runs in polynomial-time for $k=1$, yet we have no efficient estimator for $k \ge 2$. 

\paragraph{Planted linear regression.} For the purpose of studying the information-computation landscape of both heteroskedastic and adaptive linear regression, we suggest the following problem, which we call ``planted linear regression.''

\begin{definition}[Planted linear regression]\label{def:planted}
In \emph{planted linear regression}, the covariates $X_i \sim N(0,I_d)$, and the labels $Y_i = X_i^\top \beta + \eps_i$ for exchangeable $\eps_i$, exactly $m$ of which are zero and the remaining are drawn i.i.d. from $N(0,1)$. A similar i.i.d. version draws each $\eps_i$ from the mixture $\frac{m}{n} \delta_0 + \frac{n-m}{n} N(0,1)$. The goal is to recover $\beta$ exactly (or with much better error than the OLS rate $\sqrt{d/n}$). 
\end{definition}

Once $m \ge d+1$, this task is information-theoretically trivial: any hyperplane containing at least $d+1$ samples must be the true $\beta$ almost surely. However, it is not clear how to do this in polynomial time; a naive implementation would run in $n^{\Omega(d)}$ time. Most well-studied efficient estimators for linear regression are convex $M$-estimators, yet we later show that \textit{no convex $M$-estimator exactly recovers $\beta$} when $m \ll \sqrt{dn}$. In contrast, our estimator for heteroskedastic regression can recover $\beta$ once $m \gg d^{3/4}n^{1/4}$.
The success of our method in the regime $n^{1/4}d^{3/4}\ll m \ll \sqrt{dn}$ highlights the value of going beyond traditional estimators in finite-sample settings. We note that planted linear regression is a special case of \textit{noiseless linear regression} (see \cite{diakonikolas2026information}; we discuss more in \cref{subsec:related-work}).

In the regime $d + 1 \le m \ll d^{3/4}n^{1/4}$, recovering $\beta$ is information-theoretically possible, yet we know no polynomial-time algorithm; we conjecture that this is an information-computation gap, and we prove a statistical query (SQ) lower bound in support of this claim (\cref{thm:SQ}). 

Additionally, we remark that planted linear regression is a special case of heteroskedastic linear regression, and of adaptive linear regression (in the i.i.d. formulation). Our heteroskedastic regression estimator only yields polynomial-time Subset-of-Signals guarantees when $m \gg d^{3/4}n^{1/4}$; obtaining \textit{$\poly(n,d,1/\delta)$ Subset-of-Signals estimation error in $\poly(n,d)$-time} when $m \ll d^{3/4}n^{1/4}$ would imply an estimator for planted linear regression below the suggested computationally-hard regime. Likewise, a polynomial-time version of the adaptive linear regression guarantee (\cref{thm:main-adapt}) for $k=2$ would imply an estimator for planted linear regression when $m \gg d \polylog(n,d)$. This indicates an intertwined computational-statistical landscape for these three problems. We illustrate the planted linear regression landscape in \cref{fig:transition}.

\begin{figure}
\centering
\begin{subfigure}{.4\linewidth}
    \centering
    \includegraphics[width=\linewidth]{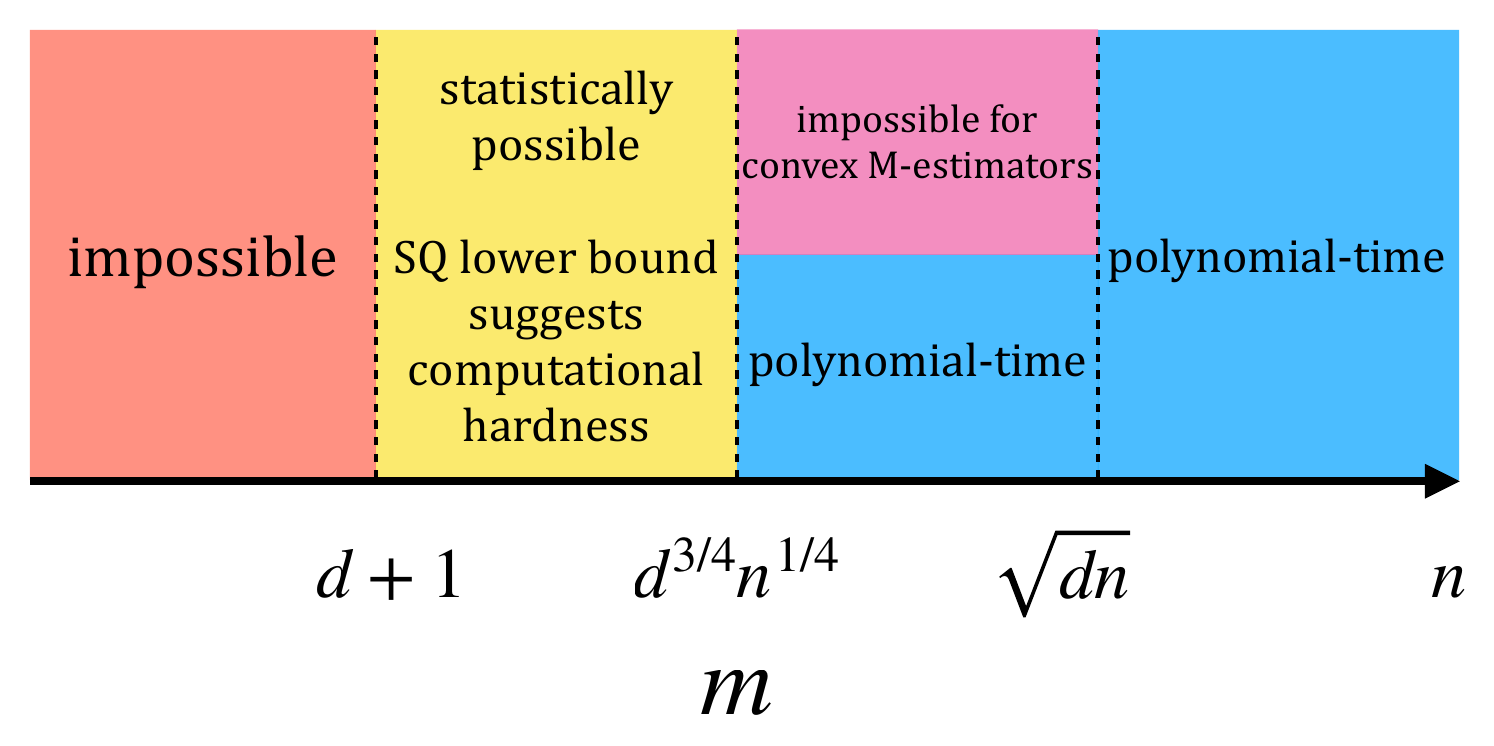}
    \caption{Our understanding of exact recovery for planted linear regression as a function of $m$.}\label{fig:transition}
\end{subfigure}
    \hfill
\begin{subfigure}{.58\linewidth}
    \centering
    \includegraphics[width=\linewidth]{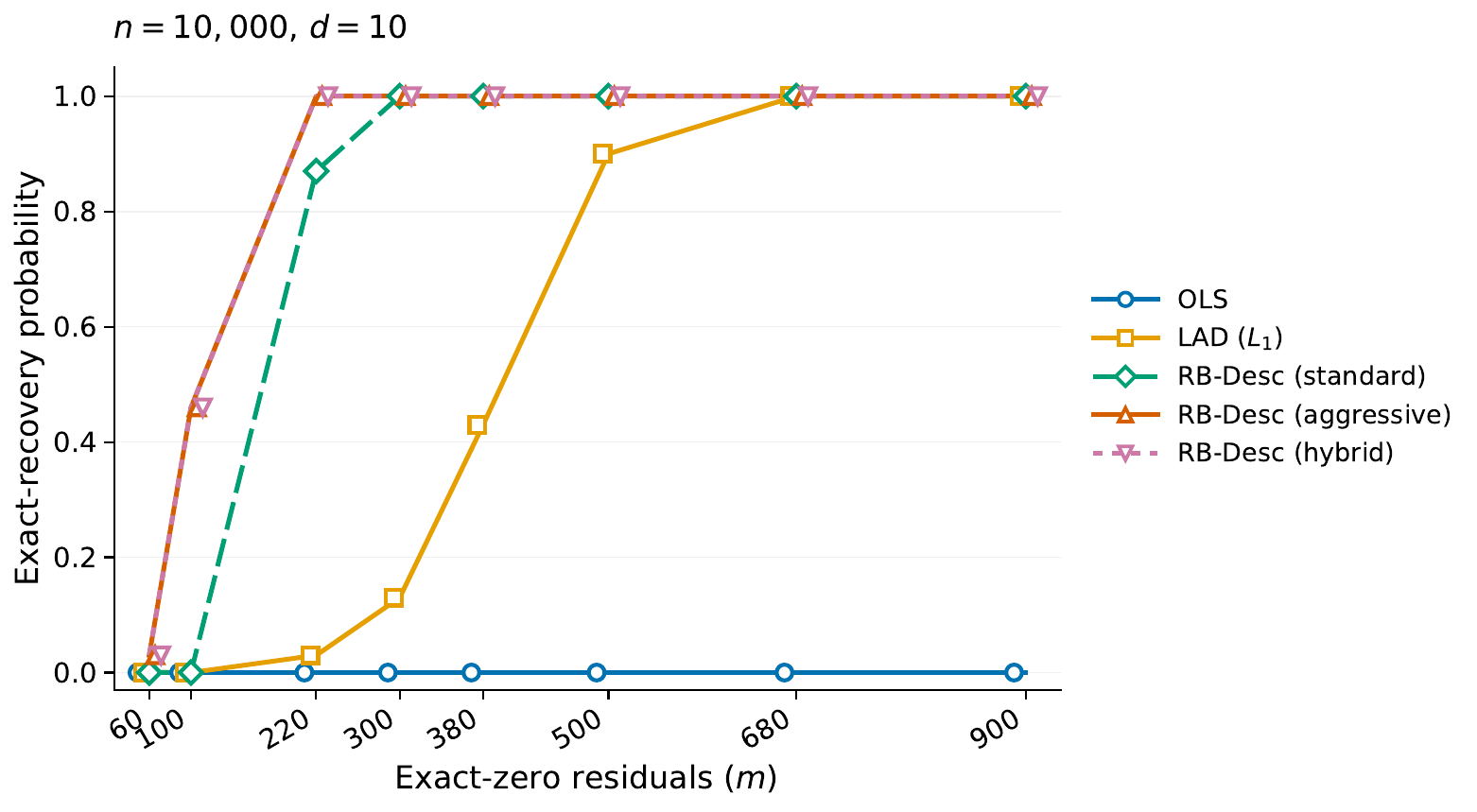}
    \caption{Empirical probability of recovery within $\norm{\wh\beta-\beta}_2 \le 10^{-5}$ (over 100 repetitions) for planted linear regression (see discussion in \cref{sec:simulations}).}\label{fig:planted-plot}
\end{subfigure}
{\caption{Planted linear regression.}}
\end{figure}

\subsection{Preliminaries}
For all linear regression tasks, the samples will take the form $(X_i,Y_i)\in \R^d \times \R$, where $Y_i = X_i^\top \beta + \eps_i$. The errors $\eps_i$ are independent from other samples, and independent of $X_i$. 
We will exclusively consider the setting where $X_i$ are i.i.d., and for constants $K,\kappa \ge 1$, we assume they have a $\kappa$ well-conditioned second-moment matrix 
\begin{equation}\label{eq:well-conditioned}
        \Sig=\E[XX^\top]
        \qquad\qquad \kappa^{-1}I\preceq \Sig\preceq \kappa I
\end{equation}
and are $K$ sub-Gaussian
\begin{equation}\label{eq:subg}
        \norm{\ip{u}{X}}_{\psi_2}
        \le K\sqrt{u^\top\Sig u}
        \qquad\text{for every }u\in\R^d,
\end{equation}
as would be the case when the variables have already been appropriately preprocessed/whitened. 
Constants throughout the proofs may depend on $K$ and $\kappa$ implicitly; we will be explicit in our theorem statements when the constants depend on $K,\kappa$. 

Throughout the paper we will use $c,C>0$ to denote constants, which may vary line-by-line and only depend on $K,\kappa$. We use $a \lesssim b$ to denote the same meaning as $a \le C b$ (meaning $a$ is at most a constant multiplied by $b$); and similarly for $\gtrsim$ notation. Let $a \asymp b$ denote both $a \lesssim b$ and $a \gtrsim b$. $O_{K,\kappa}(\cdot)$ notation suppresses multiplicative factors depending only on $K,\kappa$. We use $\log$ to denote the natural logarithm. We only study and discuss measurable sets.

For a candidate estimate $b\in\R^d$, we denote the $i$th residual as
\begin{equation}\label{eq:residual}
        r_i(b)=y_i-X_i^\top b.
\end{equation}

We often encounter symmetric, unimodal random variables (in 1-dimension), and use the following known decomposition (e.g. see \cite{jones2002khintchine}) of any such random variable:
\begin{definition}[Symmetric unimodal decomposition]\label{def:unimodal}
A real random variable $\eps$ is symmetric unimodal about zero if
\[
        \eps\stackrel{d}=TU,
\]
where $U\sim\Unif[-1,1]$, $T\ge0$, and $T$ is independent of $U$. 
\end{definition}
In this representation, $\Var(\eps)=\E T^2/3$; this means $\Var(\eps)\le1$ is
equivalent to $\E T^2\le3$.

For a 1-dimensional distribution $p$, let $p_\mu(x)$ denote the same distribution translated by $\mu$, meaning $p_\mu(x) = p(x-\mu)$.

\paragraph{Hellinger modulus of continuity.} Hellinger distance is central to our work. The squared Hellinger distance between two distributions $P,Q$ with densities $p,q$ is defined as
\[
\hels(P,Q) = \frac{1}{2} \int_\mathcal{X} \left(\sqrt{p(x)} - \sqrt{q(x)} \right)^2 \, dx.
\]
A crucial property of Hellinger distance is that it \textit{tensorizes}, meaning the distance between product distributions may be written in terms of the distance between the marginals:
\begin{fact}[e.g. page 45 of \cite{le2000asymptotics}]
Let $P^{\otimes n}$ and $Q^{\otimes n}$ denote the distribution of $n$ i.i.d. samples from $P$ and $Q$ respectively. Then,
\[
\hels(P^{\otimes n},Q^{\otimes n}) = 1 - (1 - \hels(P,Q))^n .
\]
\end{fact}
Squared Hellinger distance is also quadratically related to TV distance:
\begin{fact}[e.g. page 44 of \cite{le2000asymptotics}] $\hels(P,Q) \le \tv(P,Q) \le \sqrt{2 \hels(P,Q)}$.
\end{fact}
Using both facts, the TV distance  $\tv(P^{\otimes n},Q^{\otimes n})$ can be bounded in terms of the squared Hellinger distance between the marginals,
\[
1 - (1 - \hels(P,Q))^n \lesssim \tv(P^{\otimes n},Q^{\otimes n}) \lesssim \sqrt{1 - (1 - \hels(P,Q))^n}.
\]
Thus, the value of $n$ where $\tv(P^{\otimes n},Q^{\otimes n})$ goes from $\approx 0$ to $\approx 1$ is $n = \Theta(1/\hels(p,q))$. In this sense, the squared Hellinger distance between two distributions characterizes how many samples are required to distinguish them (while TV distance does not). This is helpful for cleanly establishing lower bounds; we will focus on lower bounds in the context of mean estimation. If $\hels(p,p_\mu) \le \frac{1}{n}$, then with only $n$ samples, any hypothesis tester must fail with some constant probability, and hence (by Le Cam's two-point testing method) must incur error $\mu/2$ with constant probability. 
Donoho and Liu \cite{donoho1991geometrizing1,donoho1991geometrizing2,donoho1991geometrizing3} established the influential concept of the \textit{Hellinger modulus of continuity}, which captures this style of argument. In the context of mean estimation, the Hellinger modulus is the function
\[
\omega_p(\eps) := \sup\{| \mu_1 - \mu_2 | : \hels(p_{\mu_1}, p_{\mu_2}) \le \eps, \, \mu_1,\mu_2 \in \R \} = \sup\{ t : H_p(t) \le \eps, \, t \ge 0\},
\]
where $H_p(t)$ is shorthand for $\hels(p,p_t)$.
By the same reasoning, any mean estimator for the class of translations of $p$ with access to only $n$ samples must incur error $c \cdot \omega_p(\frac{1}{n})$ with at least constant probability.

This perspective is widely useful both for providing lower bounds, and for obtaining upper bounds in terms of the Hellinger modulus (for just a few notable examples of the influence of this modulus perspective, see \cite{juditsky2009nonparametric,cai2015framework,foster2021statistical,duchi2024right,DualizingPolyanskiyWu,han2026empirical}). In the context of adaptive location estimation, the lower bound is immediate by definition, and the related work of \cite{compton2026attainability} studies conditions where the Hellinger modulus error is nearly achievable.
For mixtures of $k$ symmetric, log-concave distributions, they design an adaptive estimator with error on the order of $\omega_p(\frac{k \polylog(n)}{n})$; combined with the immediate Hellinger modulus lower bound, this implies that their adaptive estimator with $n$ samples performs nearly as well as the best custom estimator that knows $p$ and has $\frac{n}{k\polylog(n)}$ samples.
Our work will prove upper and lower bounds for adaptive \textit{linear regression} in terms of the Hellinger modulus for 1-dimensional \textit{mean estimation} (neither direction is immediate).
\subsection{Our contributions}\label{sec:contribs}
We now outline our results for heteroskedastic, adaptive, and planted linear regression. We additionally run experiments in \cref{sec:simulations}.

\paragraph{Polynomial-time heteroskedastic linear regression.} We design a polynomial-time algorithm that yields near-optimal guarantees for the Subset-of-Signals benchmark (we allow any symmetric, unimodal error, not just $N(0,\sigma_i^2)$):

\begin{theorem}\label{thm:main-het}
Suppose $\norm{\beta}_2 \le B$ and the distribution of $X$ is $\kappa$ well-conditioned and $K$ sub-Gaussian. There exists a constant $C_{K,\kappa} \ge 1$ depending only on $K,\kappa$ such that the following holds. Our residual-balance descent estimator ($\,\operatorname{RB-Desc}$) takes in parameters $n,d,B,\delta,K,\kappa$ where $n \ge d \ge 1$, $B,K,\kappa \ge 1$, and $\delta \in (0,1/2]$. Suppose 
\begin{equation}\label{eq:threshold}
        m\ge C_{K,\kappa} \log^5 \left( \frac{ndB}{\delta}\right) d^{3/4}n^{1/4}.
\end{equation}
Assume that each $\eps_i$ is distributed according to a symmetric unimodal distribution $p_i$, $m$ of which have variance at most $1$. Given access only to the samples (and not to $m$ nor the variances of the $p_i)$, with probability at least $1-\delta$, $\operatorname{RB-Desc}$ returns
$\wh\beta$ satisfying
\begin{equation}\label{eq:main-bound}
        \norm{\wh\beta-\beta}_2
        \le
        C_{K,\kappa} \log^5 \left( \frac{ndB}{\delta}\right) \cdot 
        \sqrt d\left(\frac{n}{m^4}\right)^{1/6},
\end{equation}
and runs in time
\[
        O_{K,\kappa}\left( n^2 d \log^2 \left(\frac{ndB}{\delta}\right)\right).
\]
\end{theorem}

\noindent We give a nearly matching minimax lower bound in \cref{sec:subset-lb}.

We briefly offer some intuition for the proof. 
We take inspiration from the balance-finding algorithm of Compton and Valiant \cite{compton2024near} for 1-dimensional heteroskedastic mean estimation. 
The main idea of their algorithm is that any sufficiently bad estimate $\wh\mu$ of the mean will exhibit ``imbalance'' at some scale $w>0$:
if $\wh\mu < \mu$ and $\norm{\wh\mu - \mu}_2$ is sufficiently large, this means that there exists some $w>0$ where the number of samples in $[\wh\mu-w,\wh\mu]$ is noticeably smaller than the number of samples in $[\wh\mu, \wh\mu + w]$. Accordingly, whenever some candidate estimate $\wh\mu$ has large imbalance to the left or right, we may restrict our search to that direction. 
As $w$ increases, the expected difference in sample counts becomes more favorable, yet the empirical fluctuations get worse; the optimal choice of $w$ entails a tradeoff that depends on the collection of variances. For example, if $m$ samples have $\sigma_i \le 1$ and the remaining $n-m$ samples have extremely large $\sigma_i$, then a ``bump'' in the histogram may be noticeable around $\mu$ at scale $w=1$; however, if the $n-m$ samples have moderate values of $\sigma_i$, then they may obscure the bump at scale $1$, yet the signal may be visible with a larger choice of $w$.
\cite{compton2024near} search over many choices of $w$ to obtain near-optimal Subset-of-Signals guarantees for mean estimation. Unfortunately, there is not an obvious analog for this approach in linear regression.

A well-known reduction from linear regression to mean estimation entails using that if $X$ is mean zero with identity covariance, then $\E[X_i y_i] = \beta$, and hence we may simply run a mean estimation procedure on $X_iy_i$ (e.g. see \cite[Section 7.2]{diakonikolas2023algorithmic}). This reduction does not work for our heteroskedastic setting (the distribution of $Xy$ is not even Gaussian), but we will use a similar intuition. When $y_i$ is very positive, it stands to reason that typically $X_i^\top \beta$ is very positive; similarly, when $y_i$ is very negative, typically $ X_i^\top \beta$ is very negative. Accordingly, it makes sense that $\E[X_i \sgn(y_i)]$ may correlate with $\beta$; this has been understood since at least the work of Manski \cite{manski1985semiparametric}.
In our algorithm, we use a related statistic with a scale parameter $w$,
\[
S_w(b) := \sum_{i=1}^n X_i \sgn(r_i(b)) \one\left\{\abs{r_i(b)} \le w\right\}.
\]
This is the (negative) subgradient of a clipped LAD objective $\sum_{i=1}^n \min\{ |r_i(b)|, w\}$, and is studied in the works of \cite{rousseeuw1982most,hampel,lia2023least}.
Our analysis focuses on proving that when some candidate $b$ is sufficiently far from $\beta$, then for some scale $w$, the statistic $S_w(b)$ will strongly correlate with $\beta - b$. This can be viewed as a form of imbalance in this direction. Our algorithm is a \textit{residual-balance descent}, where we repeatedly choose a set of scales, aggregate their statistics $S_w(b)$, and move in this direction.

We remark that in the special case of standard Gaussian covariates ($X_i \sim N(0,I_d)$), designing an algorithm is much simpler; we discuss this in \cref{sec:simple-gaussian}.

\paragraph{Adaptive linear regression for symmetric, log-concave errors.} We design a polynomial-time algorithm (with guarantees in terms of the Hellinger modulus) for adaptive linear regression when the errors are symmetric and log-concave:

\begin{theorem}\label{thm:k1}
Suppose the distribution of $X$ is $\kappa$ well-conditioned, $K$ sub-Gaussian, and the distribution of $\eps_i \stackrel{i.i.d.}\sim p$ is symmetric and log-concave. There exists a constant $C_{K,\kappa} \ge 1$ depending only on $K,\kappa$ such that, with probability at least $1-\delta$ (for $\delta \in (0,1/2]$), our adaptive $L_q$ estimator returns 
$\wh\beta$ satisfying
\begin{equation}\label{eq:k1-main-bound}
        \norm{\wh\beta - \beta }_2 \le C_{K,\kappa} \cdot \omega_p\left(C_{K,\kappa} \frac{d \log^4(64n/\delta)}{n} \right) + \gamma,
\end{equation}
for $\gamma \in (0,1]$, where the estimator runs in $\poly(n,d) \cdot \polylog(1/\delta, \norm{\mathbf{Y}}_\infty, 1/\gamma)$ time.
\end{theorem}

\noindent We  complement this result with a lower bound which shows it is impossible to obtain a rate better than $c \cdot \omega_p(cd/n)$.

Again, we offer some intuition for the proof. Our algorithm outputs an optimizer for $L_q$ regression
\[
\argmin_{b \in \R^d} \sum_{i=1}^n \abs{r_i(b) }^q,
\]
where the power $q$ is chosen in a data-dependent manner. Kao, Xu, and Zhang \cite{kao2024choosing} initiate a helpful study of this style of estimator, primarily in the context of adaptive location estimation (but also touching on adaptive linear regression). For mean estimation, their main analyses (Theorems 2.1 and 3.1) demonstrate guarantees for their adaptive estimator that relate to the performance of the $L_2$ and $L_\infty$ estimator. Though their guarantees are strong, they fail to cover all symmetric log-concave $p$. For example, both $q=2$ and $q = \infty$ give polynomially-suboptimal guarantees for the previously-discussed ``smoothed uniform errors'' example ($\operatorname{Unif}[-1,1] * N(0,n^{-2a})$ for $0 < a < 1$), which is symmetric and log-concave. With examples like this in mind, it is a priori unclear whether one should expect an adaptive $L_q$ estimator to work for all symmetric, log-concave distributions: is there always a choice of $q$ that recovers the desired guarantees?

We answer this question in the affirmative. As indicated by prior works (e.g. \cite{lai2005overview,el2023optimal,kao2024choosing}), the performance on $L_q$ regression is informed by the moments of the 1-dimensional error distribution $p$; we must prove the appropriate moment bounds. With work, this will follow by demonstrating a certain control over the tails of $p$. At first glance, it is not clear why $p$ must have tail control that relates in a meaningful way to its Hellinger modulus. The key tool is that for any $p$ and $p_\mu$ with large Hellinger distance, they may be distinguished by counting the number of samples inside a particular halfline (this is observed by \cite{compton2026attainability} and follows from the reverse data processing inequality of \cite{pensia2023communication}). In this sense, the distance between $p$ and $p_\mu$ is almost entirely from the tails, and this ultimately yields the necessary tail control (a more substantive proof intuition is given in \cref{sec:k1}).

Proving the existence of a good $q$ is the main difficulty; adaptively choosing the value of $q$ is much simpler. Roughly, we break the samples into batches, perform $L_q$ regression for all batches and a collection of choices for $q$, and then choose the value of $q$ for which the empirical $L_q$ estimates concentrate best (this style of technique has appeared before in e.g. \cite{nissim2007smooth,hsu2014heavy}).

\paragraph{Adaptive linear regression for mixtures of $k$ symmetric, log-concave errors.}
We design a (computationally intractable) estimator for adaptive linear regression when the errors are \textit{mixtures} of symmetric, log-concave distributions:
\begin{theorem}\label{thm:main-adapt}
Suppose the distribution of $X$ is $\kappa$ well-conditioned, $K$ sub-Gaussian, and the distribution of $\eps_i \stackrel{i.i.d.}\sim p$ is a mixture of $k$ symmetric, log-concave densities. There exists a constant $C_{K,\kappa} \ge 1$ depending only on $K,\kappa$ such that, with probability at least $1-\delta$, our (computationally intractable) estimator $\wh\beta$ satisfies
\begin{equation}\label{eq:adapt-main-bound}
        \norm{\wh\beta - \beta }_2 \le C_{K,\kappa} \cdot \omega_p \left( C_{K,\kappa} \frac{dk \log(2n/\delta) \log(2kn) \log(2n)} {n}\right).
\end{equation}
\end{theorem}

Compton and Valiant \cite{compton2026attainability} design an analogous polynomial-time estimator for the task of 1-dimensional adaptive location estimation when $p$ is likewise a mixture of $k$ symmetric, log-concave densities; our estimator will be a natural high-dimensional generalization of their technique.

In their adaptive location estimation algorithm, they output any estimate $\wh\mu$ such that the empirical samples look symmetric around $\wh\mu$ according to interval tests: meaning the number of samples in $[\wh\mu-b,\wh\mu-a]$ is approximately the same as in $[\wh\mu+a,\wh\mu+b]$, for all $0 \le a < b < \infty$. The success of this estimator is proven by a technical result where they show that whenever $p$ and $p_t$ have large Hellinger distance (and $p$ is a symmetric, log-concave mixture with small $k$), then there exists an interval $[l,r]$ where 
\[
\left( \sqrt{\Pp_{p}(x \in [l,r])} - \sqrt{\Pp_{p_t}(x \in [l,r])}\right)^2 
\]
is large. In other words, an interval $[l,r]$ ``witnesses'' the Hellinger distance between translations. The proof of this builds upon the reverse data processing inequality of Pensia, Jog, and Loh \cite{pensia2023communication}.

We design an estimator with an analysis that leverages this technical result. Suppose $b \in \R^d$ is a candidate estimate where $\norm{b - \beta}_2$ is large, and let $h = \beta - b$. The residuals satisfy
\[
r_i(b) = y_i - X_i^\top b = \eps_i + X_i^\top h.
\]
Suppose we defined a threshold $s > 0$, and two sets where $S_+$ contains all the residuals for samples where $X_i^\top h \ge s$, and $S_-$ contains all the residuals for samples where $X_i^\top h \le -s$. Observe how each sample in $S_+$, conditioned on $X_i$, is distributed like $\eps_i$ with a positive translation of at least $s$ (and similarly for $S_-$ with negative translations). If $b$ is a very poor estimate, we may expect that for the right choice of $s$, the means of $S_+$ and $S_-$ will be very different, and we want to detect when this occurs, so we may rule out this candidate estimate.

For any candidate estimate $b$, if there exists a choice of projection direction $u \in \R^d$, threshold $s>0$, and an interval $I=[l,r]$, where the number of samples from $S_+$ in $I$ is substantially different than the number from $S_-$, we will conclude $b$ is a poor estimate. By comparison, if $b = \beta$, then $X_i^\top u$ is independent of $r_i(\beta)$ for any $u$, so the sets $S_+,S_-$ are drawn from the same distribution, and $\beta$ should not fail any such test. The estimator is inefficient since we require a search over all $u \in \R^d$ to test a candidate estimate; a poor estimate $b$ will violate the test for $u=h$. Naturally, our analysis heavily leverages the technical interval witness result in \cite{compton2026attainability}. In reality, we use slightly different sets than $S_+,S_-$, but the intuition remains the same.

\paragraph{A complementary lower bound for linear regression in terms of the Hellinger modulus.} We prove a general lower bound for standard Gaussian covariates with symmetric, unimodal errors:
\begin{theorem}[Hellinger modulus lower bound for regression]\label{thm:main-hel-lb}
    There exists a universal constant $0 < c \le 1$ such that for every symmetric unimodal density $p$, every $n \ge d \ge 1$, and every estimator $\wh\beta$ given $n$ samples where $X_i \sim N(0,I_d)$, $\eps_i \sim p$, $Y_i = X_i^\top \beta + \eps_i$, it holds that
    \[
    \sup_{\beta \in \R^d} \Pp\left( \norm{\beta - \wh\beta}_2 \ge c \cdot \omega_p \left(c \frac dn \right) \right) \ge 3/8.
    \]
\end{theorem}

We aim to leverage a proof similar in spirit to Fano's method (see e.g. \cite[Section 31.4]{polyanskiy2025information} for related exposition), where we draw the parameter $\beta$ from a prior and show: (i) if the estimate $\wh \beta$ is often sufficiently close to $\beta$, then $I(\beta ; X^n, Y^n)$ must be large, and (ii) an upper bound for $I(\beta ; X^n, Y^n)$. Together, these imply a tradeoff that yields a lower bound. In a typical lower bound, the proof of (ii) can be implied by a suitable upper bound for $\E_{\beta, X_i}[\kll{p_{X_i^\top \beta}}{p}]$; this works well when $p \sim N(0,1)$, but can fail miserably when $p$ is an arbitrary symmetric unimodal distribution (e.g. when $p = \operatorname{Unif}[-1,1]$, the $\kl$ for translations are infinite). The main difficulty in our lower bound is sidestepping this technical obstacle. 

A key observation is that whenever two distributions $P$ and $Q$ have small Hellinger distance, then $P$ can be decomposed into two components $P = (1-\alpha)G + \alpha R$, where $\alpha \lesssim  \hels(P,Q)$ and $\kll{G}{Q} \lesssim \hels(P,Q)$. The decomposition follows intuitively from leveraging how Hellinger and KL are closely related when the densities are within a constant factor of each other, and that the mass where densities are not within a constant factor is bounded in terms of the Hellinger distance. While there are more technical steps, our proof aims to handle samples corresponding to $G$ in the more typical Fano style, and handle samples corresponding to $R$ separately.

\medskip

This gives a user-friendly lower bound for a more general class of errors than we are aware of in existing literature. For example, when the errors are uniform $\operatorname{Unif}[-1,1]$, the Hellinger modulus is $\omega_p(t) = 2t$ for $t \in [0,1)$, and hence any estimator must incur $\Omega(\frac{d}{n})$ error with constant probability. As far as we know, the optimal rate's dependence on $d$ for uniform errors was previously unknown, and in prior work it is suggested that $\sqrt{d}/n$ might be the optimal rate.\footnote{In particular, \cite[Theorem B.6]{yi2024non} claims a $\sqrt{d}/n$ upper bound which contradicts our result, but in a correspondence with the authors we have confirmed that there is a mistake in their proof and their estimator does not achieve the claimed upper bound.}
We expect there is a simpler lower bound which specifically targets the uniform setting, but our lower bound must handle all symmetric, unimodal distributions.

\paragraph{Adaptive estimators that nearly match the optimal non-adaptive estimator.} Let us denote the minimum possible worst-case expected error (risk) for standard Gaussian design, known noise $p$, and $m$ samples as
\[
\mathcal{R}(p,m) := \inf_{\text{estimator } \wh\beta} \, \, \sup_{\beta \in \R^d} \,\, \Ex_{\substack{X_i \sim N(0,I_d),\\ \eps_i \sim p,\\ Y_i = X_i^\top \beta + \eps_i}}\left[ \norm{ \wh\beta(X_1,\dots,X_m,Y_1,\dots,Y_m) - \beta}_2\right].
\]
When combined with our upper bounds, this lets us compare the performance of our adaptive estimator to the best estimator that knows $p$:
\begin{corollary}[Adaptive linear regression for symmetric, log-concave errors]\label{cor:k1}
There exists a constant $C_{K,\kappa} \ge 1$ such that when $n \ge C_{K,\kappa} d \log^4(64n/\delta)$ and the conditions of \cref{thm:k1} hold,  then with probability at least $1-\delta$, the adaptive $L_q$ estimator incurs error 
\[
\norm{\wh\beta - \beta}_2 \le \mathcal{R}\left(p,\left\lfloor\frac{n}{C_{K,\kappa} \log^4(64n/\delta)}\right\rfloor\right) + 3\gamma.
\]
\end{corollary}

\begin{corollary}[Adaptive linear regression for symmetric, log-concave mixture errors]\label{cor:adapt-mix}
There exists a constant $C_{K,\kappa} \ge 1$ such that when $n \ge C_{K,\kappa} dk \log(2n/\delta) \log(2kn) \log(2n)$ and the conditions of \cref{thm:main-adapt} hold, then with probability at least $1-\delta$, the computationally intractable adaptive estimator incurs error 
\[
\norm{\wh\beta - \beta}_2 \le \mathcal{R}\left(p,\left\lfloor\frac{n}{C_{K,\kappa} k \log(2n/\delta) \log(2kn) \log(2n)}\right\rfloor\right).
\]
\end{corollary}
Both follow from the lower bound of \cref{thm:main-hel-lb} and after a small simplification using \cref{claim:move-constant}.

\paragraph{Planted linear regression.} 
We suggest the planted linear regression problem as a simple, special case of both heteroskedastic and adaptive linear regression which may encapsulate the computational barriers to improving on our results.
We diagram our understanding in \cref{fig:transition}. 

In \cref{cor:planted-positive}, we show that our heteroskedastic linear regression algorithm recovers $\beta$ in the $m \gg d^{3/4} n^{1/4}$ regime. In contrast, \cref{thm:no-m-est} demonstrates that convex $M$-estimators fail to recover $\beta$ even when $m \ll \sqrt{dn}$. In \cref{thm:SQ}, we prove a statistical query (SQ) lower bound for the $m \ll d^{3/4}n^{1/4}$ regime; this provides supporting evidence for an information-computation gap in the regime $d+1 \ll m \ll d^{3/4}n^{1/4}$.

\subsection{Related work}\label{subsec:related-work}

\paragraph{Asymptotic adaptive estimation.} In the setting of adaptive \textit{location estimation}, the goal is to adaptively learn the location parameter of a distribution without knowing $p$. A substantial body of work (see \cite{stein1956efficient,van1970efficiency,stone1975adaptive,sacks1975asymptotically,beran1978efficient}) captures the setting where $p$ is fixed and symmetric, the Fisher information is finite, and $n \rightarrow \infty$; in this case, estimators may asymptotically attain the Fisher information rate. The work of Bickel \cite{bickel1982adaptive} includes a study into the analogous setting for adaptive linear regression. The work of Laha \cite{laha2021adaptive} studies adaptive location estimation for symmetric, log-concave $p$, and proposes an estimator that requires no tuning parameter. 
Even when the Fisher information $\mathcal{I}$ is infinite, it is still possible to achieve  $n^{1/2} (\wh \theta_n - \theta) \rightarrow N(0,1/\mathcal{I})$ as $n \rightarrow \infty$, for symmetric yet unknown $p$; see the work of Stone \cite{stone1975adaptive} on adaptive location estimation and of Koul and Susarla \cite{koul1983adaptive} on adaptive linear regression. We note that this still requires $p$ to be fixed as $n \rightarrow \infty$, and when $\mathcal{I} = \infty$ this implies convergence faster than $n^{-1/2}$, but not necessarily the optimal rate (e.g. $p = \operatorname{Unif}[-1,1]$ permits $n^{-1}$ convergence rate).
For location estimation where $p$ is known, the work of Le Cam \cite{lecam1973convergence} implies Hellinger modulus guarantees in the asymptotic setting under particular metric dimensionality assumptions. To our knowledge, we do not know of any result in the asymptotic setting that always yields the optimal convergence rate when $p$ may have infinite Fisher information, and $p$ is symmetric yet unknown (for context, \cite[Theorem 1.5]{compton2026attainability} implies that the analogous style of result is impossible in the finite-sample setting, even when assuming $p$ is symmetric and unimodal, but it does not rule out the asymptotic setting).

Feng, Kao, Xu, and Samworth \cite{feng2026optimal} study adaptive linear regression with convex $M$-estimators (this subsumes $L_q$ losses). Their \textit{antitonic score matching (ASM) estimator} chooses the convex loss in a data-dependent manner, where the loss is chosen to minimize the asymptotic variance among all convex losses.  They focus on the setting where $p$ is unknown yet has finite Fisher information. The estimator's performance in experiments is very strong, and we include their estimator as a benchmark in our simulations.

\paragraph{Finite-sample adaptive estimation.} The most directly related work is that of Compton and Valiant \cite{compton2026attainability}, which studies location estimation. When $p$ is known, they attain nearly-optimal Hellinger modulus guarantees if $p$ is unimodal, yet show this is impossible when only assuming $p$ is symmetric. In the adaptive setting where $p$ is unknown, they attain nearly-optimal Hellinger modulus guarantees if $p$ is a mixture of $k = O(1)$ symmetric, log-concave distributions, yet show this is impossible when only assuming $p$ is symmetric and unimodal.

The work of Kao, Xu, and Zhang \cite{kao2024choosing} primarily aims to better understand adaptive location estimation and regression when $p$ is symmetric, unknown, and may have infinite Fisher information. Focusing mostly on adaptive location estimation, they study an adaptive $L_q$ estimator, where $q$ is chosen in a data-dependent manner; this yields nearly optimal rates for examples such as when $p$ is Gaussian or Uniform. Our work will leverage this style of estimator for adaptive linear regression with symmetric, log-concave errors. As noted by \cite{kao2024choosing}, this estimator is unable to leverage discontinuities in the interior of $p$; as an example, the mixture $p = \frac{1}{2} N(0,1) + \frac{1}{2} \operatorname{Unif}[-1,1]$ has optimal rate on the order of $n^{-1}$ for location estimation, but an $L_q$ estimator would only attain an $n^{-1/2}$ rate. 

Gupta, Lee, Price, and Valiant study location estimation when $p$ is known \cite{gupta2022finite}, and when $p$ is symmetric yet unknown \cite{gupta2023finite}. They obtain guarantees in terms of the \textit{smoothed Fisher information} of the $p$; when $p$ has variance $\sigma^2$, this corresponds to the Fisher information of $p$ after it is convolved with $N(0,\sigma^2 r^2)$, where $r$ is roughly $n^{-c}$ for a small constant $c>0$. For some distributions (e.g. Cauchy, Gaussian), this yields extremely strong guarantees with nearly-sharp constant factors. For other distributions (e.g. Uniform, heteroskedastic mixtures), this convolution destroys most of the important signal. The same authors prove that when $p$ is known, a variant of the MLE is optimal for location estimation \cite{gupta2023minimax}.
A related line of work \cite{catoni2012challenging,lee2022optimal,lee2022optimal2,gupta2024beyond} studies estimators with sharp constant factor dependence for the sub-Gaussian error rate (where $p$ is unknown and has finite variance).

The method of $\rho$-estimation (see \cite{baraud2017new,baraud2016rho,baraud2018rho}) is also a very relevant technique; at a minimum, this technique would yield certain finite-sample guarantees in terms of the Hellinger modulus when $p$ is known.

\paragraph{Heteroskedastic estimation.} Chierichetti, Dasgupta, Kumar, and Lattanzi \cite{chierichetti2014learning} initiated a study of heteroskedastic mean estimation: in one dimension, consider samples $X_i \sim N(\mu,\sigma_i^2)$, where the goal is to estimate $\mu$ without knowing the variances. They studied the performance of an estimator that leveraged an empirical median confidence interval, and the $k$-shorth estimator (which finds the smallest interval containing at least $k$ samples). This motivated an interesting body of work studying the performance of different estimators for heteroskedastic mean estimation. Pensia, Jog, and Loh \cite{pensia2022estimating} designed hybrid median/modal/shorth estimators (in higher-dimensional settings as well). Devroye, Lattanzi, Lugosi, and Zhivotovskiy \cite{devroye2023mean} studied error guarantees for a different hybrid median/modal estimator. Xia \cite{xia2019non} and Louati \cite{louati2025performance} analyzed the performance of the empirical median. 
Liang and Yuan \cite{yuan2020learning} studied an iterative trimming estimator, and also introduced the Subset-of-Signals model, for which they proved lower bounds \cite{liang2020learning}.
Compton and Valiant \cite{compton2024near} designed a balance-finding estimator that matched these lower bounds, attaining near-optimal Subset-of-Signals guarantees (in near-linear time). Han, Shetty, and Shkrob \cite{han2026empirical} also attained near-optimal Subset-of-Signals guarantees, and further proved strong Hellinger modulus guarantees, with an estimator that leverages empirical Bayes methods.

There has also been interest in higher-dimensional heteroskedastic mean estimation. The work of \cite{chierichetti2014learning} studies the task where $X_i \sim N(\mu, \sigma_i^2 I_d)$, inspired by a setting where users of unknown, varying skills rate $d$ items each (higher-skill users have smaller $\sigma_i$). In general, heteroskedastic mean estimation is much easier when the $\sigma_i$ for each sample is known by the estimator; the optimal rate is much better, as an estimator may use a weighted mean with $X_i$ weighted proportionally to $1/\sigma_i^2$ \cite{ibragimov2013statistical}. In this high-dimensional setting, \cite{chierichetti2014learning} show that the known-variance rate is nearly attainable even without knowing the variances, so long as $d = \Omega(\log n)$. Later, the work of \cite{compton2024near} demonstrated this is possible for any $d \ge 2$.
Li \cite{li2024entangled} proposed an alternative high-dimensional mean estimation task where $X_i \sim N(\mu, \Sigma_i)$ for PSD $\Sigma_i$, and $m$ of the samples satisfy $\Sigma_i \preceq I_d$. Diakonikolas, Kane, Liu, and Pittas \cite{diakonikolas2025entangled} designed a near-optimal estimator in this high-dimensional Subset-of-Signals setting; their estimator employs an iterative dimension-reduction strategy, and uses the 1-dimensional estimator of \cite{compton2024near} as a subroutine.

The works of \cite{liang2020learning,pensia2022estimating,chaudhary2026linear} have also studied the task of heteroskedastic linear regression, and the survey of \cite[Section 3.4.3]{maleki2026high} highlighted the open problem of whether optimal guarantees are achievable in polynomial time. Even information-theoretically, the optimal guarantees were not yet established; for example, all previous Subset-of-Signals guarantees attaining $o(1)$ error required $m \ge \Omega(\sqrt{nd})$ (this is achieved by the Huber loss estimator via the analysis of \cite{d2021consistent}).

We note that heteroscedastic/heteroskedastic linear regression may also refer to a task where the error scales with an unknown direction (scaling with $X_i^\top f$ for an $f \in \R^d$); this is a very different task from what we study, and we refer to \cite{das2023near} for discussion of this problem.

\paragraph{Oblivious contamination and noiseless linear regression.}
In the setting of linear regression with \textit{oblivious contamination}, an adversary adds noise $\eps_1,\dots,\eps_n$, but the adversary does not observe $X_1,\dots,X_n$ (and hence the noise is independent, or ``oblivious,'' to the covariates). Framed in terms of our notation, suppose $m$ of the noise entries satisfy $|\eps_i| \le 1$. A main question in oblivious contamination is understanding what values of $m$ permit $o(1)$ coefficient error. A line of research including \cite{tsakonas2014convergence,bhatia2017consistent,suggala2019adaptive}, and culminating in the work of d’Orsi, Novikov, and Steurer \cite{d2021consistent}, proved that $o(1)$ error is possible once $m \gg \sqrt{nd}$ (with only mild assumptions on the design matrix); as noted in \cite{d2021consistent}, a near-matching statistical lower bound follows immediately from the setting where all $\eps_i \sim N(0,n/d)$. Note that the Subset-of-Signals model for heteroskedastic regression is (roughly) a special case of oblivious contamination: where $m$ samples have oblivious error $N(0,1)$, and $n-m$ samples have oblivious error $N(0,\sigma_i^2)$ for any $\sigma_i$. While $m \gg \sqrt{nd}$ is optimal for oblivious contamination, this contrasts with the lower optimal threshold of  $m \gg  d^{3/4}n^{1/4}$ for heteroskedastic regression (ignoring logarithmic factors).

The main predecessor of planted linear regression is the task of \textit{noiseless linear regression}. In this task, the covariates are also $N(0,I_d)$, and $m$ samples are noiseless, yet the remaining $n-m$ samples have error drawn from any fixed distribution $E$ (a related i.i.d. version has all noise drawn from a distribution $E$, where $\Pp_{Z \sim E}(Z = 0) \ge \frac{m}{n}$). Diakonikolas, Gao, Kane, Lafferty, and Pensia 
\cite{diakonikolas2026information} study a possible information-computation gap for this problem. When $m \ge d+1$, this problem is information-theoretically trivial for the same reasons as planted linear regression. When $m \gg \sqrt{nd}$ (ignoring logarithmic factors), the work of \cite[Theorem 3.2]{gao2020model} demonstrates that LAD attains exact recovery of $\beta$. The work of \cite{diakonikolas2026information} conjectures that any polynomial-time algorithm requires $m = \Omega(\sqrt{nd})$ to exactly recover $\beta$, and they provide Statistical Query (SQ) lower bounds to support their claim (the SQ lower bounds are not yet tight to their conjecture, this is left as an open problem). Note that noiseless linear regression is a special case of oblivious contamination, and further, planted linear regression is a special case of noiseless linear regression. While $m \gg  \sqrt{nd}$ is the conjectured computational threshold for noiseless linear regression (and is attained by LAD, a convex $M$-estimator), our task of planted linear regression permits polynomial-time recovery when $m \gg d^{3/4}n^{1/4}$ (ignoring logarithmic factors), even though we show that convex $M$-estimators cannot recover $\beta$ when $m \ll \sqrt{dn}$.

We refer to \cite{diakonikolas2026robust} for discussion of works where the labels may be adversarially corrupted.

\paragraph{Information-computation gaps.} When an estimation problem is statistically tractable, but no time-efficient algorithm is known, the problem is said to exhibit an ``information-computation gap.'' 
Any such gap could be a product of researchers' ignorance and/or laziness.
Hence, when confronted with an information-computation gap, it is desirable to prove a lower bound that suggests that no state-of-the-art algorithmic technique could overcome the gap.
There are two common approaches: the first is to give a reduction from a statistical problem that is widely believed to be hard, for example planted clique (e.g. \cite{bresler2018sparse,brennan2020reducibility,buhai2024computational}).
The second approach is to prove a lower bound against a restricted class of algorithms or computational model; popular frameworks include statistical query (SQ) algorithms (e.g. \cite{FGRVX,diakonikolas2026information}), low-degree algorithms (e.g. \cite{hopkins2018statistical,wein2025computational}), and generalized first-order methods (e.g. \cite{celentano2020estimation}).
Some problems are more natural to consider in one particular framework (though the frameworks are directly comparable for many canonical problems \cite{pmlr-v134-brennan21a,montanari2025equivalence}). 
Problems in which the signal-to-noise ratio is measured using the number of samples, and in which no one sample carries too much information, are often most natural to consider in the SQ framework. 
In this work, we encounter an information-computation gap for \emph{planted linear regression}, and substantiate it with an SQ lower bound: 
we show that when our algorithm fails, it is also the case that no query-efficient SQ algorithm can achieve error which is significantly better than the OLS estimator.
We note that because it is essential that the planted linear regression problem is \emph{noiseless} on a subset of labels, the low-degree framework does not make sense for this problem (and indeed, it may be checked that the low-degree method obtains a pessimistic threshold).

\paragraph{Reverse data processing inequalities.}
The results of Bhatt, Nazer, Ordentlich, and Polyanskiy \cite[Theorem 1]{bhatt2021information} and Pensia, Jog, and Loh \cite[Corollary 3.4]{pensia2023communication} both imply a ``reverse data processing inequality'' for Hellinger distance: for any $P,Q$, there exists a $\tau > 0$ where the function $f_\tau(x) = \one\{P(x)/Q(x) \ge \tau \}$ is such that $\hels(f_\tau(P),f_\tau(Q))$ is at worst a logarithmic factor smaller than $\hels(P,Q)$. This result is naturally useful for communication-constrained settings, but it has recently been helpful in settings with no communication constraints as well (see \cite{compton2026attainability,blanc2025instance,compton26density}). We will often use the form given by \cite[Corollary 3.4]{pensia2023communication} in this work. The work of \cite{kazemi25a} more generally proves reverse data processing inequalities for ``TV-like $f$-divergences.''

\subsection{Discussion}
In this work, we have studied the computational-statistical landscape of heteroskedastic, adaptive, and planted linear regression. 
For heteroskedastic linear regression, we design a polynomial-time algorithm with near-optimal Subset-of-Signals guarantees when $m \gg d^{3/4}n^{1/4}$ (\cref{thm:main-het}). 
For adaptive linear regression, we design estimators for mixtures of $k$ symmetric, log-concave $p$; our estimators perform nearly as well as the optimal estimator that knows $p$ and has $\tilde{\Theta}(n/k)$ samples. Our estimator for $k=1$ uses adaptive $L_q$ regression and runs in polynomial time (\cref{thm:k1}); for $k \ge 2$, we design an inefficient estimator (\cref{thm:main-adapt}). Finally, we introduce the task of planted linear regression, for which we demonstrate a polynomial-time estimator and a near-matching SQ lower bound at $m \sim d^{3/4}n^{1/4}$; any computationally tractable estimator improving on our heteroskedastic or adaptive results would imply an estimator for planted linear regression that succeeds below this threshold.

More broadly, our work highlights the helpfulness of the finite-sample perspective for these problems, and how this lens inspires new estimators with guarantees that might not be possible with classical techniques.

\medskip

We make some suggestions for follow-up work.
First, a few technical (and perhaps minor) points for improvement:
For heteroskedastic linear regression, it would be desirable to remove the dependence on $B$, the radius of the ball containing $\beta$ in \cref{thm:main-het}. We also omit a Subset-of-Signals information-theoretic upper bound for the $m \ll d^{3/4}n^{1/4}$ regime; we expect an analysis in the style of \cref{thm:main-adapt} would yield the correct result, or possibly an analysis in the style of the metric entropy techniques in \cite{han2026empirical}. For adaptive linear regression (\cref{thm:main-hel-lb}), our Hellinger modulus lower bound is stated in terms of standard Gaussian covariates, instead of arbitrary well-conditioned and sub-Gaussian covariates. For planted linear regression, our SQ lower bound (\cref{thm:SQ}) adds arbitrarily small noise (e.g. superexponentially small in $n$) to the $m$ noiseless samples for technical convenience.

We next outline three broader potential directions for future study:

\paragraph{Practical implementations.}
We believe our simulations in \cref{sec:simulations} are generally positive, yet there is still room for improvement. First, our $\operatorname{RB-Desc}$ estimator needed important modifications for strong empirical performance (we discuss this in \cref{sec:simulations}); it would be nice if there is a simpler estimator with the same theoretical guarantees that attains strong empirical performance more naturally. Second, the estimators have favorable empirical performance, yet no single estimator attains the best performance simultaneously for all five test distributions in \cref{fig:adaptive-plots-2}. 
Designing a natural estimator that simultaneously achieves the best performance for all these settings seems interesting. On a related note, we were unable to produce a reasonable implementation for our inefficient estimator in \cref{thm:main-adapt}, even for very small $d$; a core obstacle is that the estimator requires a search over an extremely fine net. A practical version of this estimator for small $d$ (yet still moderately large $n$) would be quite interesting, as the analogous estimator for mean estimation simultaneously nearly achieves the best empirical performance for all examples in \cite{compton2026attainability}.

\paragraph{Alternative estimators. } 
We are optimistic there may be alternative estimators that also achieve the theoretical guarantees we study. 
For example, in the case of 1-dimensional heteroskedastic mean estimation, the known estimators that achieve error $\ll 1$ when $m\gg n^{1/4}$ are the balance-finding algorithm of \cite{compton2024near} and the empirical Bayes estimator of \cite{han2026empirical}. We remark that a sharper analysis of a hybrid median/modal estimator (in the style of \cite{devroye2023mean}\footnote{Devroye, Lattanzi, Lugosi, and Zhivotovskiy \cite{devroye2023mean} study a hybrid median/modal estimator. In their discussion of the Subset-of-Signals model, they obtain $\ll 1$ error when $m \gg n^{1/2}$. We observe that a different invocation of their key technical ingredients would recover a sharper guarantee of $\lesssim  1$ error when $m \gg n^{1/4}$. For intuition, consider their ``two variances'' example in their Section 5.1 with $\sigma = 1$: if $\sigma' \lesssim  \sqrt{n}$ then the median has error $\lesssim 1$, otherwise if $\sigma' \gg \sqrt{n}$, their modal term incurs error $\lesssim 1$ when $m \gg n^{1/4}$. This does not necessarily imply near-optimal Subset-of-Signals guarantees for the entire $m \gg n^{1/4}$ regime.}) can actually recover this guarantee as well; this may be helpful knowledge when designing future estimators.

\paragraph{Wider applicability of the (computationally inefficient) adaptive linear regression estimator.} Our result in \cref{thm:main-adapt} focuses on the setting where $p$ is a mixture of $k$ log-concave distributions, each symmetric around 0. This estimator applies to more general $p$, and better understanding which classes of distributions could be interesting. This would entail proving that the Hellinger distance between translations is witnessed by intervals for a larger class of distributions (generalizing Corollary 2.18 of \cite{compton2026attainability}). 
As an example, the Cauchy distribution is not contained within our \cref{thm:main-adapt}, but it is simple to prove that its translations have Hellinger distance witnessed by intervals; in this case, $p(x) \ge \tau p_t(x)$ in precisely the region where a degree-2 polynomial in $x$ is non-negative, and hence the region is at most two intervals. 
The desired adaptivity guarantees should follow for any symmetric density that can be $\eps$-approximated in Hellinger distance by a piecewise polynomial with $\polylog(n/\eps)$ pieces and degree (related ideas are studied for Hellinger density estimation in \cite{compton26density}, which builds upon \cite{chan2013learning,chan2014efficient,acharya2015fast,acharya2017sample}).
See also the discussion in \cite{compton2026attainability} on ``adaptive location estimation for more general distributions.''

\input{het}

\input{adaptive}
\include{planted}
\include{simulations}

\section*{Acknowledgements}
We thank Frederic Koehler and Gregory Valiant for discussions in earlier stages of this project. We thank Adityanand Guntuboyina and Richard Samworth for bringing related work to our attention. ChatGPT was used during the process of developing proofs and implementing simulations; we take full responsibility for the content of our paper. T.S. and S.C. are supported by T.S.'s NSF CAREER
Grant no. 2143246 and the NSF AI Institute for Foundations of Machine Learning (IFML). S.C. is also supported by a Two Sigma PhD Fellowship, and by Gregory Valiant’s Simons Foundation Investigator Award and NSF award
AF-2341890.

\bibliography{ref}
\appendix
\include{simple-gaussian}
\include{subset-signals-lb}
\input{additional-sims}
\input{unif-conv}
\input{misc-app}
\end{document}

%% file: macros.tex
\documentclass[11pt,letterpaper]{article}%
\usepackage[utf8]{inputenc}
\usepackage[rflt]{floatflt}
\usepackage{fullpage}
\usepackage[T1]{fontenc}
\usepackage{graphicx}
\usepackage{epsfig}
\usepackage{color}
\usepackage[shortlabels]{enumitem}
\usepackage{verbatim}
\usepackage{amsthm,amssymb}
\usepackage{amsmath}
\usepackage{aliascnt}
\usepackage{wrapfig}
\usepackage{xspace}
\usepackage{amsfonts}
\usepackage{bbm}
\usepackage{algorithm, algpseudocode, caption, subcaption}

\usepackage{authblk}

\usepackage{todonotes}

\definecolor{darkgreen}{rgb}{0,0.5,0}
\usepackage{hyperref}
\hypersetup{
    unicode=false,          %
    colorlinks=true,        %
    linkcolor=blue,          %
    citecolor=darkgreen,        %
    filecolor=magenta,      %
    urlcolor=cyan           %
}

\usepackage[disableredefinitions,roman]{complexity}
\usepackage[framemethod=tikz]{mdframed}
\global\mdfdefinestyle{myframe}{leftmargin=.75in,rightmargin=.75in,linecolor=black,linewidth=1.5pt,innertopmargin=10pt,innerbottommargin=10pt} 
\usepackage{mdframed}
\usepackage{framed}
\usepackage{nicefrac}
\usepackage[capitalize, nameinlink]{cleveref}

\date{}

\usepackage[normalem]{ulem}

\newtheorem{theorem}{Theorem}[section]

\newcommand{\newsharedtheorem}[3]{%
    \newaliascnt{#1}{theorem}%
    \newtheorem{#1}[#1]{#2}%
    \aliascntresetthe{#1}%
    \crefname{#1}{#2}{#3}%
    \Crefname{#1}{#2}{#3}%
}

\newsharedtheorem{lemma}{Lemma}{Lemmas}
\newsharedtheorem{proposition}{Proposition}{Propositions}
\newsharedtheorem{claim}{Claim}{Claims}
\newsharedtheorem{infclaim}{Informal Claim}{Informal Claims}
\newsharedtheorem{observation}{Observation}{Observations}
\newsharedtheorem{corollary}{Corollary}{Corollaries}
\newsharedtheorem{fact}{Fact}{Facts}

\theoremstyle{definition}
\newsharedtheorem{definition}{Definition}{Definitions}
\newsharedtheorem{remark}{Remark}{Remarks}

\def\polylog{\operatorname{polylog}}

\DeclareMathOperator{\sgn}{sgn}

\newcommand{\eps}{\varepsilon}

\newcommand\numberthis{\addtocounter{equation}{1}\tag{\theequation}}
\newcommand{\rtag}[1]{\text{\quad $\left(\text{#1}\right)$}}

\renewcommand{\Pr}{\operatorname{{Pr}}}

\DeclareMathOperator*{\argmin}{arg\,min}
\DeclareMathOperator*{\Ex}{\mathbb{E}}

\DeclareMathOperator{\Unif}{Unif}
\DeclareMathOperator{\clip}{clip}

\renewcommand{\R}{\mathbb R}
\newcommand{\Pp}{\mathbb P}
\newcommand{\one}{\mathbf 1}
\newcommand{\wh}{\widehat}
\newcommand{\wt}{\widetilde}
\newcommand{\norm}[1]{\left\lVert #1\right\rVert}
\newcommand{\abs}[1]{\left\lvert #1\right\rvert}
\newcommand{\ip}[2]{\left\langle #1,#2\right\rangle}
\newcommand{\Sig}{\Sigma}
\newcommand{\calW}{\mathcal W}
\newcommand{\calV}{\mathcal V}
\newcommand{\calA}{\mathcal A}
\newcommand{\calC}{\mathcal C}
\newcommand{\calP}{\mathcal P}
\newcommand{\calN}{\mathcal N}
\newcommand{\kl}{\operatorname{KL}}
\newcommand{\kll}[2]{\kl ({#1}, {#2})}
\newcommand{\hel}{\operatorname{H}}
\newcommand{\hels}{\hel^2}
\newcommand{\tv}{\operatorname{TV}}

\allowdisplaybreaks[1]

\newcommand{\calD}{\mathcal{D}}
\newcommand{\Id}{\mathrm{I}}
\DeclareMathOperator{\Var}{Var}
\newcommand{\VSTAT}{\mathrm{VSTAT}}
\newcommand{\OLS}{\mathrm{OLS}}
\newcommand{\Ind}{\mathbf{1}}

\usepackage{titlesec}

\titleformat{\subparagraph}[runin]
  {\normalfont\normalsize\itshape}
  {\thesubparagraph}{1em}{}

\titlespacing*{\subparagraph}
  {0pt}             %
  {\medskipamount}  %
  {0.5em}  %

%% file: het.tex
\section{Heteroskedastic linear regression}\label{sub:general-het}
In this section, we discuss a polynomial-time algorithm for heteroskedastic linear regression with well-conditioned, sub-Gaussian covariates.
We remark that this task is substantially simpler with standard Gaussian covariates (see \cref{sec:simple-gaussian} for further discussion).

As is outlined in \cref{sec:contribs}, the main intuition for our estimator is that we aim to show an estimate $b \in \R^d$ is either close to $\beta$, or it exhibits a notion of imbalance in the direction of $\beta - b$ that we may use to iteratively improve the estimate. We start by introducing the relevant notation.
Let
\begin{equation*}
        L=c_{L}\log\!\left(\frac{en dB}{\delta}\right).
\end{equation*}
It is helpful to think of $c_{L}$ as the last constant chosen in the proof (throughout the proof we will often say ``for sufficiently large choice of constant in $L$'' to denote that $c_{L}$ must be chosen sufficiently large); this is because $c_L$ is chosen so that it is large enough to handle terms involving other constants (this is never done in the reverse direction, or the reasoning would be circular). Note that $c_L$ still only depends on $K,\kappa$ (meaning it is independent of $n,d,m,B,\delta$, etc).

For a window with scale $w>0$, we define
\begin{equation*}
        \psi_w(t)=\sgn(t)\one\{\abs{t}\le w\},
        \qquad \sgn(0)=0.
\end{equation*}
The residual-balance vector at scale $w$ is then
\begin{equation*}
        S_w(b)=\sum_{i=1}^n X_i\psi_w(r_i(b)).
\end{equation*}
Roughly, $S_w(b)$ points in the direction of a signed imbalance, and we will show that $S_w(b)$ is correlated with the direction $\beta - b$ when they are sufficiently far apart. We further note that $S_w(b)$ is the (negative) subgradient of a clipped LAD objective $\sum_{i=1}^n \min\{ |r_i(b)|, w\}$, and is studied in the works of \cite{rousseeuw1982most,hampel,lia2023least}.

The local residual count is
\begin{equation*}
        A_w(b)=\sum_{i=1}^n\one\{\abs{r_i(b)}\le w\},
\end{equation*}
which quantifies the number of samples that are active at scale $w$, and will help bound the amount of noise in $S_w(b)$.
Let $w_0=w_0(K,\kappa)$ be a constant defined later in
\cref{lem:scalar-drift}.  Define the dyadic grid (with $\infty$)
\begin{equation}\label{eq:grid}
        \calW=\{w_0,2w_0,4w_0,\ldots,2^{J}w_0\} \cup \{\infty\},
\end{equation}
where $J$ is the smallest nonnegative integer satisfying $2^J\ge\max\{1,2\kappa^{1/2}B\}$; this will capture all the potentially relevant scales. The grid construction satisfies
\begin{equation}\label{eq:Nsc-L}
        |\calW|\le L,
\end{equation}
since $|\calW |\lesssim1+\log(B)$ is less than {$L$ when it is defined with a sufficiently large constant.}
For $w\in\calW $, we use $A_w(b)$ to define a threshold for the uncertainty in the residual-balance vector
\begin{equation}\label{eq:tau-def}
        \tau_w(b)=L^2
        \left(\sqrt{d(A_w(b)+d)}+d\right).
\end{equation}
A scale $w$ is called \emph{violated} at $b$ if
\begin{equation*}
        \norm{S_w(b)}_2> | \calW|  \cdot \tau_w(b).
\end{equation*}
Consider the collection of violated scales
\[
        \calV(b)=\{w\in\calW:w\text{ is violated at }b\}.
\]
The aggregate residual-balance direction is then
\begin{equation*}
        G_{\rm agg}(b)=\sum_{w\in\calV(b)}\frac{S_w(b)}{\tau_w(b)},
\end{equation*}
with the convention $G_{\rm agg}(b)=0$ if $\calV(b)=\varnothing$. Intuitively, this roughly corresponds to performing gradient descent with a collection of clipped LAD objectives, where the gradients are rescaled according to their corresponding uncertainty threshold $\tau_w(b)$.

With these definitions in hand, the estimator is simple: we repeatedly normalize $G_{\rm agg}(b)$ and move in that direction. We next prove guarantees about $G_{\rm agg}(b)$, and later in \cref{subsec:het-alg} we provide detailed pseudocode and conclude performance of the algorithm.

\subsection{Residual-balance signal}\label{subsec:population-signal}
In this section, we prove results about how well $G_{\rm agg}(b)$ aligns with $\beta - b$. 
We will begin by showing in-expectation bound.
We define population analogs of $S_w$ and $A_w$ with
\[
        \overline S_w(b)=\E S_w(b),
        \qquad
        \overline A_w(b)=\sum_{i=1}^n\Pp\left(\abs{y_i-X_i^\top b}\le w\right).
\]

Fixing $b\in B_2(B)$, we will also use the notation
\[
        h=\beta-b,
        \qquad
        r=\norm{h}_\Sig.
\]
We will choose a relevant scale $w=w(r)$ in terms of $r$ and a later-specified constant $w_0$; if $r \le 1$ we choose $w = w_0$, if $r > 1$ we choose a dyadic $w \in \calW$ such that $w_0r \le w \le 2w_0 r$ (such a $w \in \calW$ always exists by $r \le 2 \kappa^{1/2}B$ and \eqref{eq:grid}).
We now prove the in-expectation signal guarantee:
\begin{lemma}[Population signal]\label{prop:drift}
For the scale $w=w(r)$ chosen above,
\begin{equation}\label{eq:finite-drift}
        \ip{h}{\overline S_w(b)}\ge c mr \min\{r,1\},
\end{equation}
\begin{equation}\label{eq:sign-drift}
        \ip{h}{\overline S_\infty(b)}\ge cr \min\{r,1\} \overline A_w(b),
\end{equation}
and for every tested scale $u\in\calW$,
\begin{equation}\label{eq:all-mono}
        \ip{h}{\overline S_u(b)}\ge0.
\end{equation}
\end{lemma}

\begin{proof}
If $r=0$, then $h=0$ and all three inequalities are immediate.
Assume $r>0$. We define the random variable
\[
Z_i = \frac{X_i^\top h}{r},
\]
and note that $\E Z_i^2 = 1$ and \eqref{eq:subg} gives
$\norm{Z_i}_{\psi_2}\le K$. 
We can then rewrite 
\begin{equation}\label{eq:zi-form}
\langle h, \overline S_w(b)\rangle  =  \sum_{i=1}^n \E[ X_i^\top h \cdot  \psi_w (X_i^\top h + \eps_i )] =  \sum_{i=1}^n \E[ r Z_i \cdot  \psi_w (r Z_i + \eps_i )].
\end{equation}
\Cref{eq:finite-drift,eq:sign-drift,eq:all-mono} will all follow from \cref{eq:zi-form}. 
Our proof will show that, no matter the realization of $Z_i$, the expectation of each term is non-negative. Later, we will show that when $\abs{Z_i}$ is not too small, then the expectation will be large enough to prove our desired lower bounds.

First, we work towards proving \eqref{eq:all-mono} by showing each summand $\E[ X_i^\top h \cdot \psi_w(X_i^\top h + \eps_i)] \ge 0$.
Recall the decomposition for symmetric, unimodal variables given in \cref{def:unimodal}.
In this spirit, for a fixed $T\ge0$, let $U\sim\Unif[-1,1]$, and define
\[
        q_{T,w}(a)=\E_U[\sgn(a+TU)\one\{\abs{a+TU}\le w\}],
        \qquad
        s_T(a)=\E_U[\sgn(a+TU)].
\]
For $T>0$,
\begin{equation}\label{eq:sT-clip}
        s_T(a)=\clip(a/T,-1,1),
\end{equation}
and for $T=0$, $s_0(a)=\sgn(a)$. Then:
\begin{claim}[Non-negative contribution]\label{lem:pointwise}
For every $a\in\R$, $T\ge0$, and $w>0$,
\[
        a \cdot  q_{T,w}(a)\ge0,
        \qquad
        a \cdot s_T(a)\ge0.
\]
Consequently, if $\eps$ is symmetric unimodal about zero, then
\[
        a\E[\psi_w(a+\eps)]\ge0.
\]
\end{claim}
\begin{proof}
The claim $a \cdot s_T(a) \ge 0$ follows from \eqref{eq:sT-clip}.  For the other claim,
since $q_{T,w}$ is odd it suffices to consider $a>0$.  The $T=0$ case is immediate.  For
$T>0$, the density of $a+TU$ is
\[
        f_a(t)=\frac{1}{2T}\one\{\abs{t-a}\le T\}.
\]
Thus
\[
        q_{T,w}(a)=\int_0^w\{f_a(t)-f_a(-t)\}\,dt.
\]
For every $t\ge0$, $\abs{t-a}\le\abs{t+a}$.  Hence $f_a(-t)>0$ implies
$f_a(t)>0$, and the integrand is nonnegative.  Therefore $q_{T,w}(a)\ge0$.
\end{proof}
\eqref{eq:all-mono} follows immediately by conditioning on $X_i$ and applying
\cref{lem:pointwise} to $a=X_i^\top h$.
For the remaining claims, we use that $|Z_i|$ is often in a constant-size window bounded away from 0:
\begin{claim}\label{lem:annulus}
Let $Z$ be a real random variable satisfying $\E Z^2=1$ and
$\norm{Z}_{\psi_2}\le K$.  There exist $a,p\in(0,1)$, depending only on
$K$, such that
\[
        \Pp(a\le\abs{Z}\le a^{-1})\ge p.
\]
\end{claim}
\begin{proof}
The sub-Gaussian bound gives $\E Z^4\le C$.  Paley-Zygmund applied to
$Z^2$ gives
\[
        \Pp\left(Z^2\ge\frac12\E Z^2\right)
        \ge \frac{(1/2)^2(\E Z^2)^2}{\E Z^4}
        \ge c.
\]
Choose $R\ge1$ so that $\Pp(\abs{Z}>R)\le c/2$, and set
$a=\min\{1/\sqrt2,R^{-1}\}$ and $p=c/2$.
\end{proof}

We are now ready to state the main quantitative bounds in terms of $Z_i,\eps_i$:
\begin{claim}\label{lem:scalar-drift}
Let $Z$ satisfy $\E Z^2=1$ and $\norm{Z}_{\psi_2}\le K$; $Z$ may have nonzero mean.  Let $\eps$ be independent of $Z$ and symmetric unimodal about zero.
There are constants $c,w_0>0$, depending only on $K$, such that the following
hold.  For $r>0$, recall that $w=w(r)=w_0$ if $0<r\le1$, and $w_0r\le w \le2w_0r$ if $r > 1$.
Then:
\begin{enumerate}[leftmargin=2em]
\item If $\Var(\eps)\le1$, 
\begin{equation}\label{eq:finite-scalar}
        \E[Z\psi_w(rZ+\eps)]\ge c\min\{r,1\}.
\end{equation}
\item For arbitrary symmetric unimodal $\eps$,
\begin{equation}\label{eq:sign-scalar}
        \E[Z\sgn(rZ+\eps)]
        \ge c\min\{r,1\}\Pp(\abs{rZ+\eps}\le w).
\end{equation}
\end{enumerate}
\end{claim}
\begin{proof}
Use the decomposition $\eps=TU$ as in \cref{def:unimodal}.  Let $a,p$ be supplied by
\cref{lem:annulus}.  If $\Var(\eps)\le1$, then $\E T^2\le3$.  Choose
\[
        M=\max\{a^{-1}+1,(6/p)^{1/2}\},
        \qquad w_0=a^{-1}+M.
\]
This was chosen such that $\Pp(T\le M)\ge1-p/2$.
Since $Z$ and $T$ are independent, the event 
\[
\mathcal E=\{a\le\abs{Z}\le a^{-1} \textrm{ and }\ T\le M\} 
\]
satisfies $\Pp(\mathcal E)\ge p/2$.

\textit{Proving \eqref{eq:finite-scalar}.} We prove this in two cases. First, suppose $0<r\le1$.  On $\mathcal E$, the support of $rZ+TU$ is
contained in $[-w_0,w_0]$, so the windowed sign equals the full sign.  Using \eqref{eq:sT-clip},
\[
        Z\E_U[\sgn(rZ+TU)]
        =\abs{Z}\min\left\{\frac{r\abs{Z}}{T},1\right\}.
\]
On $\mathcal E$, this
quantity is at least $ra^2/M$.  By \cref{lem:pointwise}, the contribution
outside $\mathcal E$ is nonnegative.  Hence
$\E[Z\psi_{w_0}(rZ+\eps)]\ge cr = c \min\{r,1\}$.

Otherwise, $r>1$.  We use $w \ge w_0 r \implies w\ge r(a^{-1}+M)\ge ra^{-1}+M$.  Thus, on $\mathcal E$ the window again contains the
full support of $rZ+TU$, and
\[
        \abs{Z}\min\left\{\frac{r\abs{Z}}{T},1\right\}
        \ge a\min\left\{\frac{ra}{M},1\right\}\ge c = c \min\{r,1\}.
\]
Again, contribution outside $\mathcal E$ is nonnegative, concluding \eqref{eq:finite-scalar}.

\textit{Proving \eqref{eq:sign-scalar}.}  Condition on $T=t$.  Using
\eqref{eq:sT-clip}, let
\[
        D_t:=\E_Z[Zs_t(rZ)]
        =\E_Z\left[\abs{Z}\min\left\{\frac{r\abs{Z}}{t},1\right\}\right].
\]
\cref{lem:annulus} gives
\begin{equation}\label{eq:Dt-lower}
        D_t\ge c\min\left\{1,\frac r t\right\}.
\end{equation}
When $T=0$, the following bounds will hold trivially. Otherwise, conditioning on $Z$ and $T=t>0$, the random variable $tU$ has density bounded by
$(2t)^{-1}$ on an interval of length $2t$, so
\begin{equation*}
        \Pp_U(\abs{rZ+tU}\le w\mid Z,T=t)
        \le\min\left\{1,\frac w t\right\}.
\end{equation*}
Since $w\asymp \max\{r,1\}$ and
$\min\{r,1\}\max\{r,1\}=r$, we may control the quantity relating to the RHS of \eqref{eq:sign-scalar},
\[
        \min\{r,1\}\Pp(\abs{rZ+tU}\le w\mid T=t)
        \le C\min\left\{1,\frac r t\right\}.
\]
Combining this with \eqref{eq:Dt-lower} and integrating over $T$ proves
\eqref{eq:sign-scalar}.
\end{proof}

We are finally prepared to prove \eqref{eq:finite-drift} and \eqref{eq:sign-drift}.
For $i$ where $\Var(\eps_i)\le1$, applying \cref{lem:scalar-drift} to
$Z_i=X_i^\top h/r$ gives
\[
        \E[Z_i\psi_w(rZ_i+\eps_i)]\ge c\min\{r,1\}.
\]
Multiplying by $r$ gives the form of \eqref{eq:zi-form}, and summing over the good indices yields
\eqref{eq:finite-drift}; other indices have nonnegative contribution by
\cref{lem:pointwise}.  
\cref{lem:scalar-drift} also gives, for every $i$,
\[
        \E[Z_i\sgn(rZ_i+\eps_i)]
        \ge c\min\{r,1\}\Pp(\abs{rZ_i+\eps_i}\le w).
\]
Similarly, multiplying by $r$ and summing over $i$ gives \eqref{eq:sign-drift}. 
\end{proof}

We now must show that our algorithm empirically witnesses this signal with high probability.
The following uniform convergence event will be crucial (the proof follows from existing techniques and is deferred to \cref{app:empirical-process}):
\begin{proposition}[Uniform convergence]\label{prop:sample-event}
With probability at least $1-\delta$, all of the following statements hold
simultaneously.
\begin{enumerate}[label=\textnormal{(\roman*)},leftmargin=2.5em]
\item \textbf{Residual-balance vector accuracy.}  For every $b\in B_2(B)$ and every
$w\in\calW$,
\begin{equation}\label{eq:score-conc}
        \norm{S_w(b)-\overline S_w(b)}_2\le \tau_w(b) = L^2
        \left(\sqrt{d(A_w(b)+d)}+d\right).
\end{equation}

\item \textbf{Count comparability.}  For every $b\in B_2(B)$ and every
$w\in\calW$,
\begin{equation}\label{eq:count-comp}
        A_w(b)\le 2\overline A_w(b)+CLd,
        \qquad
        \overline A_w(b)\le 2A_w(b)+CLd.
\end{equation}

\item \textbf{Feature norms.}
\begin{equation}\label{eq:max-X}
        \max_{1\le i\le n}\norm{X_i}_2\le \sqrt{Ld}.
\end{equation}
\end{enumerate}
\end{proposition}

For some brief intuition about the concentration of $S_w(b)$ given by \eqref{eq:score-conc}, observe that $S_w(b)$ can be written as $\sum_{i \textrm{ s.t. } r_i(b) \in (0,w] } X_i - \sum_{i \textrm{ s.t. } r_i(b) \in [-w,0)} X_i$. Both sums correspond to the sum of all $X_i$ in an intersection of two halfspaces over $\R^{d+1}$. Since the intersection of two such halfspaces has VC dimension $\lesssim d$, then the \textit{counts} of the number of samples inside the halfspace intersection should concentrate on the order of $\sqrt{d \overline A_w(b)} + d$ (ignoring logarithmic factors). An analogous bound is proven for $\norm{S_w(b)-\overline S_w(b)}_2$ by considering a net of 1-dimensional projections.

Using this concentration result, observe that the concentration of $S_w(b)$ degrades as $\tau_w(b)$ increases (or equivalently, as $A_w(b)$ increases). This illuminates a tradeoff at the heart of the algorithm's analysis: (i) if $A_w(b)$ is small, then the signal at scale $S_w(b)$ is visible; otherwise, if $A_w(b)$ is large, then \cref{prop:drift} implies there will be strong signal at scale $S_\infty(b)$. We now turn towards making this intuition rigorous.

For $1\le s\le n$, define
\begin{equation}\label{eq:rho-s}
        \eta(s)=L^5
        \left[
        \sqrt d\left(\frac{n}{s^4}\right)^{1/6}+\frac d s
        \right].
\end{equation}
The function $s\mapsto\eta(s)$ is nonincreasing, and the first summand is larger in the regime of $m$ we study \eqref{eq:threshold}. In this notation, the desired error bound \eqref{eq:main-bound} for \cref{thm:main-het} may be replaced with $C_{K,\kappa}\eta(m)$. Using the aforementioned tradeoff, we now prove that when $\norm{b-\beta}_\Sig>\eta(m)$, there is an $S_{w^*}$ which is strongly correlated with $\beta - b$:

\begin{lemma}[A far estimate has a good residual-balance vector]\label{lem:strong-score}
Assume the lower bound condition on $m$ \eqref{eq:threshold} and let $b\in B_2(B)$.  If
$\norm{b-\beta}_\Sig>\eta(m)$, then some
$w_\star\in\calW\cup\{\infty\}$ satisfies
\begin{equation}\label{eq:directional-strong}
        \ip{S_{w_\star}(b)}{\beta-b}
        \ge 2\kappa^{1/2}| \calW |
        \norm{b-\beta}_\Sig\,\tau_{w_\star}(b),
\end{equation}
under the uniform convergence event. In particular, $w_\star$ is violated at $b$.
\end{lemma}

\begin{proof}
Let $w=w(r)$ as before. \cref{prop:drift,prop:sample-event} imply
\begin{equation}\label{eq:finite-emp-drift}
        \ip{S_w(b)}{h}
        \ge c m\min\{r,1\} r-\kappa^{1/2}r
        L^2\left(\sqrt{d(A_w(b)+d)}+d\right).
\end{equation}

On the other hand, at the $\infty$ scale, \cref{prop:drift,prop:sample-event} imply
\begin{equation*}
        \ip{S_\infty(b)}{h}
        \ge c r \min\{r,1\} \overline A_w(b) -\kappa^{1/2}r L^2 \left(\sqrt{2d^2}+d\right).
\end{equation*}

Since $A_w(b)$ and $\overline A_w(b)$ concentrate near each other, this motivates the proof strategy: prove that the signal at scale $w$ is visible when $A_w(b),\overline A_w(b)$ are small enough, and that otherwise the signal at scale $\infty$ is visible.

First, we focus on scale $w$. If $m\min\{r,1\}\ge C | \calW |\tau_w(b)$ for a large enough constant $C \ge 1$, then
\eqref{eq:finite-emp-drift} proves \eqref{eq:directional-strong} with
$w_\star=w$.

Otherwise, using \eqref{eq:Nsc-L} and
\eqref{eq:tau-def} yields the following, which will later help us lower bound $A_w(b)$ in a helpful way,
\begin{equation}\label{eq:failed-finite}
        m\min\{r,1\}\lesssim L^3\left(\sqrt{dA_w(b)}+d\right).
\end{equation}
Our plan is to convert this into a lower bound for $\overline A_w(b)$, which will yield strong signal at the scale $w_* = \infty$.
We claim the $L^3d$ term cannot account for more than a fixed fraction (say, half) of the LHS.  For
$r\le1$, that would imply $r\lesssim L^3d/m$, whereas
$r>\eta(m)\ge L^5d/m$.  For $r>1$, it would imply
$m\lesssim L^3d$,  yet $n \ge d$ and \eqref{eq:threshold} imply $m \ge L^5d$. Both cases incur a contradiction when {L is defined with a sufficiently large constant.}  Thus the square-root term in
\eqref{eq:failed-finite} must account for a fixed fraction, and
\begin{equation}\label{eq:A-lower}
        A_w(b)\gtrsim\frac{m^2\min\{r,1\}^2}{L^6d}.
\end{equation}

The count comparison \eqref{eq:count-comp} gives
$\overline A_w(b)\ge A_w(b)/2-CLd$.  To absorb this additive term, we claim that
\[
        \frac{m^2\min\{r,1\}^2}{L^6d}\ge L^4d.
\]
If $r\le1$, this follows from
$r>\eta(m)\ge L^5d/m$; otherwise, if $r>1$, this follows from \eqref{eq:threshold} and $n \ge d$.
Together with \eqref{eq:A-lower}, this shows $A_w(b)\gtrsim L^4d$, so {when $L$ is defined with a sufficiently large constant},
\begin{equation*}
        \overline A_w(b)\gtrsim\frac{m^2\min\{r,1\}^2}{L^6d}.
\end{equation*}
We now aim to use this lower bound on $\overline A_w(b)$ to imply signal at scale $S_\infty(b)$. 
Since $n\ge d$, 
\[
        \tau_\infty(b)
        =L^2\left(\sqrt{d(n+d)}+d\right)
        \le3L^2\sqrt{dn}.
\]
Using \eqref{eq:sign-drift} and \eqref{eq:score-conc} yields
\begin{equation}\label{eq:sign-emp-drift}
        \ip{S_\infty(b)}{h}
        \ge c\frac{m^2\min\{r,1\}^3r}{L^6d}-CL^2r\sqrt{dn}.
\end{equation}
It remains only to compare the leading term with the concentration error.  If
$r\le1$, then the first term in \eqref{eq:rho-s} and $r>\eta(m)$ give
\[
        \frac{m^2\min\{r,1\}^3}{L^6d}
        =\frac{m^2r^3}{L^6d}
        \ge L^9\sqrt{dn}.
\]
If $r>1$, then \eqref{eq:threshold} implies
$m^2/(L^6d)\ge L^4\sqrt{dn}$.  Thus in both cases the leading term in
\eqref{eq:sign-emp-drift} is at least $cL^4r\sqrt{dn}$.  {When $L$ is defined with a sufficiently large constant}, the leading term dominates the concentration
error and the desired margin
\[
        | \calW |r\tau_\infty(b)\le3L^3r\sqrt{dn}. \rtag{via \eqref{eq:Nsc-L}}
\]
Thus
\eqref{eq:directional-strong} holds with $w_\star=\infty$.

Finally, $\norm{h}_2\le\kappa^{1/2}r$, so \eqref{eq:directional-strong} gives
\[
        \norm{S_{w_\star}(b)}_2
        \ge2| \calW |\tau_{w_\star}(b)
        >| \calW |\tau_{w_\star}(b).
\]
Hence, we conclude that $w_\star$ is violated.
\end{proof}
We now show the signal is still preserved once aggregated into $G_{\rm agg}(b)$:
\begin{lemma}[Normalized aggregate direction]\label{lem:aggregate-margin}
Define $D=\sqrt{\kappa n} \ge 1$.
Assume the $m$ lower bound \eqref{eq:threshold} and $b\in B_2(B)$.  Under the uniform convergence event, if
$\norm{b-\beta}_\Sig>\eta(m)$, then $\calV(b)\ne\varnothing$,
$G_{\rm agg}(b)\ne0$, and
\begin{equation}\label{eq:aggregate-normalized-margin}
        \left\langle
        \frac{G_{\rm agg}(b)}{\norm{G_{\rm agg}(b)}_2},\beta-b
        \right\rangle
        \ge\frac{\norm{b-\beta}_2}{D}.
\end{equation}
\end{lemma}

\begin{proof}
\cref{prop:drift,prop:sample-event} imply that for every tested scale $u$,
\[
        \ip{S_u(b)/\tau_u(b)}{h}\ge-\kappa^{1/2}r.
\]
\cref{lem:strong-score} supplies one violated scale $w_\star$ for which
\[
        \ip{S_{w_\star}(b)/\tau_{w_\star}(b)}{h}
        \ge2\kappa^{1/2}|\calW|r.
\]
There are at most $|\calW|-1$ other violated scales.  Hence,
\begin{equation}\label{eq:aggregate-inner}
\begin{aligned}
        \ip{G_{\rm agg}(b)}{h}
        &\ge 2\kappa^{1/2}|\calW|r
             -(|\calW|-1)\kappa^{1/2}r \ge |\calW| r.
\end{aligned}
\end{equation}

The feature norms guarantee of \cref{prop:sample-event},
$\norm{S_w(b)}_2\le \sqrt{Ld} A_w(b)$, and $A_w(b) \le n$ yield
\[
        \frac{\norm{S_w(b)}_2}{\tau_w(b)}
        \le \frac{\sqrt{Ld} \cdot A_w(b)}{L^2 \sqrt{d A_w(b)}}
        \le \sqrt{ n}.
\]
Therefore, 
\begin{equation}\label{eq:aggregate-norm}
        \norm{G_{\rm agg}(b)}_2
        \le |\calW| \cdot   \sqrt{ n}.
\end{equation}
Equations \eqref{eq:aggregate-inner} and \eqref{eq:aggregate-norm}, together
with $r\ge\norm{h}_2/\sqrt\kappa$, prove
\eqref{eq:aggregate-normalized-margin}.  They also show $G_{\rm agg}(b) \ne 0$. 
\end{proof}

\subsection{Concluding algorithmic guarantees}\label{subsec:het-alg}
\textbf{Algorithm description.} We now more formally introduce our algorithm that uses the residual-balance vectors to descend towards an estimate close to $\beta$. Our algorithm will repeatedly move in the direction of $G_{\rm agg}(b)$, with step sizes changing in different stages. 
Define the smallest stage radius as
\begin{equation*}
        \eta_{\min}=\eta(n)=L^5\left(\sqrt{d/n}+d/n\right).
\end{equation*}
Since $m\le n$ and $\eta$ is nonincreasing,
$\eta_{\min}\le\eta(m)$.  If $B\le\eta_{\min}$, the zero vector already has
the desired accuracy, so the algorithm
returns it.  Otherwise set
\begin{equation*}
        q=\frac34,
        \qquad
        T_{\rm stage}=
        \left\lceil\frac{\log(B/\eta_{\min})}{\log(4/3)}\right\rceil.
\end{equation*}
For $\ell=0,\ldots,T_{\rm stage}$, let $R_\ell=q^\ell B$, which is chosen such that $R_{T_{\rm stage}}\le\eta_{\min}$. 
Finally, for each stage we choose the step size and number of iterations as
\begin{equation}\label{eq:stage-step}
        \alpha_\ell=\frac{R_\ell}{8D},
        \qquad
        N_{\rm it}=\lceil64D^2\rceil.
\end{equation}

\paragraph{Algorithm $\mathsf{RB\mbox{-}Desc}$.}
If $B\le\eta_{\min}$, return $0$.  Otherwise initialize $b=0$.  For stages $\ell=0,1,\ldots,T_{\rm stage}-1$:
\begin{enumerate}[leftmargin=2em]
\item Set $\alpha=\alpha_\ell$.
\item For $j=1,\ldots,N_{\rm it}$:
    \begin{enumerate}[leftmargin=2em]
    \item Compute all scores $S_w(b)$, counts $A_w(b)$, and thresholds
    $\tau_w(b)$.
    \item If $\calV(b)=\varnothing$, return $b$.
    \item Compute $G_{\rm agg}(b) = \sum_{w \in \calV(b)} \frac{S_w(b)}{\tau_w(b)}$.  If
    $G_{\rm agg}(b)=0$, return $b$.
    \item Set $u=G_{\rm agg}(b)/\norm{G_{\rm agg}(b)}_2$ and update
    \begin{equation*}
        b\leftarrow \Pi_{B_2(B)}(b+\alpha u),
    \end{equation*}
    where $\Pi_{B_2(B)}$ is Euclidean projection onto $B_2(B)$.
    \end{enumerate}
\end{enumerate}
Return the final iterate.

\textbf{Algorithm analysis.} The following will let us show $b$ gets close to $\beta$:
\begin{lemma}[One-stage contraction]\label{lem:stage-contract}
Assume \eqref{eq:threshold} and work on the uniform convergence event.  Consider a
stage with radius $R$ and the step size and iteration count as in
\eqref{eq:stage-step}.  If the stage begins at $b$ with
$\norm{b-\beta}_2\le R$ and
\begin{equation}\label{eq:pre-target-stage}
        R\ge4\sqrt\kappa\,\eta(m),
\end{equation}
then, unless the algorithm returns during the stage, its output $b^+$ satisfies
\[
        \norm{b^+-\beta}_2\le\frac34R.
\]
\end{lemma}

\begin{proof}
For the $t$-th iteration in a stage, let $b_t$ be the value of $b$ at the start of the iteration,  $D_t=\norm{b_t-\beta}_2^2$, and $u_t$ is the value of $u$ during this iteration.  Whenever
$\norm{b_t-\beta}_2>R/2$, \eqref{eq:pre-target-stage} implies
$\norm{b_t-\beta}_\Sig>\eta(m)$.  Hence
\cref{lem:aggregate-margin} gives
\[
        \ip{u_t}{\beta-b_t}
        \ge\frac{\norm{b_t-\beta}_2}{D}
        >\frac{R}{2D}=4\alpha.
\]
Since projection onto a closed convex set containing $\beta$ satisfies $\norm{\Pi_{B_2(B)}(z)-\beta}_2\le\norm{z-\beta}_2$, 
\begin{equation}\label{eq:dist-shrink}
        D_{t+1} 
        \le \norm{b_t + \alpha u_t - \beta}_2^2 
        = D_t - 2\alpha \ip{u_t}{\beta - b_t} + \alpha^2
        \le D_t-7\alpha^2.
\end{equation}
If the stage never entered $B_2(\beta,R/2)$, this decrease would occur on every
iteration, whereas
\[
        N_{\rm it}\alpha^2
        \ge64D^2\frac{R^2}{64D^2}=R^2,
\]
contradicting $D_0\le R^2$ and $D_t > 0$.  Thus, the half-radius ball is entered.  From a point within $R/2$, the next
update has length at most $R/(8D)\le R/8$ and therefore lands within $5R/8$.
From a point whose distance lies between $R/2$ and $5R/8$, the preceding
descent inequality makes the distance decrease.  Induction shows that every
later iterate remains within $5R/8<3R/4$.
\end{proof}
Similar logic shows that once $b$ is close to $\beta$, it will not move away:
\begin{lemma}\label{lem:post-target-stability}
Assume \eqref{eq:threshold} and work on the uniform convergence event.
Define $\overline R=4\sqrt\kappa\,\eta(m)$.
If $\norm{b-\beta}_2\le\overline R$ and an update is made during a stage with
$R_\ell\le\overline R$, then the next iterate also lies in
$B_2(\beta,\overline R)$.
\end{lemma}

\begin{proof}
The update length is at most $\overline R/(8D)\le\overline R/8$.  Thus an
iterate in $B_2(\beta,\overline R/2)$ remains inside
$B_2(\beta,\overline R)$ after one update.  If instead
$\overline R/2<\norm{b-\beta}_2\le\overline R$, then
$\norm{b-\beta}_\Sig>\eta(m)$ and
\[
        \ip{u}{\beta-b}
        \ge\frac{\norm{b-\beta}_2}{D}
        >\frac{\overline R}{2D}
        \ge4\alpha_\ell.
\]
By the same argument as \eqref{eq:dist-shrink}, the squared Euclidean distance decreases, and hence no update crosses the boundary of $B_2(\beta,\overline R)$.
\end{proof}
We combine \cref{lem:stage-contract,lem:post-target-stability} to conclude our desired guarantee:
\begin{lemma}\label{prop:descent-localization}
Assume \eqref{eq:threshold} and work on the uniform convergence event. The output of
$\mathsf{RB\mbox{-}Desc}$ satisfies
\[
        \norm{\widehat\beta-\beta}_2\le C\eta(m).
\]
\end{lemma}
\begin{proof}
If $B\le\eta_{\min}$, then
$\norm{\beta}_2\le B\le \eta(m)$, so the initial
return is valid.  Moving forward, we assume $B>\eta_{\min}$.

An early return with $\calV(b)=\varnothing$ or $G_{\rm agg}(b)=0$ is within
$C\eta(m)$ by \cref{lem:aggregate-margin}.

Suppose no early return occurs.  As long as
$R_\ell\ge\overline R$, \cref{lem:stage-contract} and induction show that
the stage begins within $R_\ell$ and ends within
$R_{\ell+1}=3R_\ell/4$.  If some stage is the first with
$R_\ell<\overline R$, its starting point is already in
$B_2(\beta,\overline R)$, and \cref{lem:post-target-stability} keeps all
later iterates there. Otherwise, the last stage ends within
\[
        R_{T_{\rm stage}}
        \le\eta_{\min}
        \le\eta(m)
        < \overline R.
\]
Thus in every case
$\norm{\widehat\beta-\beta}_2\le\overline R =4 \sqrt{\kappa} \eta(m)$, proving the claim.
\end{proof}
This concludes the high-probability estimation error guarantee. We also observe the algorithm has the desired runtime:
\begin{proposition}\label{prop:descent-runtime}
$\mathsf{RB\mbox{-}Desc}$ runs in time
\[
        O_{K,\kappa}(L^2 n^2 d) = O_{K,\kappa}\left( n^2 d \log^2 \left(\frac{end(B+1)}{\delta}\right)\right).
\]
\end{proposition}

\begin{proof}
There are $O(\log(2+B/\eta_{\min}))$ stages, and each stage has $O(D^2)=O_\kappa(n)$ inner iterations. One inner iteration costs
$O(nd |\calW|)$ arithmetic operations.  In total, the number of arithmetic operations is 
\[
        O_{K,\kappa}(L^2 n^2 d) = O_{K,\kappa}\left( n^2 d \log^2 \left(\frac{end(B+1)}{\delta}\right)\right). \qedhere
\]
\end{proof}

\begin{proof}[Proof of \cref{thm:main-het}]
The theorem follows immediately from the uniform convergence event (\cref{prop:sample-event}), estimation error guarantee (\cref{prop:descent-localization}), and runtime bound (\cref{prop:descent-runtime}).
\end{proof}

%% file: adaptive.tex
\section{Adaptive linear regression}
In this section, we study the task of adaptive linear regression. In \cref{subsec:adapt-mixture}, we give an (inefficient) estimator for when the errors are mixtures of symmetric, log-concave distributions. In \cref{sec:k1}, we give an efficient estimator for when the errors are a single symmetric, log-concave distribution.

\subsection{Adaptive guarantees for symmetric, log-concave mixture errors}\label{subsec:adapt-mixture}
In this section, we show positive guarantees for an (inefficient) statistical estimator when the errors are mixtures of symmetric log-concave densities. This regression setting generalizes the main result of \cite{compton2026attainability}, and our estimator is a natural high-dimensional version of their approach. 

We outline our approach in \cref{sec:contribs}; we now set up our estimator more precisely. 

Formally, the error density belongs to the class
\[
\calP_k:=
\left\{
 p=\sum_{\ell=1}^k w_\ell p_\ell:
 w_\ell\ge0,\ \sum_{\ell=1}^k w_\ell=1,\ 
 p_\ell\text{ log-concave and symmetric about }0
\right\}.
\]

Later, we will also use the following condition:
\begin{definition}[Small-ball condition]\label{def:small-ball}
The design satisfies an $(c_1,c_2)$-small-ball condition if
\begin{equation*}
\Pp\!\left(|X^\top u|\ge c_1\norm{u}_\Sigma\right)\ge c_2
\qquad
\text{for every }u\in\R^d\text{ with }\norm{u}_\Sigma>0.
\end{equation*}
\end{definition}

\paragraph{Symmetrizing the covariate distribution.} Since our algorithm relies on testing some notion of symmetry, we will require that $X^\top v \stackrel{d}= X^\top (-v)$. This is not automatically true, since we do not assume any symmetry in the covariate distribution, yet the following classical transformation (e.g. this also appears in \cite[Section A.2]{d2021consistent}) resolves this issue.

Let $S_i$ be independent Rademacher signs, independent of the sample, and define
\begin{equation*}
\widetilde X_i=S_iX_i,
\qquad
\widetilde y_i=S_i y_i,
\qquad
\widetilde\varepsilon_i=S_i\varepsilon_i.
\end{equation*}
Then,
\[
\widetilde y_i=\widetilde X_i^\top\beta+\widetilde\varepsilon_i.
\]
Since the distribution of $\eps_i$ is symmetric, the distribution of $\widetilde\varepsilon_i$ is unchanged, and $\widetilde X_i$ and $\widetilde\varepsilon_i$ are independent. $\widetilde X_i$ now satisfies our desired symmetry condition, while its tail bounds and second moment matrix are unaffected since
\[
|\widetilde X^\top u|=|X^\top u|,
\qquad
\E[\widetilde X\widetilde X^\top]=\Sigma.
\]
For ease of notation, we will assume this preprocessing has already been applied in the remainder of this section, meaning that $X_i,y_i$ will implicitly represent $\widetilde X_i, \widetilde y_i$.

\paragraph{Designing the detection sets.}
For a candidate $b\in\R^d$, recall that the residual is $r_i(b) = y_i - X_i^\top b$.
For $c\in\R^d$, $s>0$, and an interval $I=[-b_0,-a_0]$ with $0\le a_0<b_0<\infty$, define
\begin{align}
N^+_{b,c,s,I}
&:=\sum_{i=1}^n
\one\{X_i^\top c \ge s,\ r_i(b)+X_i^\top c  \in I\},
\label{eq:Nplus}\\
N^-_{b,c,s,I}
&:=\sum_{i=1}^n
\one\{X_i^\top c \le-s,\ r_i(b)- X_i^\top c \in I\},
\label{eq:Nminus}
\end{align}
and let
\begin{equation*}
\ell_{n,\delta}:=\max\{1,\log(2en/\delta)\},
\qquad
A_{n,d,\delta}:=(d+1)\ell_{n,\delta},
\qquad
\Gamma:=C_\Gamma\sqrt{A_{n,d,\delta}},
\end{equation*}
{where $C_\Gamma$ is chosen to be sufficiently large later}. Then, we define a confidence set
\begin{equation}
\calC_n:=
\left\{
 b\in\R^d:
 \left|\sqrt{N^+_{b,c,s,I}}-\sqrt{N^-_{b,c,s,I}}\right|
 \le\Gamma
 \text{ for every }c,s,I
\right\}.
\label{eq:confidence-set}
\end{equation}
Currently, it should not be clear to the reader why this confidence set will typically contain only good estimates. Let us defer that question and first discuss some simpler intuition for why the confidence set will usually contain $\beta$. For the correct estimate $b=\beta$, the residuals $r_i(b)\stackrel{d}=\eps_i$ and are independent of $X_i$. Additionally, by symmetry of $X_i$, the distribution of $X_i^\top c$ restricted to $X_i^\top c \ge s$ should be exactly the same as the distribution of $-X_i^\top c$ restricted to $X_i^\top c \le -s$. Hence, the expectation of both counts $N^+_{b,c,s,I}, N^-_{b,c,s,I}$ are the same when $b=\beta$, and if we employ a uniform convergence result and set the threshold accordingly, then we should be able to show that $\beta \in \calC_n$ with high probability. The more interesting component of our proof is showing that only good estimates will be inside this confidence set.

\paragraph{Analysis.} Observe how each detection set is an intersection of three halfspaces in $\R^{d+1}$, so the relevant concept class has VC dimension $O(d)$.
As we examine the counts inside these sets, the following relative uniform convergence guarantee will be helpful (the proof is standard and is deferred to \cref{app:sqrt-unif}):
\begin{lemma}[Uniform convergence for square-root counts]\label{lem:sqrt-vc}
Let $Z_1,\ldots,Z_n$ be i.i.d. random elements of a Euclidean space. Let $\calA$ be a concept class with VC dimension at most $V\ge1$. For $A\in\calA$, define
\[
N(A):=\sum_{i=1}^n\one\{Z_i\in A\},
\qquad
Q(A):=\sum_{i=1}^n\Pp(Z_i\in A).
\]
For $\delta\in(0,1/2)$, let $\ell_{n,\delta}=\log(2en/\delta)$.
Then, with probability at least $1-\delta$,
\begin{equation}
\sup_{A\in\calA}
\left|\sqrt{N(A)}-\sqrt{Q(A)}\right|
\le C\sqrt{V\ell_{n,\delta}},
\label{eq:sqrt-vc}
\end{equation}
for a universal constant $C>0$.
\end{lemma}
The following is the main theorem, which will immediately imply \cref{thm:main-adapt}:
\begin{theorem}\label{thm:adapt-technical}
Assume $p\in\calP_k$, the covariates are $\kappa$-well-conditioned, and satisfy a $(c_1,c_2)$-small-ball condition, where $\kappa \ge 1$ and $c_1,c_2>0$ are all constants. There exists a constant $C_{c_1,c_2,\kappa} \ge 1$ depending only on $c_1,c_2,\kappa$ such that the following holds. If $C_\Gamma$ is sufficiently large, then with probability at least $1-\delta$ the confidence set~\eqref{eq:confidence-set} satisfies:
\begin{enumerate}[label=(\alph*)]
\item $\beta\in\calC_n$;
\item for every $b\in\calC_n$,
\begin{equation*}
\norm{b - \beta }_2 \le C_{c_1,c_2,\kappa} \omega_p \left( C_{c_1,c_2,\kappa} \frac{dk \log(2n/\delta) \log(2kn) \log(2n)} {n}\right) .
\end{equation*}
\end{enumerate}
\end{theorem}
\begin{proof}
Work on the event of \cref{lem:sqrt-vc} which holds with probability $1-\delta$.

\emph{The true parameter passes.}
This follows from the earlier-discussed intuition.
At $b=\beta$, $r_i(b)=\varepsilon_i$. For any $c,s,I$, since $X^\top c \stackrel{d}= -X^\top c$ and $X^\top c$ is independent of $\varepsilon$,
\[
\Pp\{X^\top c\ge s,\varepsilon+X^\top c\in I\}
=
\Pp\{X^\top c\le-s,\varepsilon-X^\top c\in I\}.
\]
Thus, the expected counts are equal. By \cref{lem:sqrt-vc}, their empirical square-root counts differ by at most $2C\sqrt{A_{n,d,\delta}}$; we choose $C_\Gamma$ large enough to cover this constant, and thus $\beta\in\calC_n$.

\emph{A bad estimate produces a population imbalance.}
Fix $b\in\calC_n$ and let $u=b-\beta$. We will aim to witness a poor estimate with a test that chooses $c=u$. We focus on the case where $u \ne 0$ as otherwise $b=\beta$. For convenience, we define
\[
T=X^\top u,
\qquad
T_i=X_i^\top u ,
\]
and remark that $r_i(b) = \eps_i - T_i$. At this point, it will be very informative to revisit the detection sets with these choices,
\begin{align}
&N^+_{b,c,s,I}
:=\sum_{i=1}^n
\one\{X_i^\top c\ge s,\ r_i(b)+ X_i^\top c \in I\}
=\sum_{i=1}^n
\one\{T_i \ge s,\ \eps_i  \in I\},\label{eq:mod-stat-1}
\\
&N^-_{b,c,s,I}
:=\sum_{i=1}^n
\one\{X_i^\top c\le-s,\ r_i(b)-X_i^\top c\in I\}
=\sum_{i=1}^n
\one\{T_i \le -s,\ \eps_i - 2T_i  \in I\}.\label{eq:mod-stat-2}
\end{align}
This is the key insight for the analysis: the first statistic \eqref{eq:mod-stat-1} roughly corresponds to counting the number of samples where $\eps_i$ is inside an interval $I$, and the second statistic \eqref{eq:mod-stat-2} roughly corresponds to counting the number of samples where a translation of $\eps_i$ is inside $I$, with the translation being at least $2s$. The main point is that we have now roughly reduced the problem to distinguishing samples from $\eps_i$ against samples from a translation, via an interval statistic. This is precisely the main technical ingredient in \cite{compton2026attainability}, and we may use their following result:
\begin{lemma}[Interval witness; follows from Corollary 2.18 of \cite{compton2026attainability}]\label{lem:interval-witness}
There is a universal constant $c_0>0$ such that the following holds. Let $p\in\calP_k$, $\Delta>0$, and $h=H_p(\Delta)\in(0,1]$. There is an interval
\[
I=[-b,-a],
\qquad 0\le a<b<\infty,
\]
such that
\begin{equation*}
\left(\sqrt{\Pp_p(I)}-\sqrt{\Pp_{p_\Delta}(I)}\right)^2
\ge
c_0\frac{h}{k\log(4k/h)\log(4/h)}.
\end{equation*}
Moreover, for every $\Delta'\ge\Delta$,
\begin{equation}
\Pp_{p_{\Delta'}}(I)\le \Pp_{p_\Delta}(I)\le \Pp_p(I).
\label{eq:interval-monotone}
\end{equation}
\end{lemma}
We now work to use this in our setting. Define 
\[
m(s):=\Pp(T\ge s)=\Pp(T\le-s),
\qquad
h(s):=H_p(2s).
\]
Invoking \cref{lem:interval-witness} with $\Delta=2s$ yields an interval $I_s$. We compute the probability of a sample contributing to $N^+_{b,c,s,I_s}$ as
\[
P_s^+ := \Pp(T \ge s,\ \eps  \in I_s)
=m(s)\Pp_p(I_s).
\]
For $N^-_{b,c,s,I_s}$, we use that the conditional translation is $-2T \ge 2s$, and \eqref{eq:interval-monotone}, to bound
\[
P_s^- := \Pp(T \le -s,\ \eps - 2T  \in I_s)
\le m(s)\Pp_{p_{2s}}(I_s).
\]
Consequently,
\begin{equation*}
\left(\sqrt{P_s^+}-\sqrt{P_s^-}\right)^2 \ge
m(s)\left(\sqrt{\Pp_p(I_s)}-\sqrt{\Pp_{p_{2s}}(I_s)}\right)^2 \ge
c_0m(s)\frac{h(s)}{k\log(4k/h(s))\log(4/h(s))}.
\end{equation*}
Leveraging the $(c_1,c_2)$-small-ball condition, by choosing $s = c_1 \norm{u}_\Sigma$, we may use $m(s) \ge c_2 /2$ which implies
\begin{equation}
\left(\sqrt{P_s^+}-\sqrt{P_s^-}\right)^2  \gtrsim \frac{h(s)}{k\log(4k/h(s))\log(4/h(s))}.
\label{eq:population-gap}
\end{equation}
This roughly indicates that as $h(\norm{b - \beta}_\Sigma)$ gets large, this becomes detectable by our tests.

\textit{Passing the empirical test bounds the estimation error.}
Since $b\in\calC_n$ and both counts satisfy \cref{lem:sqrt-vc},
\begin{align*}
\sqrt n\left|\sqrt{P_s^+}-\sqrt{P_s^-}\right|
&\le \left|\sqrt{n P_s^+} - \sqrt{N^+_{b,c,s,I_s}} \right| + \left|\sqrt{N^+_{b,c,s,I_s}} - \sqrt{N^-_{b,c,s,I_s}} \right|  + \left|\sqrt{n P_s^-} - \sqrt{N^-_{b,c,s,I_s}} \right| \\
&\le (C_\Gamma+2C)\sqrt{A_{n,d,\delta}}.
\end{align*}
 
Combining with~\eqref{eq:population-gap} yields
\begin{equation*}
h(s)
\le
C'\frac{kA_{n,d,\delta} \log \frac{4k}{h(s)} \log\!\frac{4}{h(s)}}{n},
\end{equation*}
for some constant $C' \ge 1$. We observe that this further implies
\begin{equation}
    h(s) \le C' \frac{kA_{n,d,\delta} \log(4kn) \log(4n)} {n}.\label{eq:final-h-bound}
\end{equation}
For sake of contradiction, if \eqref{eq:final-h-bound} did not hold, then using $h(s)\ge 1/n$,
\[
C' \frac{kA_{n,d,\delta} \log(4kn) \log(4n)} {n} < h(s) \le C'\frac{kA_{n,d,\delta}\log\frac{4k}{h(s)}\log\frac{4}{h(s)}}{n}
\le C'\frac{kA_{n,d,\delta} \log(4kn) \log(4n)}{n},
\]
which is a contradiction. We recall our choice of $s = c_1 \norm{u}_\Sigma = c_1 \norm{b - \beta}_\Sigma$ to conclude
\begin{align*}
    &H_p(2 c_1 \norm{b - \beta}_\Sigma) \le C \frac{kA_{n,d,\delta} \log(4kn) \log(4n)} {n}\\
    &\implies 2c_1 \norm{b - \beta}_\Sigma \le \omega_p \left( C \frac{kA_{n,d,\delta} \log(4kn) \log(4n)} {n}\right) \\
    &\implies \norm{b - \beta }_2 \le C \omega_p \left( C \frac{dk \log(2n/\delta) \log(2kn) \log(2n)} {n}\right) \quad\qedhere
\end{align*}
\end{proof}
\begin{proof}[Proof of \cref{thm:main-adapt}] This follows immediately from \cref{thm:adapt-technical}; the small-ball condition is implied, for example, by \cref{lem:annulus} (where $c_1,c_2$ depend only on $K$). We output any $b \in\calC_n$ (tiebreak arbitrarily), or an arbitrary estimate if $\calC_n$ is empty.
\end{proof}

\subsection{Efficient estimation for symmetric, log-concave errors ($k=1$)}\label{sec:k1}
In this section, we efficiently obtain Hellinger modulus rates for symmetric, log-concave errors (\cref{thm:k1}).

\textit{Proof intuition.} 
Our estimator for this task will be a form of $L_q$ regression where we choose $q$ depending on the data. This style of estimator has been comprehensively studied by \cite{kao2024choosing}; in \cref{sec:contribs} we discuss some of the technical obstacles for this technique. When analyzing the estimator, there are two main components: (i) proving there exists a $q$ where $L_q$ regression produces a good enough estimate, and (ii) adaptively choosing the right $q$. 

The adaptive choice (ii) of $q$ follows simply for our purposes. Suppose we have divided our samples into a small number of batches, and we are considering a fixed $q$. If we compute the $L_q$ estimator for each batch, and most of these estimates are close to some value $\wh\beta$, then we may show that it is very likely the true $\beta$ is close to $\wh\beta$. This intuition leads to a rigorous method for choosing the estimate, where you choose the value of $q$ with the best concentration. (In particular, we choose an estimate that has a large fraction of the other estimates in the smallest-possible radius around it; this style of technique appears previously in works such as \cite{nissim2007smooth,hsu2014heavy,pensia2022estimating}.)

The main difficulty is in (i): showing there exists a choice of $q$ where $L_q$ regression recovers the Hellinger modulus error. 
A guiding intuition comes from the work of \cite{compton2026attainability}, where they observe the Hellinger distance between a log-concave distribution and its translation is nearly captured by counting the number of samples inside a particular halfline (this uses the reverse data processing inequality of \cite{pensia2023communication}). This may be converted into a quantitative guarantee about the tail of $p$ that relates to the Hellinger modulus; we will use this tail control to prove guarantees for $L_q$ regression.

Another useful intuition is that $L_q$ regression may have guarantees in terms of the moments of $\eps \sim p$; this is crucially leveraged in prior works. Consider the following quantities,
\[
M_q = \E[|\eps|^q],
\qquad
V_p(q) = \frac{M_{2q-2}}{(q-1)^2 M_{q-2}^2}
\qquad
Q_p(q) = \frac{M_{2q-4}}{M_{q-2}^2}.
\]
We will leverage guarantees in terms of such quantities, as is done in e.g. \cite{el2023optimal,kao2024choosing}. Roughly, $Q_p(q)$ is a term that controls curvature, and when $Q_p(q) \ll \frac{n}{d}$, then $L_q$ estimation will typically have error at the scale
\[
\norm{\wh\beta - \beta}_2 \lesssim \sqrt{\frac{d V_p(q)}{n}}.
\]
This form of guarantee is quite strong; for example, it yields the right error for Gaussian errors with $q=2$, and uniform errors with $q \asymp n$. 

The main goal is to turn the earlier-mentioned tail control into a strong enough bound for $V_p(q)$ with the correct $q$; the hope is that the right choice of $q$ can extract the signal from the tail decay. We now very loosely describe the structure of this proof. Consider any $t$ where $\hels(p,p_t) \approx \frac{d}{n}$. If we could show there exists a choice of $q$ where
\begin{equation}\label{eq:vq-goals}
V_p(q) \ll \frac{t^2n}{d},
\qquad
Q_p(q) \ll \frac{n}{d},
\end{equation}
then this would be sufficient for our goals. Suppose we know the distribution $p$ has a tail such that 
\[
\Pp_{\eps \sim p}(|\eps| \ge r) \approx \tau,
\]
and the remaining $\tau$ mass decays exponentially over a local length $\ell$. In the most optimistic setting, there might be a choice of $q$ where the moment bounds satisfy
\[
M_{q-2} \approx \tau r^{q-2},
\qquad
M_{2q-2} \approx \tau r^{2q-2},
\qquad
M_{2q-4} \approx \tau r^{2q-4}.
\]
For this target, a natural choice of $q$ is the largest value such that the integrand of the moment integral is decreasing at $r$. Differentiating gives
\begin{equation*}
\frac{d}{dx}\left(q(r+x)^{q-1} \Pp(|\eps| \ge r+x)\right)\bigg|_{x=0} \approx \frac{d}{dx}\left(q(r+x)^{q-1} \tau\exp(-x/\ell)\right)\bigg|_{x=0} \propto -r^{q-1} \frac{1}{\ell} + (q-1) r^{q-2},
\end{equation*}
which is $0$ when $q \approx \frac{r}{\ell}$.
If this optimistic heuristic were true, then we could use the bounds
\[
V_p(q) \approx \frac{\ell^2}{\tau},
\qquad
Q_p(q) \approx \frac{1}{\tau}.
\]
Fortuitously, the Hellinger tail control result will control precisely these two quantities $\frac{\ell^2}{\tau}$ and $\frac{1}{\tau}$ such that $V_p(q)$ and $Q_p(q)$ would satisfy the desired conditions \eqref{eq:vq-goals}. While many of these calculations are heuristic, a proof along these lines can be made rigorous.

\paragraph{Algorithm description. }
Our algorithm will consist of computing $L_q$ regression estimators for many batches of samples.
Let
\[
B = \lceil C \log(64n/\delta)\rceil,
\qquad
m = \left\lfloor \frac{n}{B} \right\rfloor,
\]
where $C>0$ is a sufficiently large universal constant. We split the data into $B$ batches of size $m$, where the batches will be used to generate candidate estimates.
As is discussed and justified in \cref{subsec:adapt-mixture}, we will consider a preprocessing of all samples where $(X_i,Y_i) \rightarrow (S_i X_i, S_i Y_i)$ for independent Rademacher signs $S_i$; hence, we may consider the $X_i$ as centrally symmetric and the $\eps_i$ as still independent of $X_i$. 
Our algorithm will only consider $L_q$ regression for $q$ that are powers of 2 (as is done in \cite{kao2024choosing}):
\[
\mathcal{G}_m = \{2^j : j \in \{1,\dots,\lfloor \log_2 m \rfloor \}\}.
\]
We will compute the $L_q$ regression estimator for every batch and every $q \in \mathcal{G}_m$. For a fixed $q$, we denote our polynomial-time approximate optimizers for each of the $B$ batches as
\[
\wh{A}_{q,1},\dots,\wh{A}_{q,B},
\]
and denote the exact optimizers as
\[
A_{q,1},\dots,A_{q,B},
\]
where they can be any exact optimizer if there are multiple.
Roughly, we will use that an estimate $\wh{A}_{q,i}$ is better if there are many estimates $\wh{A}_{q,j}$ near it. We introduce $r_{q,i}$ as an effective radius for the estimate $\wh{A}_{q,i}$; we define it as the empirical $3/4$ quantile of the Euclidean distances
\[
\norm{\wh{A}_{q,j} - \wh{A}_{q,i}}_2,
\qquad
j \in \{1,\dots,B\} \setminus  \{i\}.
\]
Finally, we will choose the estimate with the smallest effective radius, denoted by
\[
(\wh q, \wh i) = \argmin_{q \in \mathcal{G}_m, \, i \in \{1,\dots,B\}} r_{q,i},
\qquad
\wh \beta = \wh{A}_{\wh q, \wh i} \,.
\]

The algorithm runs in polynomial time since there are $\polylog(n/\delta)$ instances of $L_q$ regression computed, each taking polynomial time.

\paragraph{$L_q$ regression guarantee in terms of moment control.}
We now provide a regression guarantee in terms of the moment control. 
{We do not believe this guarantee requires anything conceptually novel compared to prior works (e.g. \cite{el2023optimal,kao2024choosing}), yet we require a non-asymptotic guarantee with the right dependence on $q$ (since our $q$ may take super-constant values) in a way that we do not know an external result we may invoke as a black-box.}
We defer the proof of the following to \cref{sec:fixed-q}:
\begin{proposition}[Fixed-$q$ regression bound]\label{prop:fixed-q}
There exists a constant $C_{K,\kappa}$ such that, if
\begin{equation}\label{eq:prop-n-cond}
m\ge C_{K,\kappa}dQ_p(q)\log(2m),
\end{equation}
then with probability at least $17/18$, the exact empirical $L_q$ optimizer (for $q \ge 2$) given $m$ samples satisfies 
\begin{equation*}
\norm{\widehat\beta_q-\beta}_2
\le C_{K,\kappa}\sqrt{\frac{dV_p(q)}m},
\end{equation*}
and the design matrix $\mathbf{X} \in \R^{m \times d}$ has $\operatorname{rank}(\mathbf{X}) = d$, implying the optimizer is unique.
\end{proposition}

Note that $\wh\beta_q$ in \cref{prop:fixed-q} refers to the exact empirical $L_q$ optimizer. $L_q$ regression is a convex optimization task which permits efficient algorithms, but not with exact solutions. We use existing algorithms that enable an approximate solution (proof deferred to \cref{sec:poly-lq}):
\begin{corollary}[Polynomial-time regression accuracy]\label{cor:poly-lq}
There exist constants $c_{K,\kappa} > 0$ and $C_{K,\kappa} \ge 1$ such that, if
\begin{equation}\label{eq:ptime-m-cond}
m\ge C_{K,\kappa}d \log(2m),
\end{equation}
then with probability at least $17/18$, 
\begin{equation}\label{eq:X-singular}
    \sigma_{\text{min}}^2(\mathbf{X}) \ge c_{K,\kappa} m
\end{equation} 
holds for the design matrix $\mathbf{X} \in \R^{m \times d}$. Further, when \eqref{eq:X-singular} holds, the algorithm of \cite{adil2024fast} outputs an estimate $\wh A_q$ such that it is close to the unique, exact $L_q$ optimizer $\wh\beta_q$,
\[
\norm{\wh A_q - \wh \beta_q}_2 \le \gamma,
\]
and only requires 
\[
O\left(q^2 m^{1/3} \max\left\{1, \log\left(\frac{C_{K,\kappa} m \norm{\mathbf{Y}}_\infty}{\gamma}\right)\right\}\right) \qedhere
\]
calls to a linear system solver (for $0 < \gamma \le 1$).
\end{corollary}
Note that the condition on $m$ \eqref{eq:ptime-m-cond} holds since we may assume $n \ge C_{K,\kappa} d \log^4(64n/\delta)$ (for a large enough choice of constant), as otherwise the desired coefficient bound for the theorem is infinite.

We will use the algorithm of \cite{adil2024fast} for efficiently estimating the $L_q$ regression parameters when the singular value condition of \eqref{eq:X-singular} holds, and otherwise we may output anything.

\paragraph{Proving desirable moment control.}
We will leverage the reverse data processing inequality proven by \cite{pensia2023communication}; we state a version that follows after using \cite[Remark 2.3]{compton2026attainability}:
\begin{theorem}[Reverse data processing inequality; follows from Corollary 3.4 of \cite{pensia2023communication}]\label{thm:rdpi}
Let $P,Q$ be distributions with densities $p,q$, and let
\[
\eta=\hels(P,Q)\in(0,1].
\]
After exchanging $P$ and $Q$ if necessary, there exists a $\lambda\ge1$ where the likelihood-threshold event
\[
A=\{x:p(x)/q(x)\ge\lambda\}
\]
satisfies
\begin{equation*}
\left(\sqrt{P(A)}-\sqrt{Q(A)}\right)^2
\ge\frac{\eta}{1800\log(4/\eta)},
\qquad P(A)\ge\frac{\eta}{1800\log(4/\eta)}.
\end{equation*}
\end{theorem}
Further (as noted in \cite{compton2026attainability}), when $P=p$ and $Q=p_t$ for a log-concave $p$, then the likelihood ratio is monotone, and hence the set $A$ is simply a halfline. For some rough intuition, this result indicates that if $p$ and $p_t$ are distinguishable from $N \approx 1/\eta$ samples, then they are also distinguishable from $\approx N \log N$ samples just by counting the number of samples in some halfline. \cref{thm:rdpi} will enable helpful tail control for $p$ {that aligns with the desired control in the proof intuition}:
\begin{lemma}[Informative tail scale]\label{lemma:tail}
Let $t>0$, $p$ is symmetric and log-concave, and
\[
\eta=H_p(t),
\qquad
L_\eta=\log(4/\eta),
\]
where $0<\eta\le1$. There exists a choice of $r\ge0$, one-sided tail mass
\[
\tau=\Pp_{\eps \sim p}(\varepsilon\ge r),
\]
and a local tail length
\begin{equation*}
\ell=\frac{\tau}{p(r)},
\end{equation*}
such that
\begin{equation}\label{eq:tau-lb}
\tau\ge\frac{\eta}{3600L_\eta},
\end{equation}
and
\begin{equation}\label{eq:ell-tau-ratio}
\frac{\ell^2}{\tau}\le1800\frac{t^2L_\eta}{\eta}.
\end{equation}
\end{lemma}

\begin{proof}
Apply \cref{thm:rdpi} to $p$ and $p_t$. By symmetry, there exists a choice of $A$ of the form $A=(-\infty,a]$. Let $F$ be the CDF of $p$, and denote
\[
P(A)=F(a),
\qquad
Q(A)=F(a-t),
\qquad
\Delta=P(A)-Q(A)\ge0,
\]
where \cref{thm:rdpi} implies
\[
(\sqrt{P(A)}-\sqrt{Q(A)})^2
\ge\frac{\eta}{1800L_\eta}.
\]
We distinguish three positions of the interval $[a-t,a]$ relative to the mode zero.

\textit{Case (i): $a \le 0$.} Choose $r=-a$ and $\tau=F(a)=\Pp(\varepsilon\ge r)$ by symmetry. Since $p$ is nondecreasing on $(-\infty,0]$,
\[
\Delta=\int_{a-t}^ap(x)\,dx\le tp(a)=tp(r).
\]
Moreover,
\[
\frac{\Delta^2}{\tau}
=(\sqrt{P(A)}-\sqrt{Q(A)})^2 \frac{(\sqrt{P(A)}+\sqrt{Q(A)})^2}{P(A)}
\ge (\sqrt{P(A)}-\sqrt{Q(A)})^2
\]
and
\[
\qquad
\tau=P(A)\ge (\sqrt{P(A)}-\sqrt{Q(A)})^2.
\]

\textit{Case (ii): $a-t\ge0$.} Choose $r=a-t$ and $\tau=1-F(a-t)=\Pp(\varepsilon\ge r)$. Since $p$ is nonincreasing on $[0,\infty)$,
\[
\Delta=\int_r^{r+t}p(x)\,dx\le tp(r),
\]
and similarly using $P(A),Q(A) \ge 1/2$,
\[
\frac{\Delta^2}{\tau} = (\sqrt{P(A)}-\sqrt{Q(A)})^2 \frac{(\sqrt{P(A)} + \sqrt{Q(A)})^2}{1 - Q(A)} \ge (\sqrt{P(A)}-\sqrt{Q(A)})^2
\]
and
\[
\tau \ge P(A) - Q(A) \ge (\sqrt{P(A)}-\sqrt{Q(A)})^2.
 \]

\textit{Case (iii): $a-t < 0 < a$.} Choose $r=0$ and $\tau=1/2$. Since $p(x)\le p(0)$,
\[
\Delta=\int_{a-t}^ap(x)\,dx\le tp(r).
\]
and similarly using $P(A) \ge 1/2$,
\[
\frac{\Delta^2}{\tau}=2(\sqrt{P(A)}-\sqrt{Q(A)})^2(\sqrt{P(A)} + \sqrt{Q(A)})^2\ge (\sqrt{P(A)}-\sqrt{Q(A)})^2
\]
and
\[
\tau=1/2\ge (\sqrt{P(A)}-\sqrt{Q(A)})^2/2.
\]

\textit{Combining cases.} All cases satisfy
\[
\tau\ge\frac{\eta}{3600L_\eta},
\qquad
\frac{\Delta^2}{\tau}\ge\frac{\eta}{1800L_\eta},
\qquad
\Delta\le tp(r).
\]
Finally,
\[
\frac{\ell^2}{\tau}
=\frac{\tau}{p(r)^2}
\le\frac{t^2\tau}{\Delta^2}
\le1800\frac{t^2L_\eta}{\eta}.\qedhere
\]
\end{proof}

This enables the desired moment control (this is mostly calculation):

\begin{lemma}[Moment control]\label{lemma:moment-control}
Under the assumptions of \cref{lemma:tail}, there is a positive integer $2 \le q\le2+\frac{900L_\eta}{\eta}$ such that 
\begin{equation*}
V_p(q)\lesssim \frac{t^2L_\eta^3}{\eta},
\qquad
Q_p(q)\lesssim \frac{L_\eta^2}{\eta}.
\end{equation*}
\end{lemma}
\begin{proof}
Let
\[
R=\abs\varepsilon,
\qquad
S(x)=\Pp(R\ge x),
\]
and let $(r,\tau,\ell)$ be supplied by \cref{lemma:tail}. Thus
\begin{equation}\label{eq:s-deriv}
S(r)=2\tau,
\qquad
-\left.\frac{d}{dx}\log S(x)\right|_{x=r}=\frac1\ell.
\end{equation}
First, we note that $S(x)$ is log-concave. To prove this, consider the function $F(x,y) = \one\left[y \ge x \right] R(y)$. Observe that $F$ is log-concave because both $\one\left[y \ge x \right]$ and $R(y)$ are log-concave, and the product of log-concave functions is log-concave. We can write 
\[
S(x) = \int_\R F(x,y) \, dy,
\]
and by Pr\'{e}kopa's theorem \cite[Theorem 6]{prekopa1973logarithmic} this implies $S(x)$ is log-concave. Using \eqref{eq:s-deriv} and log-concavity of $S(x)$, we may bound that for every $y\ge0$ and every $0\le y\le r$, respectively,
\begin{equation}\label{eq:s-bounds}
S(r+y)\le2\tau e^{-y/\ell},
\qquad
S(r-y)\le\min\{1,2\tau e^{y/\ell}\}.
\end{equation}
Let us denote
\[
L_\tau=\log(e/\tau),
\qquad
h=r/\ell.
\]
Using \eqref{eq:tau-lb},
\[
L_\tau
\le\log\left(\frac{3600eL_\eta}{\eta}\right)
\lesssim  L_\eta.
\]
We will give the desired moment control based on two separate cases, depending on whether $h$ is small (corresponding to when $r$ is not too far out in terms of the local tail length), or whether $h$ is large.

\textit{Case (i): $h\le24L_\tau$.}
In this case, the tail is not too far out and we will choose $q=2$. We may write the second moment as
\[
M_2=2\int_0^\infty xS(x)\,dx.
\]
We split the integral at $r$ and bound
\begin{align*}
M_2
&=2\int_0^r xS(x)\,dx+2\int_r^\infty xS(x)\,dx\\
&\le2\int_0^r x\,dx+4\tau\int_0^\infty(r+y)e^{-y/\ell}\,dy \rtag{via $S(x)\le 1$ and the tail bound in \eqref{eq:s-bounds}}\\
& = r^2+4\tau(r\ell+\ell^2)\le r^2+4r\ell+4\ell^2=(r+2\ell)^2\\
&\lesssim  L_\tau^2\ell^2. \rtag{via $r=h\ell\le24L_\tau\ell$}
\end{align*}
Hence, we may bound $V_p(2)$ and $Q_p(2)$ as desired,
\begin{align*}
&V_p(2)=M_2
\lesssim  L_\eta^2\frac{\ell^2}{\tau}
\lesssim \frac{t^2L_\eta^3}{\eta} \rtag{via \eqref{eq:ell-tau-ratio}} \\
&Q_p(2) = 1 \lesssim L_\eta^2/\eta
\end{align*}

\textit{Case (ii): $h>24L_\tau$.}
In this case, we choose a larger $q$. Consider any even integer $k$ satisfying
\begin{equation}\label{eq:k-range}
h\left(1-\frac1{3L_\tau}\right)
\le k\le
h\left(1-\frac1{4L_\tau}\right).
\end{equation}
This interval has length $h/(12L_\tau)>2$, so it contains an even integer. We choose
\[
q=\frac{k+2}{2}.
\]
The relevant moments in the numerators of $V_p(q)$ and $Q_p(q)$ are $M_s$ for $s \in [k-2,k]$. We will prove that for every such $s \in [k-2,k]$,
\begin{equation}\label{eq:m-goal}
M_s\lesssim L_\tau\tau r^s.
\end{equation}
We again use the expression
\[
M_s=s\int_0^\infty x^{s-1}S(x)\,dx
\]
and will split this integral into
\[
M_s=I_-+I_+,
\qquad
I_-=s\int_0^r x^{s-1}S(x)\,dx,
\qquad
I_+=s\int_r^\infty x^{s-1}S(x)\,dx.
\]

For the right-tail integral, write $x=r+y$. Since
\[
(r+y)^{s-1}
=r^{s-1}\left(1+\frac yr\right)^{s-1}
\le r^{s-1}e^{(s-1)y/r},
\]
we may bound $I_+$ by
\begin{align*}
I_+
&\le2s\tau r^{s-1}\int_0^\infty e^{-y/\ell + (s-1)y/r}\,dy\rtag{via \eqref{eq:s-bounds}}\\
&=2s\tau r^{s-1}\int_0^\infty e^{-\delta_sy}\,dy, \numberthis\label{eq:Iplus-inter}
\end{align*}
if we define $\delta_s$ as
\[
\delta_s = \frac{1}{\ell} - \frac{s-1}{r}.
\]
This will be an important quantity that controls the exponent. Using $s \le k$ and \eqref{eq:k-range}, we may lower bound $\delta_s$ by
\[
\delta_s\ge\frac1{4L_\tau\ell}.
\]
Similarly, using $s-1 \ge k-3$ and \eqref{eq:k-range} yields
\begin{align*}
\delta_s
&\le\frac{1}{3L_\tau \ell}+\frac3r\\
&\le\frac{1}{3L_\tau \ell}+\frac{1}{8L_\tau \ell}\rtag{via $h>24L_\tau$ and $3/r=3/(h\ell)\le1/(8L_\tau\ell)$}\\
&\le\frac{1}{2L_\tau \ell} \numberthis \label{eq:last-delta-ub}.
\end{align*}
Together these imply
\begin{equation}\label{eq:delta-range}
\frac1{4L_\tau\ell}
\le\delta_s\le
\frac1{2L_\tau\ell}.
\end{equation}
Revisiting our bound on $I_+$ with \eqref{eq:Iplus-inter} and \eqref{eq:delta-range},
\begin{align*}
I_+
&\le2s\tau r^{s-1}\int_0^\infty e^{-\delta_sy}\,dy\\
&=2s\tau r^{s-1}\frac1{\delta_s}\\
&\le8sL_\tau\ell\tau r^{s-1}\\
&\le8L_\tau\tau r^s. \rtag{via $s\le k\le h=r/\ell$}
\end{align*}
We now turn towards bounding $I_-$. Since
\[
(r-y)^{s-1}
=r^{s-1}\left(1-\frac yr\right)^{s-1}
\le r^{s-1}e^{-(s-1)y/r},
\]
using \eqref{eq:s-bounds} yields
\[
I_-
\le sr^{s-1}\int_0^r
 e^{-(s-1)y/r}\min\{1,2\tau e^{y/\ell}\}\,dy.
\]
We will split this bound into $I_- \le I_{-,1} + I_{-,2}$ where
\[
I_{-,1} = 2s\tau r^{s-1}\int_{0}^{y_0} e^{\delta_s y} \,dy,
\qquad
I_{-,2} = sr^{s-1}\int_{y_0}^r e^{-(s-1)y/r}\,dy,
\qquad
y_0=\ell\log\frac1{2\tau}.
\]
Observe that $0\le y_0<r$ because $\tau \le 1/2$, $\log(1/(2\tau))\le L_\tau$, and $h>24L_\tau$. 

For $I_{-,1}$, we bound
\begin{align*}
I_{-,1}
&=2s\tau r^{s-1}\frac{e^{\delta_sy_0}-1}{\delta_s}\\
&\le 2s\tau r^{s-1}\frac{e^{1/2}-1}{\delta_s} \rtag{via \eqref{eq:delta-range} and $\delta_sy_0 \le\frac1{2L_\tau\ell} \, \ell\log\frac1{2\tau} \le\frac12$}\\
&\lesssim sL_\tau\ell\tau r^{s-1} \lesssim L_\tau\tau r^s. \rtag{again via $s\le k\le h=r/\ell$}
\end{align*}

For $I_{-,2}$, we bound
\begin{align*}
I_{-,2}
&=\frac{s}{s-1}r^s
\left(e^{-(s-1)y_0/r}-e^{-(s-1)}\right)\\
&\le2r^s e^{-(s-1)y_0/r} \rtag{via $s \ge h\left(1 - \frac{1}{3 L_\tau} \right) - 2 > 24 L_\tau - 10 > 2$}\\
&= 2r^s (2\tau)^{(s-1)\ell/r}\\
&\le 2r^s (2\tau)^{1 - 1/(2 L_\tau)} \rtag{via $\frac{(s-1)\ell}{r} \ge 1 - \frac{1}{2L_\tau}$ by the same steps as \eqref{eq:last-delta-ub}}\\
&= 4\tau r^s e^{\log(1/(2\tau))/(2 L_\tau)} \le 4 \tau r^s e^{1/2} \lesssim \tau r^s.
\end{align*}
Combining the bounds for $I_+,I_{-,1},I_{-,2}$ proves our desired bound \eqref{eq:m-goal} for $M_s$; this corresponds to the quantities in the numerators of $V_p(q)$ and $Q_p(q)$. For the moments in the denominators, we may use the much simpler bound
\begin{equation}\label{eq:m-lb}
M_j=\E R^j\ge r^j\Pp(R\ge r)=2\tau r^j.
\end{equation}
We may now bound $V_p(q)$ using \eqref{eq:m-goal} and \eqref{eq:m-lb}
\begin{align*}
V_p(q)
&=\frac{M_k}{(k/2)^2M_{k/2-1}^2}\\
&\lesssim \frac{L_\tau\tau r^k}{(k^2/4)(4\tau^2r^{k-2})} \lesssim \frac{L_\tau r^2}{k^2\tau}\\
& \lesssim \frac{L_\tau r^2}{h^2\tau} = \frac{L_\tau \ell^2}{\tau} \rtag{via $k \ge (2/3)h$ by \eqref{eq:k-range}}\\
& \lesssim \frac{L_\tau t^2 L_\eta}{\eta} \lesssim \frac{t^2 L_\eta^2}{\eta}. \rtag{via \eqref{eq:ell-tau-ratio}}
\end{align*}
For $Q_p(q)$ we similarly bound
\begin{align*}
Q_p(q)
&=\frac{M_{k-2}}{M_{k/2-1}^2} \lesssim \frac{L_\tau\tau r^{k-2}}{4\tau^2r^{k-2}} \lesssim \frac{L_\tau}{\tau} \\
&\lesssim \frac{L_\tau L_\eta}{\eta} \lesssim \frac{L_\eta^2}{\eta} \rtag{via \eqref{eq:tau-lb}}.
\end{align*}

Finally, we upper bound the value of $q$ by
\begin{align*}
q&=\frac{k+2}2
\le1+\frac h2 \rtag{via \eqref{eq:k-range}}\\
&\le1+\frac1{4\tau} \rtag{via $h=\frac r\ell=\frac{rp(r)}\tau\le \frac{\int_0^rp(x)\,dx}{\tau }\le \frac1{2\tau}$}\\
&\le2+\frac{900L_\eta}{\eta}.\rtag{via \eqref{eq:tau-lb}} \qquad\qedhere
\end{align*}
\end{proof}

We further use that there exists a choice of $q$ that is a power of 2:
\begin{corollary}\label{cor:power-q}
There is a choice of $q$ that is a power of 2 and satisfies the conditions of \cref{lemma:moment-control}.
\end{corollary}
\begin{proof}
We will show that for any
\[
2 \le s \le q < 2s,
\]
it holds that
\begin{equation*}
V_p(s) \le 9 V_p(q),
\end{equation*}
and
\begin{equation}\label{eq:q-goal}
Q_p(s) \le Q_p(q).
\end{equation}
Together, these would immediately imply our desired result: consider the $q$ given by \cref{lemma:moment-control}, and choose this value rounded down to the nearest power of 2.

It will be helpful to consider the function 
\[
\phi(a):=\log M_a.
\]
$\phi$ is convex since $M_a$ is log-convex. To prove the latter, we use that for $a,b\ge0$ and $\theta\in[0,1]$, H\"older's inequality gives
\[
M_{\theta a+(1-\theta)b}
=\E\left[(R^a)^\theta(R^b)^{1-\theta}\right]
\le(\E R^a)^\theta(\E R^b)^{1-\theta}
=M_a^\theta M_b^{1-\theta}.
\]
The convexity of $\phi$ will almost immediately imply our desired result. For any $c\ge0$, consider
\[
F_c(x)=\phi(2x+c)-2\phi(x),
\qquad x\ge0.
\]
$F_c$ is non-decreasing since $\phi$ has non-decreasing derivative
\[
F_c'(x)=2\phi'(2x+c)-2\phi'(x)\ge 0.
\]

In these terms, 
\[
\log Q_p(a) = F_0(a-2),
\]
and hence $Q_p(a)$ is non-decreasing for $a\ge 2$, which implies \eqref{eq:q-goal}.
We may similarly bound $V_p(s)$ by
\[
V_p(s) = \frac{e^{F_2(s-2)}}{(s-1)^2} \le \frac{e^{F_2(q-2)}}{(s-1)^2} = \left( \frac{q-1}{s-1}\right)^2 V_p(q) \le \left( \frac{2s-1}{s-1}\right)^2 V_p(q) \le 9 V_p(q).\qedhere
\] 
\end{proof}

\paragraph{Choosing the estimate and concluding the analysis.} Together, our bound for regression error in terms of moments (\cref{prop:fixed-q}), our moment control (\cref{cor:power-q}), and the polynomial-time regression error (\cref{cor:poly-lq}), imply the following guarantee:
\begin{corollary}\label{cor:lq-regress}
For $t>0$, let $\eta = H_p(t)$, and assume $\eta \ge \frac{1}{m}$. There exists a constant $C_{K,\kappa} \ge 1$ such that if 
\[
m \ge C_{K,\kappa} \frac{d L_\eta^2 \log(2n)}{\eta},
\]
then there exists a $q^* \in \mathcal{G}_m$ such that for each $i \in \{1,\dots,B\}$, it holds with probability at least $\frac{7}{8}$ that
\[
\norm{A_{q^*,i} - \beta}_2 \le t\qquad
\text{and}\qquad
\norm{\wh A_{q^*,i} - A_{q^*,i}}_2 \le \gamma.
\]
\end{corollary}
\begin{proof}
\cref{cor:power-q} yields a power of two $q^* \le 2 + 900L_\eta/\eta$. For a large enough choice of $C_{K,\kappa}$, it holds that $q^* \le m$ and hence $q^* \in \mathcal{G}_m$. Further, $q^*$ satisfies
\[
V_p(q^*) \lesssim \frac{t^2L_\eta^3}{\eta},
\qquad
Q_p(q^*) \lesssim \frac{L_\eta^2}{\eta}.
\]
This satisfies \eqref{eq:prop-n-cond} of \cref{prop:fixed-q} when our $C_{K,\kappa}$ is sufficiently large. This yields
\[
\norm{A_{q^*,i} - \beta}_2 \le C_{K,\kappa} \sqrt{\frac{dV_p(q^*)}{m}} \lesssim C_{K,\kappa}  t \sqrt{\frac{d L_\eta^3 / \eta}{m}} \le t
\]
with probability at least $17/18$, when our choice of $C_{K,\kappa}$ is large enough. \cref{cor:poly-lq} yields the final guarantee $\norm{\wh A_{q^*,i} - A_{q^*,i}}_2 \le \gamma$.
\end{proof}
Consider two parameters
\[
\eta_0 = C_{K,\kappa} \frac{d \log^4(64n/\delta)}{n},
\qquad
t_0 = \omega_p(\eta_0) + \gamma.
\]
If $\eta_0 \ge 1$, then our desired theorem is always vacuously true; we otherwise focus on $\eta_0 < 1$.
By \cref{lem:H-translation}, $H_p(t)$ is non-decreasing in $[0,\infty)$. Hence,  $H_p(t_0) >  \eta_0$. We may invoke  \cref{cor:lq-regress} with $t_0$ since the condition on $m$ holds when
\[
m \ge \frac{n}{C_{K,\kappa} \log(64n /\delta)}, 
\]
where the constant in the denominator grows with the constant in the choice of $\eta_0$, and hence the condition holds from our choice of $B$ when we last set the constant in $\eta_0$ large enough.

For a large enough choice of constant in our choice of $B$, we may invoke \cref{cor:lq-regress} and a Chernoff bound to demonstrate that, with probability at least $1-\delta/10$, more than $3/4$ of $A_{q^*,1},\dots,A_{q^*,B}$ will be within $t_0$ of $\beta$ and have $\norm{A_{q^*,i} - \wh A_{q^*,i}}_2 \le \gamma$. For any such $i$, we observe
\[
r_{q^*,i} \le 2t_0 + 2\gamma.
\]

All that remains is to show that all balls nearly contain $\beta$, meaning
\begin{equation}\label{eq:ball-contains}
\beta \in B(\wh A_{q,i},r_{q,i} + \gamma), 
\qquad
\text{for all } q \in \mathcal{G}_m \text{ and } i \in \{1,\dots,B\}.
\end{equation}
This would immediately imply
\begin{equation}\label{eq:final-k1}
\norm{\wh\beta - \beta}_2 \le \min_{q\in\mathcal{G}_m, \, i \in \{1,\dots,B\}} r_{q,i}  + \gamma  \le 2 t_0 + 3\gamma .
\end{equation}

We finally turn towards proving the ball containment claim \eqref{eq:ball-contains}. 

\begin{proof}[Proof of \eqref{eq:ball-contains}]
We use the fact
\begin{claim}\label{claim:ball-contains}
Let $P$ be a distribution that is centrally symmetric around $\beta$. If \[
P( B(c,r) ) > 1/2,
\]
then $\norm{c - \beta}_2 \le r$.
\end{claim}
\begin{proof}
For sake of contradiction, suppose this were not true. Then, $B(c,r)$ and $B(2\beta - c, r)$ are disjoint and have the same probability under $P$; yet this would imply $P(B(c,r)) + P(B(2\beta-c,r)) > 1$ which is a contradiction.
\end{proof}
\cref{claim:ball-contains} will be helpful because the distribution of the exact $L_q$ optimizer is centrally symmetric around $\beta$ for any fixed $\mathbf{X}$ where $\operatorname{rank}(\mathbf{X})=d$ (so the estimator is unique). This follows from the fact that if $v$ is the unique minimizer for a fixed $\beta$, the design matrix $\mathbf{X}$, and the collection of errors $\eps$, then $2\beta - v$ is the unique estimator for $\beta, \mathbf{X},-\eps$, and by independence and symmetry of $\eps$, both are equally likely.

Consider some function $f_q(X_1,\dots,X_m,Y_1,\dots,Y_m)$ which always outputs an exact optimizer of the $L_q$ objective on $m$ samples (it may output any optimizer if there are multiple).
Let $P_\text{all}$ denote the distribution of $f_q$ when the $m$ samples $(X_i,Y_i)$ are sampled i.i.d., and let $P_\text{good}$ denote the distribution of $f_q$ conditioned on the minimum singular value condition of \eqref{eq:X-singular} holding (this also implies $\operatorname{rank}(\mathbf{X})=d$, so the optimizer is unique). 
We have already shown that $P_\text{good}$ is centrally symmetric around $\beta$. Moreover, by \cref{cor:poly-lq}, the singular value condition holds with probability at least $17/18$, and hence we may write
\[
P_\text{all} = \alpha P_\text{good} + (1-\alpha)Q,
\]
for some distribution $Q$ and $\alpha \in [17/18,1]$. Similarly, let $\wh P_{\text{all}}$ denote the distribution of our algorithm's polynomial time estimate of the optimizer given $m$ i.i.d. samples, and let $\wh P_{\text{good}}$ denote the distribution of our algorithm's polynomial time estimate of the optimizer conditioned on the minimum singular value condition holding. Note that $\wh P_{\text{good}}$ need not be symmetric around $\beta$. 
We may likewise write 
\[
\wh P_\text{all} = \alpha \wh P_\text{good} + (1-\alpha) \wh Q,
\]
for some distribution $\wh Q$ and $\alpha \in [17/18,1]$.

Condition on a particular fixed candidate estimate $\wh A_{q,i}$. For each $\wh A_{q,j}$ ($i \ne j$), let us denote the distance from $\wh A_{q,i}$ as $D_j = \norm{\wh A_{q,i}- \wh A_{q,j}}_2$. Denote the CDF of this corresponding random variable as $F$, and the empirical CDF from $D_1,\dots,D_B$ as $F_B$. By DKW inequality \cite{massart1990tight}, we know
\[
\Pp(\sup_x | F_B(x) - F(x)| > a) \le 2e^{-2(B-1)a^2}.
\]
Hence, using $a=0.1$, it holds that
\[
\Pp( \wh P_\text{all}(B(\wh A_{q,i},r_{q,i})) < 0.65) \le 2e^{-0.02(B-1)}. 
\]
In total, there are at most $B \log_2(m)$ candidate estimates, and thus by union bound all candidate estimates satisfy 
\[
 \wh P_\text{all}(B(\wh A_{q,i},r_{q,i})) \ge 0.65
\]
with probability at least $1-\delta/10$ when the constant in $B$ is chosen large enough. Furthermore, this implies that
\[
\wh P_\text{good}(B(\wh A_{q,i},r_{q,i})) > 0.5
\]
for all candidates, as otherwise
\[
\wh P_\text{all}(B(\wh A_{q,i},r_{q,i})) \le \frac{17}{18} \cdot \frac{1}{2} + \frac{1}{18} < 0.65,
\]
which is a contradiction. 

Further, if we consider the relationship between $P_{\text{good}}$ and $\wh P_{\text{good}}$, recall that the algorithmic optimizer is always within $\gamma$ of the true optimizer under the singular value condition, and hence
\[
P_{\text{good}}(B(\wh A_{q,i},r_{q,i} + \gamma )) \ge \wh P_{\text{good}}(B(\wh A_{q,i},r_{q,i} )) > 0.5.
\]

Since $P_\text{good}$ is centrally symmetric around $\beta$, \cref{claim:ball-contains} implies that $B(\wh A_{q,i},r_{q,i} + \gamma)$ contains $\beta$ for all candidates, concluding our proof.
\end{proof}
This implies our desired guarantee by \eqref{eq:final-k1}
\[
\norm{\wh\beta - \beta}_2 \le 2 t_0 + 3\gamma  \le C_{K,\kappa} \omega_p\left( C_{K,\kappa} \frac{d \log^4(64n/\delta)}{n}\right) + 5\gamma.
\]

We may use the above guarantee with $\gamma' = \gamma/5$ for the desired theorem statement. 

Finally, we discuss the running time. The algorithm makes $O(B \log(n)) = O(\log^2(2n/\delta))$ calls to the optimizer of \cite{adil2024fast}. By \cref{cor:poly-lq}, each call solves $O\left(q^2 m^{1/3} \max\left\{1, \log\left(\frac{C_{K,\kappa} m \norm{\mathbf{Y}}_\infty}{\gamma}\right)\right\}\right)$ linear systems. Outside of this, the runtime is dominated by computing pairwise distances between candidate estimates and choosing $(\wh q, \wh i)$; this runs in $O(\log(n) \cdot d B^2) = O(d \log^3(n/\delta))$ time. In total, the algorithm runs in $\poly(n,d) \cdot \polylog(1/\delta, \norm{\mathbf{Y}}_\infty, 1/\gamma)$ time. We refer to \cite{adil2024fast} for a more fine-grained description of optimization subroutine's runtime.

\subsection{Statistical lower bound}

In this section, we will provide a statistical lower bound for adaptive linear regression in terms of the Hellinger modulus. Our lower bound (\cref{thm:main-hel-lb}) will focus on the setting where $X_i \sim N(0,I_d)$, $\eps_i \sim p$, and $p$ is symmetric unimodal.

\textit{Further proof intuition.} Please recall the proof intuition discussed in \cref{sec:contribs}. In this section, we will denote our linear regression parameter as $\Theta$, where it is drawn from a prior.

Recall the original technical obstacle of bounding $\E_{\Theta, X_i}[\kll{p_{X_i^\top \Theta}}{p}]$, and recall the KL decomposition observation discussed in \cref{sec:contribs}, where $P$ is decomposed into two components $P = (1-\alpha)G + \alpha R$, where $\alpha \lesssim  \hels(P,Q)$ and $\kll{G}{Q} \lesssim \hels(P,Q)$. A more typical Fano approach will eventually work well for the samples corresponding to $G$ in our decomposition, yet samples corresponding to $R$ may be more unwieldy. In these cases, our lower bound will consider a more extreme version where instead of observing just $Y_i = R_{X_i^\top \Theta}$, we reveal directly the value of $X_i^\top \Theta$, which is strictly more informative. The key point is that we will tune parameters such that there are typically, say, $\le d/2$ samples corresponding to $R$, and hence $\Theta$ is still uncertain over a subspace of dimension $\ge d/2$. Extra care is required for a rigorous proof. A notable consideration is the revealed subspace is dependent with $\Theta$, since samples from $R$ are more likely when $\hels(p_{X_i^\top \Theta},p)$ is larger, and hence the revealed subspace might correlate with directions where $\Theta$ is large; this will be handled more clearly later in the full proof. 

 \textit{Proof structure.} In \cref{subsec:hel-lemmas}, we begin our proof with some Hellinger helper lemmas that show the Hellinger distance has controlled growth for translations of symmetric unimodal distributions, and the existence of this decomposition that is amenable to $\kl$. In \cref{subsec:lb-conditional}, we prove our desired lower bound, conditional on a lemma that yields a lower bound whenever the dimension of the revealed subspace is not too large, the choice of subspace is not too informative about $\Theta$, and the other samples are not too informative about $\Theta$ restricted to the unrevealed subspace. Finally, in \cref{subsec:lb-lemma}, we prove said lemma.

\subsubsection{Hellinger lemmas}\label{subsec:hel-lemmas}
We first show control for Hellinger between translations:
\begin{lemma}[Translation Hellinger growth]\label{lem:H-translation}
If $p$ is symmetric unimodal, then $H_p$ is even, non-decreasing on $[0,\infty)$, and satisfies
\begin{equation}\label{eq:H-scaling}
  H_p(t)
  \le
  \left(1+\frac{|t|}{a}\right)^2H_p(a)
  \qquad(a>0,\ t\in\R).
\end{equation}
Consequently, every $|r|<\omega_p(u)$ satisfies $H_p(r)\le u$.
\end{lemma}

\begin{proof}
Let $f=\sqrt p$ and
\[
  C_f(t)=\int f(x)f(x-t)\,dx.
\]
Then $H_p(t)=1-C_f(t)$.  Since $f$ is even and non-increasing on $[0,\infty)$, its superlevel sets are centered intervals.  $C_f(t)$ can be written as a double integral over the overlap lengths of the sets where $f(x) \ge a$ and $f(x - t) \ge b$ for $a,b\ge 0$; since each of these overlap lengths are non-increasing in $t$ for $t \ge 0$, then $C_f$ is non-increasing and $H_p$ is non-decreasing for $t \ge 0$. Further, define
\[
  \rho(t)=\norm{f-f(\cdot-t)}_2=\sqrt{2H_p(t)}.
\]
Translation invariance and the triangle inequality imply $\rho(s+t)\le\rho(s)+\rho(t)$.  Thus,
\[
  \rho(t)\le \lceil|t|/a\rceil\rho(t/\lceil|t|/a\rceil)\le \lceil|t|/a\rceil\rho(a)
  \le\left(1+\frac{|t|}{a}\right)\rho(a),
\]
using the monotonicity of $\rho$; squaring proves \eqref{eq:H-scaling}.
\end{proof}

We next show the desirable $\kl$ decomposition:
\begin{lemma}[Hellinger-to-KL decomposition]\label{lem:splitting}
There are universal constants $0<C,K<\infty$ such that, for arbitrary densities $P,Q$, there exist densities $G,R$ and
\[
  \alpha=\min\{1,K \hels(P,Q) \},
\]
for which
\begin{equation*}
  P=(1-\alpha)G+\alpha R,
  \qquad
  \kll{G}{Q} \le C \hels(P,Q).
\end{equation*}
\end{lemma}
\begin{proof}
We work towards using the following convenient fact, which shows the $\kll{P}{Q}$ is bounded in terms of Hellinger when the ratio of $P(x)/Q(x)$ is bounded (this is a well-used fact in the area, and results of this form also appear in e.g. \cite[Lemma 4.4]{birge1983approximation}, \cite[Lemma 5]{birge1998minimum}, \cite[Fact 4]{tang2022divergence}):
\begin{lemma}[KL control from Hellinger and bounded ratio; Lemma 1 of \cite{compton2026ratio}]\label{lem:ratio-control}
Assume that $p,q$ are two densities on $\mathcal{X}$ with respect to $\mu$.
It holds that
\[
\kl(p,q) \le \frac{2}{(\sqrt{e}-1)^2}\,\hels(p,q)\,
\max\left\{1,\ \sup_{x\in\mathcal{X}}\log\left(\frac{p(x)}{q(x)}\right)\right\}.
\]
\end{lemma}
This fact illuminates that the problematic regions of $P$ for $\kll{P}{Q}$ are only those where, say, $Q(x) < P(x)/2$. We define a set corresponding to the favorable region of the domain as
\[
A = \{ x : Q(x) \ge P(x)/2 \}.
\]
Further, we can bound the mass outside this region by
\begin{align*}
&\frac{1}{2} \int_{A^c} \left(\sqrt{P(x)} - \sqrt{Q(x)} \right)^2 \ge \frac{1}{2} \int_{A^c} P(x) \left(\sqrt{1} - \sqrt{1/2} \right)^2 \ge 0.04 \int_{A^c} P(x) \\
&\implies P(A^c) \le 25 \hels(P,Q). \label{eq:pa-bound}\numberthis
\end{align*}
We will choose $K = 50$, and hence we may focus on $P,Q$ where $\hels(P,Q) \le \frac{1}{50}$ and $P(A^c) \le \frac12$. (When $\hels(P,Q) \ge \frac{1}{50}$, set $G=Q$ and $R=P$.) We now choose the distribution $G = (P \mid A)$, and the distribution of $R$ so that $P$ has the right marginal (this will be valid since $(1 - \alpha) \le P(A)$). 

We now must bound $\kll{G}{Q}$. Observe how $G(x) = 0$ outside $A$, and 
\[
G(x) = \frac{1}{1-P(A^c)} P(x) =  \left( 1 + \frac{P(A^c)}{1-P(A^c)}\right) P(x) \le (1 + 2 P(A^c)) P(x) \le 2 P(x) \le 4Q(x) 
\]
inside $A$. Using this and \cref{lem:ratio-control}, we finally bound
\begin{align*}
    \kll{G}{Q} &\lesssim \hels(G,Q) \lesssim \hels(P,Q) + \hels(G,P)\\
    &\le \hels(P,Q) + \int_{A} \left(\sqrt{G(x)} - \sqrt{P(x)} \right)^2 +  \int_{A^c} \left(\sqrt{G(x)} - \sqrt{P(x)} \right)^2 \\
    &\le \hels(P,Q) + \int_{A} P(x) \left(\sqrt{1 + 2P(A^c)} - 1 \right)^2 +  \int_{A^c} P(x) \\
    &\lesssim \hels(P,Q) + \left(\int_{A} P(x) P(A^c)^2 \right)+  P(A^c) \lesssim \hels(P,Q) \quad \qedhere
\end{align*}
\end{proof}

\subsubsection{Proof of \cref{thm:main-hel-lb}}\label{subsec:lb-conditional}
We now prove \cref{thm:main-hel-lb}. Let 
\[
  \eta=c\frac dn,
\]
for a $c>0$ later chosen to be sufficiently small.  We will choose a lower bound incurring error relative to a benchmark
\[
  r = \omega_p(\eta)/2
\]
(where the division by 2 is an arbitrary constant, but it implies $H_p(r) \le \eta$), and choose the prior 
\begin{equation*}
  \Theta\sim N\!\left(0,\frac{r^2}{d}I_d\right).
\end{equation*}
For an independent $X\sim N(0,I_d)$, let $T=X^\top\Theta$ be the exact inner product before adding any error distribution. For any realization $t$, we will compute the decomposition given by \cref{lem:splitting} with $P = p_t$, $Q = p$, and $\alpha(t) = \min\{1,K \hels(p,p_t)\}$.
In terms of this decomposition, the $i$-th observation can be generated by drawing
\[
  B_i\mid(\Theta,X_i)\sim\operatorname{Bern}(\alpha(T_i)),
  \qquad T_i=X_i^\top\Theta,
\]
and then drawing $Y_i$ from $G_{T_i}$ if $B_i=0$, or from $R_{T_i}$ if $B_i=1$.

We will modify this setup, by revealing even more information
\[
  X^n,
  \qquad B=(B_1,\ldots,B_n),
  \qquad T_B=(T_i:B_i=1),
  \qquad Y^n.
\]
This is strictly more informative than the original observations, as it also tells you the exact inner product $T_i$ whenever $B_i=1$.
Given $(X^n,B,T_B)$, the outputs $Y_i$ with $B_i=1$ have known laws $R_{T_i}$ independent of the remaining uncertainty about $\Theta$; hence, these values of $Y_i$ may be deleted.  We define the modified version of $Y^n$ as $W=(W_1,\ldots,W_n)$, where
\[
  W_i=\begin{cases}
    Y_i,&B_i=0,\\
    *,&B_i=1.
  \end{cases}
\]
It suffices to provide a lower bound for estimators that observe
\[
  (X^n,B,T_B,W),
\]
where $T_B$ denotes the collection of exact linear equations $X_i^\top \Theta = T_i$ for $B_i=1$. In the remaining proof, we will show that $T_B$ usually does not contain too many entries, $B$ is not too informative about $\Theta$ conditioned on $X^n$, and $W$ is not too informative about $\Theta$ conditioned on $X^n,B,T_B$. Together, these conditions will imply a lower bound by invoking the following (proof deferred to \cref{subsec:lb-lemma}):
\begin{lemma}\label{lem:rank-info}
Let $\Theta\sim N(0,\sigma^2I_d)$, and let $X$ be side information independent of $\Theta$.  Conditional on $(X,\Theta)$, let $J$ be a random subset of a finite set.  For each $(x,j)$, let $A_j(x)$ be a known matrix with $|j|$ rows, and reveal $U_j=A_j(X)\Theta$. 
Let $W$ be any additional observation.  Suppose
\begin{equation}\label{eq:rank-tail}
  \Pp(|J|>d/2)\le\frac18
\end{equation}
and
\begin{equation}\label{eq:rank-info-assumption}
  2I(\Theta;J\mid X)
  +I(\Theta;W\mid X,J,U_J)
  \le d/2.
\end{equation}
Then, every estimator based on $(X,J,U_J,W)$ satisfies
\begin{equation*}
  \Pp\left(
    \norm{\wh\Theta-\Theta}_2
    \ge e^{-8}\sigma\sqrt{d/2}
  \right)
  \ge\frac38.
\end{equation*}
\end{lemma}
We remark that the constants in the statement of \cref{lem:rank-info} are somewhat arbitrary. We invoke \cref{lem:rank-info} with $X=X^n$, $J=\{ i : B_i = 1\}$, and $A_J(X)$ equal to the matrix of selected rows of $X^n$; in the remaining proof we verify the lemma's conditions.

\paragraph{Few exact measurements.} We must prove $\Pp(|B| > d/2) \le \frac18$. Since the Bernoulli parameters for $B_i$ are $\alpha(T_i) = \min\{1,K \hels(p,p_{T_i})\}$, we naturally want to bound $\E_{\Theta,X} H_p(T)$. Using $\E_{\Theta,X}T^2=r^2$, $\E_{\Theta,X}|T|\le r$, and \cref{lem:H-translation} yields
\begin{equation*}
  \E_{\Theta,X} H_p(T)
  \le H_p(r)\E_{\Theta,X}\left(1+\frac{|T|}{r}\right)^2
  \le4H_p(r).
\end{equation*}
This directly implies
\begin{equation}\label{eq:Bmean-main}
  \E_{\Theta,X^n,B}|B|
  =n\E_{\Theta,X}\alpha(T)
  \le Kn\E_{\Theta,X}H_p(T)
  \le4KnH_p(r)
  \le 4Kcd.
\end{equation}
Our desired condition holds by Markov's inequality when we choose a sufficiently small $c$.

\paragraph{Limited information from $B$.}
We must bound $I(\Theta; B \mid X^n)$.
Condition on $X_i=x$.  Then, $T_i=x^\top\Theta$ is Gaussian with variance $\norm{x}_2^2r^2/d$, and $B_i$ depends on $\Theta$ only through $T_i$.  This means
\begin{equation}\label{eq:first-b-step}
  I(\Theta;B_i\mid X_i=x)
  =I(T_i;B_i\mid X_i=x)
\end{equation}
We will bound this quantity using the following result:

\begin{lemma}[Information in a Gaussian-shift label]\label{lem:label-information}
Let $\alpha:\R\to[0,1]$ be even, nondecreasing in $|t|$, and satisfy
\begin{equation}
  \alpha(sz)\le(1+|z|)^2\alpha(s)
  \qquad(s\ge0,z\in\R).
  \label{eq:alpha-grow}
\end{equation}
Let $Z\sim N(0,1)$ and, conditionally on $Z$, let
\[
  B\sim\operatorname{Bern}(\alpha(sZ)).
\]
Then,
\begin{equation*}
  I(Z;B)\le C\E_Z\alpha(sZ)
\end{equation*}
for a universal constant $C>0$.
\end{lemma}

\begin{proof}
Let $A=\alpha(sZ)$ and $\mu=\E_Z A$.  Monotonicity implies
\[
  \mu\ge\Pp(|Z|\ge1)\alpha(s) \gtrsim \alpha(s),
\]
and the \eqref{eq:alpha-grow} therefore gives
\begin{equation}\label{eq:A-envelope}
  A\le C\mu(1+|Z|)^2.
\end{equation}
Since the marginal of $B$ is $\operatorname{Bern}(\mu)$, the mutual information is
\[
  I(Z;B)=\E_Z\kll{\operatorname{Bern}(A)}{\operatorname{Bern}(\mu)}.
\]
If $\mu \ge 1/2$, then
\[
  I(Z;B)\le H(B)\le\log2\le2\mu\log2.
\]
Otherwise, if $\mu < 1/2$, we use that KL is bounded by chi-square divergence, and \eqref{eq:A-envelope}:
\[
  I(Z;B)
  \le\frac{\E_Z(A-\mu)^2}{\mu} + \frac{\E_Z(A-\mu)^2}{1 - \mu}
  \le\frac{2\E_Z (A- \mu)^2}{\mu}
  \le C\mu\E_Z(1+|Z|)^4
  \lesssim \mu. \qedhere
\]
\end{proof}

We bound the desired quantity by
\begin{align*}
  I(\Theta;B\mid X^n)
  &= H(B \mid X^n) - H(B \mid \Theta, X^n) \le \sum_{i=1}^n H(B_i \mid X^n) - \sum_{i=1}^n H(B_i \mid \Theta , X^n)\\
  &= \sum_{i=1}^n \left( H(B_i \mid X_i) - H(B_i \mid X_i, \Theta)\right) = \sum_{i=1}^n I(\Theta;B_i\mid X_i)\\
  &=\sum_{i=1}^n I(T_i;B_i\mid X_i) && \rtag{via \eqref{eq:first-b-step}}\\
  &\le Cn\E_{\Theta,X}\alpha(T) && \rtag{via \cref{lem:label-information}}\\
  &\le CnH_p(r) \le Ccd,&& \numberthis \label{eq:Binfo-main}
\end{align*}
which satisfies our desired guarantees when $c$ is chosen sufficiently small.

\paragraph{Limited information from $W$.}
Conditional on $(\Theta,X^n,B,T_B)$, the $W_i$ where $B_i=0$ are independent with density $G_{T_i}$.  Compare them with a product reference having density $p$ on every $B_i=0$ coordinate and the $*$ symbol on every $B_i=1$ coordinate.  This yields the bound
\begin{align}
  I(\Theta;W\mid X^n,B,T_B)
  &\le
  \E_{\Theta,X^n,B}\sum_{i:B_i=0}\kll{G_{T_i}}{p}\notag\\
  &\le Cn\E_{\Theta,X}H_p(T)
  \le CnH_p(r) \le Ccd.
  \label{eq:Winfo-main}
\end{align}
\paragraph{Concluding the lower bound.} 
Together \eqref{eq:Bmean-main}, \eqref{eq:Binfo-main} \eqref{eq:Winfo-main} imply the required conditions when $c$ is chosen small enough. This implies that with probability at least $3/8$ any estimator incurs error
\[
c \omega_p(\eta) = c \omega_p \left(c \frac dn \right).
\]

\subsubsection{Proof of \cref{lem:rank-info}}\label{subsec:lb-lemma}
\textit{Proof intuition.} After $J,U_J$ are revealed, the only remaining uncertainty in $\Theta$ is in the orthogonal complement of $\operatorname{row}(A_j(x))$. In the extreme case where $J$ was chosen independently from $\Theta$, and there is no additional observation $W$, then $\Theta$ on this remaining subspace would be distributed like $N(0,\sigma^2 I_{d - |j|})$, and it would be impossible to estimate $\Theta$ with error much better than $\sigma \sqrt{d - |j|}$. However, in our setting $J$ is not chosen independently of $\Theta$, and there is an additional observation $W$, but we assume a bounded mutual information condition for $J$ and $W$ with $\Theta$. Let $P$ be the actual joint distribution over $(X,J,U_J,W,\Theta)$, and let $Q$ denote the same distribution over the observable variables, except the distribution of $\Theta$ on the unmeasured subspace is replaced with $N(0,\sigma^2 I_{d - |j|})$. We work to bound $\kll{P}{Q}$, and then since estimation is hard under $Q$, we will show hardness for estimation under $P$.
\begin{proof}
Fix a realization $(x,j)$.  Let
\[
  \mathcal R_{x,j}=\operatorname{row}(A_j(x)),
  \qquad
  \mathcal N_{x,j}=\mathcal R_{x,j}^{\perp},
\]
We split $\Theta$ into two components, $\Theta = U+V$, where
\[
  U=\Pi_{\mathcal R_{x,j}}\Theta,
  \qquad
  V=\Pi_{\mathcal N_{x,j}}\Theta.
\]
Given $(x,j)$, the measurement $U_j=A_j(x)\Theta$ determines $U$ (and vice versa). Thus, an estimator based on $(X,J,U_J,W)$ has exactly the same information as one based on $(X,J,U,W)$.
For a fixed $(x,j)$, we denote the related Gaussian distributions over the corresponding subspaces as
\[
  \Gamma_{\mathcal R_{x,j}}=N(0,\sigma^2 I_{\mathcal R_{x,j}}),
  \qquad
  \Gamma_{\mathcal N_{x,j}}=N(0,\sigma^2 I_{\mathcal N_{x,j}}).
\]

Let $P$ be the joint distribution over $(X,J,U,V,W)$. We define the distribution $Q$ by the process:
\begin{enumerate}[label=(\roman*),leftmargin=2.2em]
\item draw $(X,J,U,W)$ from its marginal under $P$;
\item conditional on $(X,J,U,W)=(x,j,u,w)$, draw $V\sim\Gamma_{\mathcal N_{x,j}}$ independently of $(u,w)$.
\end{enumerate}
Observe how $P$ and $Q$ have exactly the same marginal for everything observable to the estimator $(X,J,U,W)$; they differ only in the conditional distribution of the hidden component $V$. Our proof will leverage that estimation of $\Theta = U+V$ under $Q$ is hard because there is no observable information about $V$, yet $V$ is still distributed in a well-spread manner under $Q$. By showing $\kll{P}{Q}$ is bounded, this will imply hardness for estimating $\Theta$ under $P$. 

\paragraph{Bounding $\kll{P}{Q}$.} For ease of notation, let the observables be represented by $O = (X,J,U,W)$. Since $P,Q$ have the same distribution over $O$, it holds that
\[
\kll{P}{Q} = \E_{O \sim P} \kll{P_{V|O}}{\Gamma_{\mathcal N_{x,j}}}.
\]
It will be helpful to introduce an intermediary $P_{V | X,J}$ and use 
\[
  \log\frac{P_{V\mid X,J,U,W}}{\Gamma_{\mathcal N_{X,J}}}(v)
  =
  \log\frac{P_{V\mid X,J,U,W}}{P_{V\mid X,J}}(v)
  +
  \log\frac{P_{V\mid X,J}}{\Gamma_{\mathcal N_{X,J}}}(v).
\]
Taking expectation under $P$, this yields
\begin{equation}\label{eq:PQ-decomp}
  \kll{P}{Q}
  = I(V;U,W\mid X,J) + 
  \E_{(X,J)\sim P}
    \kll{P_{V\mid X,J}}{\Gamma_{\mathcal N_{X,J}}}.
\end{equation}
The rest of the proof will bound this quantity. It will be helpful to work with $I(\Theta ; J | X)$ since we assume this quantity is small in \eqref{eq:rank-info-assumption}. We use
\begin{equation}\label{eq:J-posterior-KL-expanded}
  I(\Theta;J\mid X)
  =
  \E_{(X,J)\sim P}
  \kll{P_{\Theta\mid X,J}}{N(0,\sigma^2I_d)}
  =\E_{(X,J)\sim P}
  \kll{
    P_{U,V\mid X,J}}{
    \Gamma_{\mathcal R_{X,J}}\otimes
    \Gamma_{\mathcal N_{X,J}}},
\end{equation}
where the last step used that for any fixed $(x,j)$, there is a bijection between $\Theta$ and $(U,V)$. We use the fact that for any joint distribution $M_{U,V}$ and product distribution $G_U \otimes G_V$, then it holds that
\[
\kll{M_{U,V}}{G_U \otimes G_V}
= \kll{M_U}{G_U} + \kll{M_V}{G_V} + I_M(U; V),
\]
where $I_M(U; V)$ denotes the mutual information between $U,V$ when they are distributed according to $M_{U,V}$. Invoking this with \eqref{eq:J-posterior-KL-expanded} yields
\[
I(\Theta; J | X)
= \E_{(X,J)\sim P} \kll{P_{U | X,J}}{\Gamma_{\mathcal R_{X,J}}} + \E_{(X,J)\sim P} \kll{P_{V | X,J}}{\Gamma_{\mathcal N_{X,J}}} + I(U ; V | X,J),
\]
which by non-negativity of all terms implies
\begin{align*}
  \E_{(X,J)\sim P}
  \kll{P_{V\mid X,J}}{\Gamma_{\mathcal N_{X,J}}}
  &\le I(\Theta;J\mid X),\\
  I(U;V\mid X,J)
  &\le I(\Theta;J\mid X).
\end{align*}
Revisiting \eqref{eq:PQ-decomp} with these implications gives
\begin{align*}
    \kll{P}{Q}
  &= I(U; V |  X,J) + I(V; W |  U, X,J) + 
  \E_{(X,J)\sim P}
    \kll{P_{V\mid X,J}}{\Gamma_{\mathcal N_{X,J}}} \\
& \le 2I(\Theta;J\mid X) + I(V; W |  U, X,J).
\end{align*}
All that remains is to bound $I(V; W | U,X,J)$. Since there is a bijection between $U$ and $U_J$ given $(X,J)$, and a bijection between $V$ and $\Theta$ given $(U,X,J)$, this implies
\[
I(V; W | U,X,J) = I(\Theta ; W | X,J,U_J),
\]
and hence our final bound via assumption \eqref{eq:rank-info-assumption},
\begin{equation}
\kll{P}{Q} \le 2I(\Theta;J\mid X) + I(\Theta ; W | X,J,U_J) \le d/2.\label{eq:PQ-final}
\end{equation}

\paragraph{Showing hardness of estimation under $Q$.} 
Set a target accuracy parameter
\[
  \rho= e^{-8}\sigma\sqrt{d/2}.
\]
Consider an event corresponding to when $|J|$ is not too large and the estimator $\wh\Theta$ is close to $\theta$,
\[
  E=\bigl\{|J|\le d/2,\ \norm{\wh\Theta-\Theta}_2\le\rho\bigr\}.
\]
Fix the observed values $(X,J,U,W)=(x,j,u,w)$ with $|j|\le d/2$, and let $q=\dim(\mathcal N_{x,j}) \ge d/2$.
Under $Q$, conditional on the observed variables, $V$ is independent of the estimator and is distributed according to $N(0,\sigma^2I_q)$ on $\mathcal N_{x,j}$.
Observe how estimating $\Theta$ well requires estimating $V$ well, since
\[
  \norm{\wh\Theta-\Theta}_2\le\rho
  \quad\Longrightarrow\quad
  \norm{V-\Pi_{\mathcal N_{x,j}}\wh\Theta}_2\le\rho.
\]
However, for any estimate $\wh\Theta$, since $V$ is independent of the observations, we will be able to show $v_0 = \Pi_{\mathcal N_{x,j}}\wh\Theta$ is often far from $V$. This follows from how the Gaussian density is at most $(2\pi\sigma^2)^{-q/2}$ everywhere, so
\begin{align*}
  Q\!\left(\norm{V-v_0}_2\le\rho\mid x,j,u,w\right)
  &\le
  (2\pi\sigma^2)^{-q/2}\,\operatorname{vol}(B_2^q(\rho))\\
  &=
  \frac{(\rho/(\sigma\sqrt2))^q}{\Gamma(q/2+1)}\\
  &\le
  \left(\frac{e\rho^2}{\sigma^2q}\right)^{q/2}\\
  &\le
  e^{-15 (q/2)}.
\end{align*}
Here we used $\Gamma(q/2+1)\ge(q/(2e))^{q/2}$ and $\rho^2=e^{-16}\sigma^2(d/2)\le e^{-16}\sigma^2q$. Using $q \ge d/2$ yields
\begin{equation}\label{eq:Q-small-ball-expanded}
  Q(E)\le e^{-15 (d/4)}.
\end{equation}

\paragraph{Transferring hardness from $Q$ to $P$.} For ease of notation, let $\kll{a}{b}$ denote the $\kl$ divergence between $\operatorname{Bern}(a),\operatorname{Bern}(b)$ when $a,b \in [0,1]$. Let $p_E = P(E)$ and $q_E = Q(E)$. Then, by data processing,
\begin{align*}
\kll{P}{Q} &\ge \kll{p_E}{q_E} =  p_E \log \frac{1}{q_E} + (1-p_E) \log \frac{1}{1-q_E} - \left(p_E \log \frac{1}{p_E} + (1-p_E) \log \frac{1}{1-p_E}\right) \\
&\ge p_E \log \frac{1}{q_E} - \log 2.
\end{align*}
Combining this with \eqref{eq:PQ-final} and \eqref{eq:Q-small-ball-expanded} yields
\[
P(E) \le \frac{\kll{P}{Q} + \log 2}{\log \frac{1}{q_E}} \le \frac{d/2 + \log(2)}{15(d/4)} \le \frac{1}{2}.
\]

Finally, if the estimator has error at most $\rho$, then either $E$ occurs or $|J|>d/2$.  Using \eqref{eq:rank-tail},
\begin{align*}
   & \Pp\!\left(\norm{\wh\Theta-\Theta}_2\le\rho\right)
  \le P(E)+\Pp(|J|>d/2)
  \le\frac12+\frac18=\frac58 \\
  & \implies \Pp\!\left(
    \norm{\wh\Theta-\Theta}_2
    \ge e^{-8}\sigma\sqrt{d/2}
  \right)
  \ge\frac38.\quad\qedhere
\end{align*}
\end{proof}

%% file: planted.tex
\section{Planted linear regression}
In this section, we present our results for planted linear regression. In \cref{subsec:exact-positive}, we demonstrate that our heteroskedastic regression estimator yields exact recovery for planted linear regression when $m \gg d^{3/4}n^{1/4}$. In \cref{sec:sq}, we present a Statistical Query (SQ) lower bound nearly matching this threshold. Finally, in \cref{subsec:m-barrier}, we show that no convex $M$-estimator exactly recovers $\beta$ when $m \ll \sqrt{nd}$. Combined, these results inform the computational-statistical landscape in \cref{fig:transition}.

\subsection{Exact recovery via heteroskedastic linear regression}\label{subsec:exact-positive}

We briefly sketch how planted linear regression follows from heteroskedastic regression:

\begin{corollary}\label{cor:planted-positive}
Consider the planted linear regression setting with  $\norm{\beta}_2 \le B$, $n \ge d \ge 1$, $B \ge 1$, and $\delta \in (0,1/2]$. There exists a constant $C \ge 1$ such that when
\begin{equation}\label{eq:planted-threshold}
        m\ge C \log^5 \left( \frac{ndB}{\delta}\right) d^{3/4}n^{1/4},
\end{equation}
then a polynomial-time estimator recovers $\beta$ exactly (up to the precision of solving a linear system) with probability at least $1-\delta$.
\end{corollary}
\begin{proof}
Divide the samples into two halves of size $n/2$. Using a Chernoff bound, each half will have at least $m/4$ noiseless samples, with probability at least $1-\delta/4$. 

Consider a scaling factor $R := n^{10} \log^5\left( \frac{B}{\delta}\right)/\delta$.
Then, consider a rescaled $y_i' =  Ry_i$ for the first half of samples. If we run $\operatorname{RB-Desc}$ on $(X_i,y_i')$, this corresponds to a model $y_i' = X_i^\top \beta' + \eps_i'$, where $\beta' = R \beta$, $\eps_i' = R \eps_i$, and the norm bound for $\beta'$ is now $RB$. 

We now aim to use \cref{thm:main-het}. Observe that the required condition on $m$ holds, as long as we choose a large enough constant in \eqref{eq:planted-threshold}.

Let $\wh\beta'$ denote the estimate given by $\operatorname{RB-Desc}$ on $(X_i,y_i')$. Since the rescaled $\eps_i'$ still have at least $m/4$ noiseless samples, we may conclude by \cref{thm:main-het} (with $\delta' = \delta/4$) that
\[
\norm{\wh\beta' - \beta'}_2 \lesssim \log^5 \left( \frac{ndB'}{\delta}\right) \cdot \sqrt{d} \left( \frac{n}{m^4}\right)^{1/6} \lesssim \log^5 \left( \frac{ndB}{\delta}\right) \cdot \sqrt{d} \left( \frac{n}{m^4}\right)^{1/6}.
\]

We now convert $\wh\beta'$ into an estimate of $\beta$, by defining $\wh\beta = \wh\beta' / R$. This yields
\[
\norm{\wh\beta - \beta}_2 \lesssim \frac{\log^5 \left( \frac{ndB}{\delta}\right) \cdot \sqrt{d} \left( \frac{n}{m^4}\right)^{1/6}}{R} \le \frac{\delta}{n^9},
\]
for a large enough constant in \eqref{eq:planted-threshold}.

With the second half of the samples, consider $\tilde y_i = y_i - X_i^\top \wh\beta$. By a Chernoff bound, at least $d$ of the noiseless samples in this half will have $\abs{\tilde y_i } \le \frac{\delta}{n^9}$, with probability at least $1-\delta/4$. Further, by a union bound, none of the noisy samples in this half will have $\abs{\tilde y_i } \le \frac{\delta}{n^9}$, with probability at least $1-\delta/4$. Thus, we may take any $d$ samples with $\abs{\tilde y_i } \le \frac{\delta}{n^9}$, and output the line that exactly fits all of them (this line is unique with probability 1). 
\end{proof}

The simpler heteroskedastic linear regression algorithm for standard Gaussian covariates (see \cref{sec:simple-gaussian}) would have also sufficed. We remark that the dependence on $B$ in this theorem could be removed by first running OLS to localize within a ball of radius on the order of $\sqrt{d/n}$.

Using the same rescaling argument as \cref{cor:planted-positive} implies that any polynomial-time Subset-of-Signals heteroskedastic regression estimator with $\poly(n,d,1/\delta)$ error when $m \ll d^{3/4}n^{1/4}$, would yield a planted linear regression estimator in the same regime.

\subsection{Statistical Query lower bound}\label{sec:sq}

Let $d \le m \le n$ be positive integers, and $0<\eta \ll \poly(d/n)$ arbitrarily small (say, superexponentially small in $n$).
For each $\beta \in \R^d$, we define a distribution $\calD_\beta$ over $\R^d \times \R$ from which one samples $(x,y)$ as follows: 
\begin{align*}
x &\sim \mathcal{N}(0,\Id_d), \\
\eps &\sim \tfrac{m}{n}\cdot\mathcal{N}(0,\eta^2) + (1-\tfrac{m}{n}) \cdot \mathcal{N}(0,1),\\
y &= \langle \beta,x \rangle + \eps.
\end{align*}

Suppose we observe a sample from $\calD_\beta^{\otimes n}$ for some $\beta$ with $\|\beta\| > 1/2$.
The OLS estimator produces an estimate $\hat\beta_{\OLS}$ with $\|\beta-\hat\beta_{\OLS}\|^2 = (1+o(1))\frac{d}{n}$ with high probability.
We will consider the ability of \emph{statistical query (SQ)} algorithms with access to the $\VSTAT$ oracle to improve over the OLS estimator.

\begin{definition}
    Given a function $\phi:\R^k \to [0,1]$, the $\VSTAT(n)$ oracle to a distribution $\calD$ over $\R^k$ returns $\Ex_{\calD} \phi + \tau$ for an arbitrary $\tau \in \R$ satisfying $|\tau| \le \max\left(\frac{1}{n},\sqrt{\frac{(1-\Ex_{\calD}\phi)\Ex_{\calD} \phi }{n}}\right)$.
\end{definition}

\begin{theorem}\label{thm:SQ}
    Let $n,d,m$ be non-negative integers satisfying $\log n \ll d \ll m \ll n$.
    Suppose $\beta \sim \mathcal{N}(0,\Id_d)$ is unknown.
    When $m \ll (nd^3)^{1/4}$, any estimate $\hat\beta$ produced by $1 < q \le \exp(c\cdot\min(d, \frac{(nd^3)^{1/6}}{m^{2/3}}))$ or fewer $\VSTAT(n)$ queries to $\calD_\beta$ will have error $\|\beta-\hat\beta\|^2 = \Omega(\frac{d}{n \log^2 q})$ with high probability.
\end{theorem}

To prove the theorem, we will invoke the framework of \cite{FGRVX}, wherein SQ lower bounds are a consequence of bounds on a quantity called the \emph{statistical dimension}.
Here we will use slightly different definition:
\begin{definition}
    Let $\calP = \{P_\theta\}_{\theta \in \Theta}$ be a family of probability distributions and $\pi$ a distribution over $\Theta$. 
    Let $Q$ be a reference distribution with the same domain.
    The \emph{statistical dimension} of $\calP$ under $\pi$ relative to $Q$ with average correlation $\xi$ is the largest value $q$ such that for any event $E$ with $\Pr_\pi[E] \ge \frac{1}{q}$, 
    \[
    \chi^2\left( \Ex_{\theta \sim \pi \mid E} P_\theta \,\| \, Q \right) \le \xi.
    \]
\end{definition}

In \cite{FGRVX}, a slightly different notion of statistical dimension is used, in which $\chi^2(\E_{\theta} P_\theta\|Q) = \E_{\theta,\theta' \sim \pi | E} \langle \frac{P_\theta}{Q},\frac{P_{\theta'}}{Q}\rangle_Q -1$ is replaced with the potentially larger $\E_{\theta,\theta' \sim \pi|E}\left| \langle \frac{P_\theta}{Q},\frac{P_{\theta'}}{Q}\rangle_Q -1\right|$.
Their argument can easily be adapted to accommodate our definition, at a loss of constant factors.

We'll say that an SQ algorithm ``fails to reject'' $Q$ if for every query $\phi$ that it makes, the answer $\E_Q \phi$ is valid for the SQ oracle.

\begin{theorem}[Mild adaptation of Theorem 2.7 in \cite{FGRVX}]\label{thm:SQ-chisq}
Let $\delta > 0$.
    If the statistical dimension of $\calP$ under $\pi$ relative to $Q$ with average correlation $\xi$ is at least $q/\delta$, then any adaptive SQ algorithm making $q$ or fewer queries to $\VSTAT(1/9\xi)$ fails to reject $Q$ with probability at least $1-2\delta$ over the choice of $\theta \sim \pi$.
\end{theorem}
We provide a proof of this adaptation; the main difference is that we work with the chi-squared divergence rather than the signed version, for which we introduce the sets $\Theta_i^+$ and $\Theta_i^-$ (whereas in \cite{FGRVX} they are merged together and handled with an absolute value). Also, for us $\pi$ need not be uniform over $\Theta$, though this is mainly a matter of bookkeeping.
\begin{proof}
Let $\cal{A}$ be any adaptive, deterministic algorithm making $q$ queries to $\VSTAT(1/9\xi)$  (the result can be extended to randomized algorithms by averaging over the random seed).
Let $\phi_1,\ldots,\phi_q$ be the sequence of queries made by the algorithm when each $\phi_i$ is answered with $\Ex_Q \phi_i$.
For convenience, define the query means $p_i = \Ex_Q \phi_i$ and $p_{i,\theta} = \Ex_{P_\theta} \phi_i$, as well as the $\VSTAT(1/9\xi)$ tolerances $\tau_{i,Q} = \max\left(9\xi,3\sqrt{\xi p_{i}(1-p_i)}\right)$ and $\tau_{i,\theta} = \max\left(9\xi,3\sqrt{\xi p_{i,\theta}(1-p_{i,\theta})}\right)$.

Let $\Theta^* \subset \Theta$ be the set of parameters $\theta$ for which, when receiving oracle answers about $P_\theta$, $\cal{A}$ rejects $Q$, and suppose that $\pi(\Theta^*) \ge 2\delta$.
Then, for each $i \in [q]$, let $\Theta_i^+ \subset \Theta$ be the set of parameters $\theta$ such that
\[
p_{i,\theta} - p_i > \tau_{i,\theta}
\]
and define $\Theta_i^-$ analogously as the set $\{\theta \in \Theta \mid p_{i,\theta} - p_i < - \tau_{i,\theta} \}$.
Since the event of rejecting $Q$ necessitates at least one $\phi_i$ to be answered with a value other than $p_i$, it is necessarily the case that
\[
\pi(\Theta^*) \le \sum_{i=1}^q \pi(\Theta_i^+) + \pi(\Theta_i^-),
\]
and thus there must exist some $i \in [q], s\in \{\pm 1\}$ such that $\pi(\Theta_i^s) \ge \frac{\pi(\Theta^*)}{2q} \ge \frac{\delta}{q}$; note that this implies our bound on the statistical dimension applies to the event that $\theta \in \Theta_i^s$.
Let $\xi_i^s$ be the average absolute difference in the means on this set:
\[
\xi_i^s 
= \Ex_{\theta \sim \pi \mid \Theta_i^s} s \cdot \left(p_{i,\theta} - p_i\right) 
= \Ex_{\theta \sim \pi \mid \Theta_i^s} s \cdot \left\langle\phi_i, \frac{dP_\theta}{dQ}-1\right\rangle_Q
= \Ex_{\theta \sim \pi \mid \Theta_i^s} s \cdot \left\langle\phi_i - p_i, \frac{dP_\theta}{dQ}-1\right\rangle_Q.
\]
By Cauchy-Schwarz,
\begin{align*}
(\xi_i^s)^2 
&\le \Var_Q(\phi_i) \cdot \chi^2\left(\Ex_{\theta \sim \pi \mid \Theta_i^s} P_\theta\, \|\, Q\right) 
\,\le\, \Var_Q(\phi_i) \cdot \xi,
\end{align*}
where the last inequality follows by our bound on the statistical dimension.
We can also lower bound $\xi_i^s$ by noting that for each $\theta \in \Theta_i^s$,
\[
|p_{i,\theta} - p_i| \ge \frac{3}{2}\sqrt{\xi p_i(1-p_i)}.
\]
Indeed, the $1$-Lipschitzness of the function $x \mapsto x(1-x)$ implies that
\[
\sqrt{\xi p_{i}(1-p_{i})} \le \sqrt{\xi p_{i,\theta}(1-p_{i,\theta})} + \sqrt{\xi \cdot |p_i - p_{i,\theta}|}
\le \frac{1}{3}\tau_{i,\theta} + \frac{1}{3} \sqrt{9\xi \cdot |p_i - p_{i,\theta}|} < \frac{2}{3}|p_i - p_{i,\theta}|,
\]
where we have used that $A \le B+C \implies \sqrt{A} \le \sqrt{B} + \sqrt{C}$, the AM-GM inequality, and the fact that $|p_i-p_{i,\theta}| > \tau_{i,\theta} \ge 9\xi $ by assumption.
Thus
\[
(\xi_i^s)^2  
= \left(\Ex_{\theta \sim \pi \mid \Theta_i^s} |p_i-p_{i,\theta}|\right)^2 > \frac{9}{4}\xi \cdot p_i(1-p_i)
= \frac{9}{4} \Var_Q(\phi_i) \cdot \xi.
\]
But our upper bound on $(\xi_i^s)^2$ is smaller than our lower bound, a contradiction. Hence we must have $\pi(\Theta^*) < 2\delta$.
\end{proof}

Our proof will bound the statistical dimension of an appropriately chosen family of distributions, then appeal to the theorem of \cite{FGRVX}.
As usual, we will use $C,c>0$ to denote constants independent of $n,d,m,q$, whose value may change from line to line.
\begin{proof}[Proof of \cref{thm:SQ}]
We begin with a mathematically convenient way to account for the information available to the algorithm from the OLS estimator:
we will use the classic ``genie'' trick and give the algorithm access to a random vector which is at least as correlated with $\beta$ as the OLS. 
In particular, let 
$s = \sqrt{d/(\ell n)}$ with $\ell = C\log^2 q$.
We decompose
\[
\beta = \tilde{\beta} + \alpha, \qquad \text{ with } \tilde{\beta} \sim \sqrt{1-\tfrac{s^2}{d}}\cdot \mathcal{N}(0,\Id_d),\quad \alpha \sim \tfrac{s}{\sqrt{d}} \cdot \mathcal{N}(0,\Id_d)\quad \text{independently.}
\]
With probability $1-\exp(-cd)$, $a = \|\alpha\|$ satisfies $c \le a \cdot \sqrt{\ell n / d} \le C$; condition on this event from now on.
We reveal to the algorithm $\tilde{\beta}$ and $a = \|\alpha\|$, after which
\[
\beta = \tilde{\beta} + a \cdot v,
\]
where the posterior on $v = \alpha / \|\alpha\|$ is uniform over the sphere $\mathbb{S}^{d-1}$. 
Let $\pi$ denote the uniform measure on $\mathbb{S}^{d-1}$; we will make repeated use of the fact that given $\tilde{\beta},a$, $v \sim \pi$.

Given $\tilde\beta,a$, any statistical query $\phi$ depends on each datapoint $x,y$ only through $x$ and the residual $z = y - \langle x,\tilde\beta \rangle = a \langle x,v\rangle + \eps$.
Let $\mathcal{D}_v$ be the marginal law of $(x,z)$ given $\tilde{\beta},a$.

Fix $\delta > 0$.
Our result will follow from the following fact: there exist constants $C,c > 0$ depending on $\delta$ but independent of $d,n$ such that for all $n,d$ sufficiently large, for all $A \subset \mathbb{S}^{d-1}$ with $1/e \ge \pi(A) \ge e^{-c d}$,
\begin{equation}\label{eq:chisq}
\chi^2\left(\Ex_{u \sim \mathrm{Unif}(A)} \mathcal{D}_u \,\| \, Q\right) \le C\log^2\pi(A)\cdot\left(\frac{a^2}{d} + \frac{m^2}{adn^2}\right).
\end{equation}
Using convexity of chi-squared divergence (any $A$ with $\pi(A) \ge 1/e$ can be written as a mixture of subsets with each $A'$ satisfying $\pi(A') \le 1/e$), this implies nearly the same bound for the broader range $\pi(A) \in [e^{-cd},1]$, 
\begin{equation}\label{eq:v2-chisq}
\chi^2\left(\Ex_{u \sim \mathrm{Unif}(A)} \mathcal{D}_u \,\| \, Q\right) \le C\log^2\frac{e}{\pi(A)}\cdot\left(\frac{a^2}{d} + \frac{m^2}{adn^2}\right).
\end{equation}

Since $a \sim \sqrt{d/(n\log^2 q)}$ and $m \le c(\frac{nd^3}{\log^{6} q})^{1/4}$, the right-hand side of \cref{eq:v2-chisq} is $\le 1/(9n)$ for all $A$ with $\frac{\delta}{q} \le \pi(A) < 1$.
Thus, invoking \Cref{thm:SQ-chisq},  any adaptive SQ algorithm which makes at most $q$ queries to $\mathrm{VSTAT}(n)$ has probability at most $2\delta$ (with respect to the choice of $v \sim \pi$) of rejecting $Q$ when run on $\mathcal{D}_v$.
\medskip
To see why this implies \cref{thm:SQ}, consider any deterministic algorithm which fails to reject $Q$. 
To estimate $\beta$ given $a,\tilde{\beta}$, the algorithm must output a deterministic $\hat\beta = \tilde{\beta} + \hat\alpha$ which is not a function of $\mathcal{D}_v$.
Then
\[
\|\beta - \hat\beta\|^2 = \|a \cdot v - \hat \alpha\|^2 = a^2 + \|\hat\alpha\|^2 - 2a\langle\hat\alpha, v \rangle.
\]
As $v \sim \pi$ the inner product $\langle \hat\alpha,v \rangle$ is subgaussian with variance proxy $\frac{1}{d}\|\hat\alpha\|^2$, and so except with probability $e^{-cd}$, $\|\beta-\hat\beta\|^2 \ge \frac{1}{2} a^2$.
The conclusion follows by averaging over the random seed.
\medskip

We'll prove \cref{eq:chisq} by directly bounding the integral of the centered, squared density of $\E_{u\sim \pi_A} \mathcal{D}_u$ relative to $Q$.
It is crucial that we center using the mixture density $Q$; however, we will be able to substitute a simpler expression in the denominator to make the bound more tractable.
The subsequent bound follows from straightforward (and tedious) analytic arguments. We'll use Gaussian integration to express the integral as an analytic function of $\rho = \langle u,w\rangle$ in expectation under $u,w \sim \pi$ and $u,w \sim \pi|A$, and then use a polynomial expansion to get bounds in terms of the differences of low-order moments in $\rho$ under both distributions.
We will then show that the moments are within $\poly\log\pi(A)$ factors, completing the proof.

Plowing on, let $\gamma^d$ be the density of $\mathcal{N}(0,\Id_d)$ in $\R^d$, let $\gamma_\sigma$ be the density of $\mathcal{N}(0,\sigma^2)$ in $\R$.
For each $v \in \mathbb{S}^{d-1}$, we can express the density $p_v(x,z)$ of $\mathcal{D}_v$ as a mixture:
\[
   p_v(x,z) = \frac{m}{n} \cdot h_v(x,z) + \left(1-\frac{m}{n}\right)\cdot g_v(x,z),
   \]
   where $h_v$ is the small-label-noise component $h_v(x,z) = \gamma^d(x) \cdot \gamma_\eta(z - a\langle x,v\rangle)$ and $g_v$ is the large-label-noise component $g_v(x,z) = \gamma^d(x) \cdot \gamma_1(z - a\langle x,v\rangle)$.

Fix the event $A$ and let $\pi_A = \pi \mid A$.
For convenience let $p_\pi = \Ex_{u \sim \pi} p_u$, let $p_A = \Ex_{u \sim \pi_A} p_u$, and let $g_\pi,g_A,h_\pi,h_A$ be defined similarly.
We can express the chi-squared divergence
\begin{align}
    \chi^2\left( \Ex_{u \sim \pi_A} \mathcal{D}_u \,\|\, Q \right)
    &= \int \int \frac{\left(p_A(x,z) - p_\pi(x,z)\right)^2}{p_\pi(x,z)} dx dz \nonumber\\
    &\le \frac{1}{1-\frac{m}{n}}\int \int \frac{\left(p_A(x,z) - p_\pi(x,z)\right)^2}{g_\pi(x,z)} dx dz
    \le 2 \int \int \frac{\left(p_A(x,z) - p_\pi(x,z)\right)^2}{g_\pi(x,z)} dx dz\label{eq:pause},
\end{align}
where the final inequality follows because $m \ll n$.
By Jensen's inequality, we have that
\[
g_\pi(x,z) = \Ex_{v \sim \pi} g_v(x,z) 
= \gamma^d(x) \cdot \gamma_1(z) \cdot \Ex_{v\sim \pi} \exp( z a \langle x,v \rangle - \tfrac{1}{2}a^2 \langle x,v \rangle^2) 
\ge \gamma^d(x) \cdot \gamma_1(z) \cdot \exp\left(- \tfrac{a^2}{2d}\|x\|^2\right).
\]
So we may substitute the right-hand side above into the denominator in \cref{eq:pause},
\begin{align}
    \cref{eq:pause}
    &\le 2\int \int \frac{\left(p_A(x,z) - p_\pi(x,z)\right)^2}{\gamma^d(x) \gamma_1(z) e^{-\frac{a^2}{2d}\|x\|^2}} dx dz,\nonumber\\
    \intertext{By Cauchy-Schwarz,}
    &\le 4 \left(\tfrac{m}{n}\right)^2\int \int \frac{\left(h_A(x,z) - h_\pi(x,z)\right)^2}{\gamma^d(x) \gamma_1(z) e^{-\frac{a^2}{2d}\|x\|^2}} dx dz  + 4 \int \int \frac{\left(g_A(x,z) - g_\pi(x,z)\right)^2}{\gamma^d(x) \gamma_1(z) e^{-\frac{a^2}{2d}\|x\|^2}} dx dz.\label{eq:leftoff}
    \end{align}

We will now integrate over $x,z$.
To make use of the common form of both integrals, let $k_A = \Ex_{u \sim \pi_A} \gamma^d(x) \gamma_\sigma(z - a \langle x,u \rangle)$, and define $k_\pi$ similarly. 
Note that when $\sigma = 1$ then $k_A = g_A$ and when $\sigma = \eta$ then $k_A = h_A$.
We factor out a multiple of $\gamma^d(x)\gamma_1(z)$ and obtain the bound
\begin{align}
&\int\int \frac{(k_A(x,z) - k_\pi(x,z))^2}{\gamma^d(x)\gamma_1(z)e^{-\frac{a^2}{2d}\|x\|^2}}dx dz\nonumber\\
&\qquad \qquad = \frac{1}{\sigma^2} \Ex_{\substack{x\sim \gamma^d\\z \sim \gamma_1}} e^{-\left(\tfrac{1}{\sigma^2}-1\right)z^2 + \frac{a^2}{2d}\|x\|^2} \left(\Ex_{u \sim \pi_A} e^{\tfrac{a}{\sigma^2}z\langle x,u \rangle  - \tfrac{a^2}{2\sigma^2}\langle x,u\rangle^2} - \Ex_{v \sim \pi} e^{\tfrac{a}{\sigma^2}z\langle x,v \rangle  - \tfrac{a^2}{2\sigma^2}\langle x,v\rangle^2}\right)^2,\nonumber\\
&\qquad \qquad = \Ex_{u_1,u_2 \sim \pi_A} M_\sigma(u_1,u_2)+ \Ex_{v_1,v_2 \sim \pi} M_\sigma(v_1,v_2)- 2 \Ex_{u\sim \pi_A,v\sim \pi} M_\sigma(u,v),\label{eq:diffs}
\end{align}
Where $M_\sigma$ is obtained via Gaussian integration over $x,z$, using that $\sigma \le 1$ and $a \ll 1$ to ensure the integral is defined. 
Note that because of the rotational symmetry of the law of $x$, $M_\sigma(w,w')$ can depend on $w,w'$ only through their inner product.
In particular, letting $t = 1-\tfrac{a^2}{d}$ and letting $\rho = \langle w,w'\rangle$,
\begin{align*}
M_\sigma(w,w') = M_\sigma(\rho) 
&= t^{-(d-2)/2}\bigg(\left(t(2-\sigma^2) - a^2 (1+\rho)\right)\left(t\sigma^2 + a^2(1-\rho)\right)\bigg)^{-1/2}.
\end{align*}

Since $M_\sigma(w,w')$ depends only on the inner product $\rho = \langle w,w'\rangle$, and because for any unit vector $w$ the law of $\langle v,w \rangle$ when $v \sim \pi$ is independent of $w$, we can further simplify
\begin{align}
\cref{eq:diffs}
&= \Ex_{\substack{\rho_A \sim \pi_A \\ \rho \sim \pi}} M_\sigma(\rho_A) - M_\sigma(\rho), \label{eq:diffs2}
\end{align}
where we abuse notation slightly and denote by $\rho_A \sim \pi_A$ the distribution of the inner product of $u,u' \sim \pi_A$ (and similarly for $\rho \sim \pi$).

We now specialize to the different cases $\sigma = 1,\eta$, as we will require a different bound in each case.
Considering first the case $\sigma = 1$, 
\begin{align*}
M_1(\rho) 
&= c_0\left(1 - c_1 a^2 \rho + c_2 (a^2 \rho)^2\right)^{-1/2}
\end{align*}
where we can apply the fact that $t = 1 - a^2/d$ and the fact that $a = o(1)$ to conclude that $c_0,c_2 = 1 \pm o(1)$ and $c_1 = 2 \pm o(1)$.
Thinking of the coefficients $c_0,c_2$ as fixed quantities of magnitude at most 1.01 and $c_1$ as a fixed quantity of magnitude at most $2.01$, because the quantity $|a^2 \rho| < \frac{1}{2}$ by assumption, the Taylor series of $M_1(\rho)$ in $a^2 \rho$ converges:
\[
M_1(\rho) = c_0\left(1 + \frac{c_1}{2} a^2 \rho + R_1(\rho)\right)
\]
where $|R_1(\rho)| \le (a^2\rho)^2$.

Hence, applying the above and \cref{eq:diffs2} to bound the integral of the $N(0,1)$-noise component,
\begin{align}
\int \int \frac{\left(g_A(x,z) - g_\pi(x,z)\right)^2}{\gamma^d(x) \gamma_1(z) e^{-\frac{a^2}{2d}\|x\|^2}} dx dz
&= \Ex_{\substack{\rho_A \sim \pi_A\\ \rho\sim\pi}} M_1(\rho_A) - M_1(\rho)\nonumber\\
&\le \frac{c_0 c_1 a^2}{2} \left|\Ex_{\substack{\rho_A \sim \pi_A\\ \rho \sim \pi}} \rho_A - \rho \right| + c_0\Ex_{\substack{\rho_A \sim \pi_A \\ \rho\sim\pi}} |R_1(\rho)| + |R_1(\rho_A)|\nonumber\\
&\le a^2 \left|\Ex_{\substack{\rho_A \sim \pi_A\\ \rho \sim \pi}} \rho_A - \rho \right| + 2\Ex_{\substack{\rho_A \sim \pi_A \\ \rho\sim\pi}} a^4 (\rho^2 + \rho_A^2).\label{eq:gbound}
\end{align}

When $\sigma = \eta$, we take the parameterization
\begin{align*}
M_\eta(\rho) 
&= \frac{b_0}{a}\left(1 - b_1 \rho + b_2 a^2 \rho^2\right)^{-1/2} 
\end{align*}
where $b_0,b_2 = 2^{-1/2} \pm o(1)$, $b_1=1 \pm o(1)$, and all three are independent of $\rho$.
Here, we take a Taylor expansion in $\rho$:
\[
M_\eta(\rho) = \frac{b_0}{a}\left( 1 +\frac{b_1}{2}\rho + R_\eta(\rho) \right),
\]
where for all $|\rho| \le \frac{1}{2}$, $|R_{\eta}(\rho)| \le \rho^2$.
We also require a bound on $M_\eta$ for all $\rho$ which we will apply when $|\rho| \ge \frac{1}{2}$: 
\[
M_\eta(\rho) \le \frac{2}{a}\frac{1}{\sqrt{1-\rho}},
\]
which can be verified by inspection of the formula for $M_\sigma$ (the first two factors are bounded by $2$ when $\sigma = \eta$, the last term is smaller than $(a^2(1-\rho))^{-1/2}$).

Applying the above in \cref{eq:diffs2} to bound the integral of the $N(0,\eta^2)$-noise component,
\begin{align}
\int \int \frac{\left(h_A(x,z) - h_\pi(x,z)\right)^2}{\gamma^d(x) \gamma_1(z) e^{-\frac{a^2}{2d}\|x\|^2}} dx dz
&= \Ex_{\substack{\rho_A \sim \pi_A\\ \rho\sim\pi}} M_\eta(\rho_A) - M_\eta(\rho)\nonumber\\
&\le \frac{b_0 b_1}{2a} \left|\Ex_{\substack{\rho_A \sim \pi_A\\ \rho \sim \pi}} \rho_A - \rho \right| + \frac{b_0}{a}\Ex_{\substack{\rho_A \sim \pi_A \\ \rho\sim\pi}} |R_\eta(\rho)| + |R_\eta(\rho_A)|.\label{eq:partway-h}
\end{align}
We will bound the remainders by separately considering the event $E(\rho) = \{|\rho| \le \frac{1}{2}\}$ and its complement.
Treating $\rho$ (the logic is identical for $\rho_A$),
\begin{align*}
\Ex |R_\eta(\rho)|
&= \Ex |R_\eta(\rho)|\cdot \Ind_{E(\rho)} + |R_\eta(\rho)|\cdot \Ind_{\overline{E}(\rho)}\\
&\le \Ex \rho^2 + \Ex_\rho\left[\left( \frac{a}{b_0}|M_\eta(\rho)| + \left|1 + \frac{b_1}{2}\rho\right|\right)\Ind_{\overline{E}(\rho)}\right]\\
&\le \Ex \rho^2 + \Ex_\rho\left[\left( \frac{2}{b_0}(1-\rho)^{-1/2} + 2\right)\Ind_{\overline{E}(\rho)}\right]\\
&\le \Ex \rho^2 + \Ex_\rho\left[\left( 10(1-\rho)^{-1/2}\right)\Ind_{\overline{E}(\rho)}\right].
\end{align*}
So we have
\begin{align}
    \cref{eq:partway-h} 
    &\le \frac{1}{a} \left|\Ex \rho_A -\rho\right| +\frac{2}{a} \Ex(\rho_A^2 + \rho^2) + \frac{11}{a}\left(\Ex \frac{\Ind_{\overline{E}(\rho)}}{\sqrt{1-\rho}} + \frac{\Ind_{\overline{E}(\rho_A)}}{\sqrt{1-\rho_A}}\right).\label{eq:hbound}
\end{align}
So to finish up given \cref{eq:gbound,eq:hbound}, we require a bound on the difference in the first moment of $\rho,\rho_A$, an upper bound on the second moments, and control on the expectation under the tail event $\overline{E}$.

To control the moments, we invoke Lemma 4.4 of \cite{LMSY}, where it is shown that the optimal transport coupling of $\pi_A,\pi$ provides a way to jointly sample $u \sim \pi_A, v\sim \pi$ such that $\Delta = u-v$ satisfies, for all $t > 0$, $\Pr(\|\Delta\| > t + \sqrt{-2\log \pi(A)/d}) \le \exp\left(-\frac{d-1}{8}t^2\right)$.
In particular, letting $(v,v+\Delta)$ and $(v',v'+\Delta')$ be two independent samples from the optimal transport coupling of $(\pi,\pi_A)$,
\[
\Ex \rho_A - \rho
= \Ex \langle v + \Delta, v' + \Delta' \rangle - \langle v,v'\rangle
= \Ex \langle \Delta,\Delta'\rangle 
\le \left(\Ex \|\Delta\|\right)^2 \le C \frac{-\log\pi(A)}{d},
\]
where we have used that $\E[v] = 0$, Cauchy-Schwarz, the independence of $\Delta,\Delta'$, and the fact that $\|\Delta\| - \sqrt{-\log \pi(A)/d}$ is subgaussian with variance proxy $O(1/d)$.

The same coupling allows us to upper bound the second moments. 
First, we directly compute $\Ex \rho^2 = \Ex (v')^\top(vv^\top)v' = \frac{1}{d}$.
We further have that
\[
\Ex \rho_A^2 =\Ex \langle v + \Delta, v' + \Delta'\rangle^2 
\le \Ex 4\langle v,v'\rangle^2 + 8\langle v,\Delta'\rangle^2 + 4\langle \Delta,\Delta'\rangle^2
\le \frac{4}{d} + \frac{8}{d} \E[\|\Delta'\|^2] + 4\Ex[\|\Delta\|^2]^2
\le \frac{C}{d} + \frac{C(\log \pi(A))^2}{d^2},
\]
where we have again used Cauchy-Schwarz and the subgaussianity of $\|\Delta\|$.

Finally, we note that $(1-\rho)^{-1/2}\Ind_{|\rho|\ge \frac{1}{2}} \ge 0$ for $|\rho| \le 1$, and for any non-negative function $f(\rho)$,
\[
\Ex f(\rho_A) = \frac{\Ex_{\substack{v,v' \sim \pi}} f(\langle v,v'\rangle) \Ind_{v,v' \in A}}{\Pr[v,v' \in A]} \le \frac{1}{\pi(A)^2}\E[f(\rho)].
\]
And using that the density of $\rho$ under $\pi$ is $\pi(\rho) = C_d (1-\rho^2)^{(d-3)/2}$ for $C_d = \frac{\Gamma(d/2)}{\sqrt{\pi}\Gamma((d-1)/2)}$,
\begin{align*}
\Ex_\pi \frac{\Ind_{|\rho|> \frac{1}{2}}}{\sqrt{1-\rho}}
&= \int_{-1}^1 \frac{\Ind_{|\rho|> \frac{1}{2}} }{\sqrt{1-\rho}} \cdot C_d\cdot  \left(1-\rho^2\right)^{(d-3)/2} d\rho\\
&= \int_{-1}^1 \Ind_{|\rho|> \frac{1}{2}} C_d \cdot (1+\rho)^{1/2} \cdot \left(1-\rho^2\right)^{(d-4)/2} d\rho\\
&\le \int_{-1}^1 \Ind_{|\rho|> \frac{1}{2}} C_d \cdot \sqrt{2}\cdot \left(1-\rho^2\right)^{(d-4)/2} d\rho\\
&= \sqrt{2} \frac{C_d}{C_{d-1}} \Pr_{w,w' \sim \mathrm{Unif}(\mathbb{S}^{d-2})}[|\langle w,w' \rangle| \ge \tfrac{1}{2}]\\
&\le C \exp(-c' d ),
\end{align*}
where the last line follows for all $d$ sufficiently large by the $O(1/d)$-subgaussianity of the inner product $\langle w,w' \rangle$ and because $C_d/C_{d-1} = 1+o_d(1)$.

We can now use these estimates to complete our bound from \cref{eq:leftoff} using \cref{eq:hbound,eq:gbound} and the assumption that $1/e \ge \pi(A) \ge e^{-c d}$ for $c < c'/2$:
\begin{align*}
    \cref{eq:leftoff}
    &\le C\left(\tfrac{m}{n}\right)^2\cdot \left(\frac{\log^2\pi(A)}{a d} + \frac{e^{-c'd}}{a\pi(A)^2} \right) + C \left(a^2 \frac{-\log\pi(A)}{d} + a^4 \frac{\log^2 \pi(A)}{d}\right) \\
    &\le C\log^2\pi(A)\left( \frac{m^2}{ adn^2} + \frac{a^2}{d}\right),
\end{align*}
which completes the proof of \cref{eq:chisq}.
\end{proof}

\subsection{Barrier for convex $M$-estimation}\label{subsec:m-barrier}
We will show that convex $M$-estimation cannot exactly recover  $\beta$ in planted linear regression once $m \ll \sqrt{nd}$.

In particular, let $\rho: \R \rightarrow \R$ be any finite convex function chosen independently of the sample.\footnote{The proof of this lower bound would also work in a setting where e.g. $\rho$ is chosen based on the first half of the samples, and then the estimator takes the optimizer of $\rho$ on the second half of the samples.}
We define the $M$-estimation objective as
\[
L_\rho(b) := \sum_{i=1}^n \rho(y_i - X_i^\top b),
\]
and we say the estimator exactly identifies $\beta$ if it is the unique minimizer, meaning
\[
\argmin_{b \in \R^d} L_\rho(b) = \{\beta \}.
\]
We show that no such $M$-estimator succeeds in exactly identifying $\beta$ when $m \ll \sqrt{nd}$:
\begin{theorem}\label{thm:no-m-est}
    There exists a universal constant $c > 0$ such that for any finite convex loss $\rho$, if $m \le \min\{c \sqrt{nd}, n/2\}$ then
    \[
    \Pp\left(\argmin_{b \in \R^d} L_\rho(b) = \{\beta\} \right) \lesssim  e^{-cd} + e^{-cm}.
    \]
\end{theorem}
\begin{proof}
Our proof will follow by studying the conditions where $0 \in \partial L_\rho (\beta)$.
We leverage that $\rho$ has subderivatives everywhere, and is differentiable almost everywhere. We denote
\[
\partial\rho(0)=[l,r],
    \qquad
    R:=\max\{|l|,|r|\}.
\]
For $i \notin S$, it will be helpful to define
\[
a_i := \rho'(\eps_i),
\]
which is well-defined with probability 1 because $\rho$ is differentiable almost everywhere and $\eps_i$ is a continuous distribution.
We now study $\partial L_\rho(\beta)$; since the residual is $0$ for $i$ in the noiseless set 
$S$, it holds that
\begin{equation}\label{eq:subgradient-objective}
    \partial L_\rho(\beta)
    =
    -\left\{
        \sum_{i\notin S}a_i X_i+\sum_{i\in S}s_i X_i:
        s_i\in[l,r]
      \right\}.
\end{equation}
We separate this into two terms 
\[
Z:=\sum_{i\notin S}a_i X_i,
\qquad
K:=\left\{\sum_{i\in S}s_i X_i:s_i\in[l,r]\right\}.
\]
In this language, we may use the equivalent characterizations
\[
\beta\in\argmin L_\rho
    \quad\Longleftrightarrow\quad
    0\in\partial L_\rho(\beta)
    \quad\Longleftrightarrow\quad
    -Z\in K.
\]
We separate our proof into two cases: (i) when $R > 0$, we show $\beta$ is usually not a minimizer, and (ii) when $R = 0$, the flatness at $0$ will let us show that $\beta$ is usually not a unique minimizer.

\textbf{Case (i): $R>0$.} We will prove that with high probability it holds that both
\begin{align}
    \norm{Z}_2 &\ge cR\sqrt{nd}, \label{eq:score-lower}\\
    \sup_{w\in K}\ip{-Z/\norm{Z}_2}{w} &\le CRm. \label{eq:width-upper}
\end{align}
Together, these imply that if $m \le c_0 \sqrt{nd}$ for small enough $c_0$, then $-Z \notin K$ and hence $\beta \notin \argmin L_\rho$.

\textit{Proof of \eqref{eq:score-lower}.} By monotonicity of $\rho'$, it holds that $|\rho'(t)| \ge R$ either for all $t>0$, or for all $t<0$. Hence, for each $i \notin S$, it holds that $|a_i| \ge R$ with probability at least $1/2$. Using a Chernoff bound, this implies that at least $\frac{n-m}{4} \ge n/8$ samples from $i \notin S$ satisfy $|a_i | \ge R$, with probability at least $1 - e^{-cn}$. Conditioned on this event, we denote
\[
q := \sum_{i \notin S} a_i^2 \ge \frac{nR^2}{8}.
\]
Conditioned on the values of $a_i$, the value of $Z$ is distributed like
\[
Z \mid (a_i)_{i \notin S} \sim N(0,q I_d).
\]
We may use another Chernoff bound to conclude that its norm satisfies
\[
\norm{Z}_2 \gtrsim \sqrt{qd} \gtrsim R \sqrt{nd}
\]
with probability at least $1-e^{-cd}$.

\textit{Proof of \eqref{eq:width-upper}.} Let $v = \frac{Z}{\norm{Z}_2}$ denote the relevant vector in \eqref{eq:width-upper}; observe how $v$ is independent of $(X_i)_{i \in S}$. For every $w = \sum_{i \in S} s_i X_i \in K$, it holds that
\[
\langle -v , w \rangle = \sum_{i \in S} (-s_i) \langle v, X_i \rangle \le R \sum_{i \in S} | \langle v, X_i \rangle|.
\]
Conditioned on $v$, the random variables $|\langle v, X_i \rangle|$ are i.i.d. $|N(0,1)|$, and hence
\[
\sum_{i \in S} | \langle v, X_i \rangle| \lesssim m
\]
with probability at least $1 - e^{-cm}$ by Bernstein's inequality.

In total, this yields the desired guarantee for $R>0$ with probability at least $1 - e^{-cd} - e^{-cm}$.
\textbf{Case (ii): $R=0$.} When $R=0$, it holds that $\partial \rho(0) = \{0\}$, and hence 
\[
\partial L_\rho(\beta) = \{ -Z\},
\]
meaning $\beta$ can only be an optimizer when $Z=0$. If at least one value of $a_i$ is nonzero, then conditioned on the values of $(a_i)_{i \notin S}$, $Z$ is a nondegenerate Gaussian vector and hence $Z \ne 0$ with probability 1 (in which case $\beta$ is not an optimizer). 

Hence, we may focus on the event where $\rho'(\eps_i)=0$ for every $i \notin S$. Consider the region 
\[
    M:=\argmin_{t\in\R}\rho(t),
\]
or equivalently $M= \{t : 0 \in \partial \rho (t) \}$; this is a nonempty interval containing $0$. If no $\eps_i$ is at the boundary of $M$, then $\beta$ is not a unique minimizer since a small perturbation of $\beta$ is also a minimizer. For any $i \notin S$, an $\eps_i$ is realized exactly at the boundary with probability 0. We only need to consider the possibility where $0$ is at the boundary of $M$. In this case, it must hold that either: (i) all $\eps_i < 0$ for $i \notin S$, or (ii) all $\eps_i > 0$ for $i \notin S$. This event occurs with probability at most $e^{-cn}$, thus implying our desired result.
\end{proof}

%% file: simulations.tex
\section{Simulations}\label{sec:simulations}

In this section we discuss our simulations. 
We have modified the estimators given in the theorems in order to make them more practical; we detail these changes, and then discuss the results of our simulation for planted and adaptive linear regression.

\paragraph{Estimator implementations.}
The existing estimators we compare to include OLS, LAD ($L_1$ regression), $L_\infty$ regression, and the ASM (antitonic score matching) estimator of Feng, Kao, Xu, and Samworth \cite{feng2026optimal}. For the ASM estimator, we use version 0.2.4 of their R package with the default parameters except we set $\operatorname{alt\_iter}=1$ (which corresponds to their ASM estimator), and we use the default pilot estimator LAD in all settings other than in \cref{fig:plot2-location} (where the pilot estimator is OLS, since this is done in their analogous experiment).

We now discuss our estimator implementations, which are also available at \url{https://github.com/SpencerCompton/adaptive-regression}.

\subparagraph{Adaptive $L_q$. } Compared to the mathematically-analyzed version of this estimator, there are two main differences. First, we include $q=1$ (LAD) as an option for the estimator. For some theoretical justification, the only concern is whether the candidate selection in \cref{thm:k1} is still valid when including $q=1$. The proof only uses the property that when $\operatorname{rank}(\mathbf{X})=d$, the optimizer of the objective must be unique, and the optimizer's distribution is symmetric around $\beta$. This does not immediately hold for $q=1$ (the optimizer need not be unique), but it is true when we tiebreak among the LAD optimizers by choosing the optimizer whose residuals have the smallest $L_2$ norm; we use this tiebreaking version in our implementation for adaptive $L_q$.
We additionally include the powers of $2$ up to $n$ as options for $q$.
Second, instead of outputting the selected batch estimate $\wh{A}_{\wh q, \wh i}$, we run $L_q$ regression on the whole dataset with this choice of $\wh q$, and output this estimate (this version improves performance, but because $\wh q$ is not independent of the data, our proof no longer applies). 

\subparagraph{Residual-balance descent} ($\operatorname{RB-Desc}$). Our implementation incorporates substantial changes to optimize empirical performance.

First, we discuss some relatively minor changes. We initialize with the OLS estimate instead of the all-zeros vector (this is not precluded by our analysis). Instead of taking in bounds on $B,K,\kappa$, we apply a whitening preprocessing step, and then use a $B$ at the scale of the typical residual magnitude; if the estimator ever leaves the initial $B$-radius, we double $B$ and restart (repeating as necessary). To improve runtime, we reduce the number of iterations per stage from $O(n)$ to $O(\sqrt{n})$. The mathematically-analyzed version chooses $\calW$ as powers of $2$ starting at a constant $w_0$; we instead consider $\calW$ as all the $j$-th smallest absolute value of residuals at the current iteration, for $j$ of the form $(d+1) 2^{i/4} \in [d+1,n]$ for integer $i$ (this is a scale-invariant choice of $\calW$).

The main obstacle for practical performance is that the estimator only uses $S_w(b)$ with norm larger than a threshold, and the empirical performance is very sensitive to minor perturbations to this threshold. If the threshold is too large, the estimator may not leverage important signal; if the threshold is too small, the estimator may descend in erroneous directions. 

We use empirical simulation to choose our thresholding parameter for our ``standard'' implementation of $\operatorname{RB-Desc}$:
At the start of the algorithm, we fit OLS and compute the residuals for all the samples. Our aim is to choose a threshold such that if the $X_i$ were symmetrized independently of the $r_i(b)$ via random signing, then our threshold is larger than all the considered $\norm{S_w(b)}_2$; this is because after random signing, there is no correlation anymore between $r_i(b)$ and $X_i$. This gives a quantitative indicator for the typical scale of $\norm{S_w(b)}_2$ under the ``null hypothesis'' of no correlation. Thus, at the start of the algorithm, our implementation does this Rademacher signing 256 times, and chooses as the threshold the smallest number such that no $\norm{S_w(b)}_2$ exceeds the threshold for at least 95\% of those samplings. Given this threshold constant, we proceed with the iterative estimator normally (and do not change the constant again). This is the threshold-selection approach we generally recommend.

In the case of planted linear regression, the thresholds chosen by this ``standard'' variant are too conservative. In our ``aggressive'' variant, we use the same threshold as the ``standard'' variant to determine when a $w$ is violated, yet we also consider a $w$ as violated whenever it is at the scale of the current stage diameter $2R_\ell$ (or smaller). This is quite aggressive, so it does not throw away signal as much as the ``standard'' variant, yet it may also erroneously move in directions caused by noise (so we do not recommend this in general).

We also include a ``hybrid'' variant of our estimator, which is designed with planted linear regression in mind. This first runs the ``standard'' variant, then checks if the $d+1$ smallest residuals are all zero (up to numerical tolerance), and uses this estimate if so. Otherwise, it tries the ``aggressive'' variant, and uses its output if its $d+1$ smallest residuals are zero; else, we use the standard variant estimate. This corresponds to using the aggressive strategy only when it clearly attains exact recovery, and the standard strategy does not.

We emphasize there are significant modifications in these implementations, and that many ablations were run to choose and understand what modifications seemed conducive for practical performance. A natural hope for future work is to design other estimators with the same desirable guarantees, yet with fewer degrees of freedom that require tuning for practical performance.

\subsection{Planted linear regression.}\label{subsec:planted-sims}

Our simulations evaluate the empirical likelihood of exactly recovering $\beta$ for planted linear regression (which is the definitive criterion for the estimator's success in this case; for completeness we include figures showing the median coefficient error and average runtime in \cref{sec:extra-sims}). The model is as described in \cref{def:planted}, and we sample $\beta$ from the unit sphere. Each data point corresponds to 100 repetitions. We evaluate the three versions of our $\operatorname{RB-Desc}$ estimator, with comparisons to LAD and OLS. Recall that our theoretical guarantees show that a version of our $\operatorname{RB-Desc}$ estimator attains exact recovery when $m \gg d^{3/4}n^{1/4}$, and that no convex $M$-estimator (including LAD and OLS) may attain exact recovery when $m \ll \sqrt{dn}$. We accordingly observe a notable gap between what is possible with the best, practical versions of our $\operatorname{RB-Desc}$ estimators, and what is achieved by LAD and OLS.

\begin{figure}[H]
\centering
\begin{subfigure}{.32\linewidth}
    \centering
    \includegraphics[width=\linewidth]{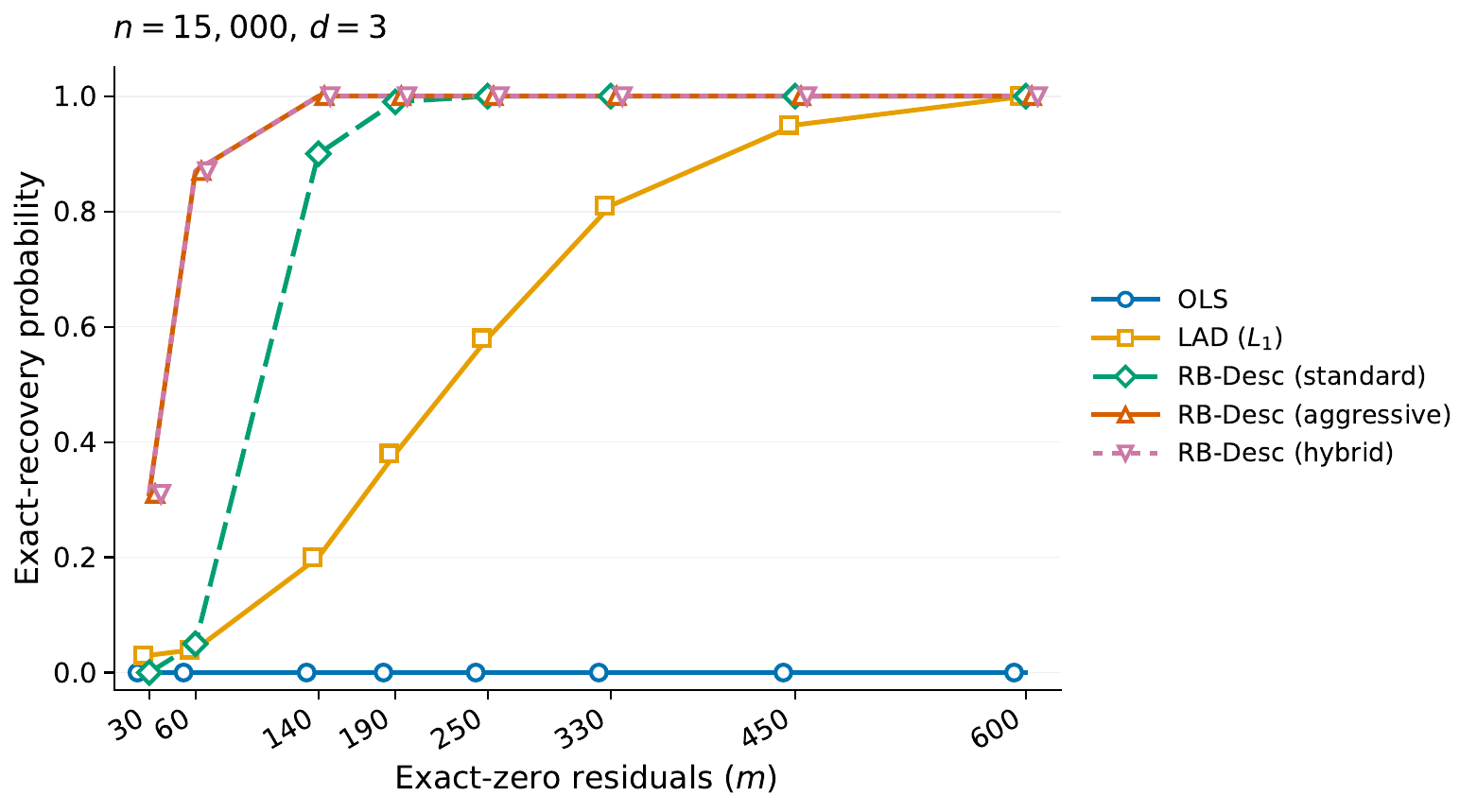}
    \caption{}
\end{subfigure}
    \hfill
\begin{subfigure}{.32\linewidth}
    \centering
    \includegraphics[width=\linewidth]{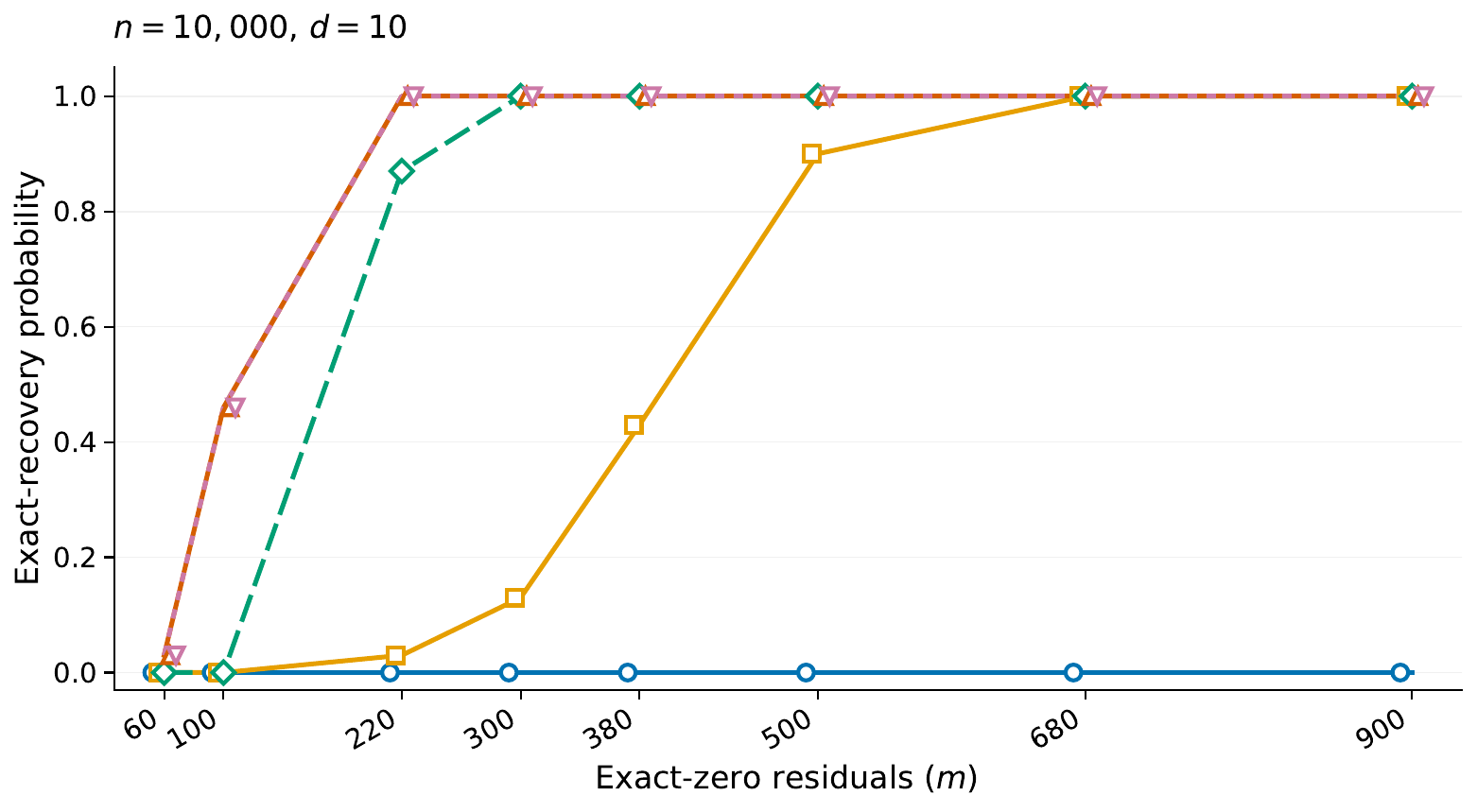}
    \caption{}
\end{subfigure}
   \hfill
\begin{subfigure}{.32\linewidth}
    \centering
    \includegraphics[width=\linewidth]{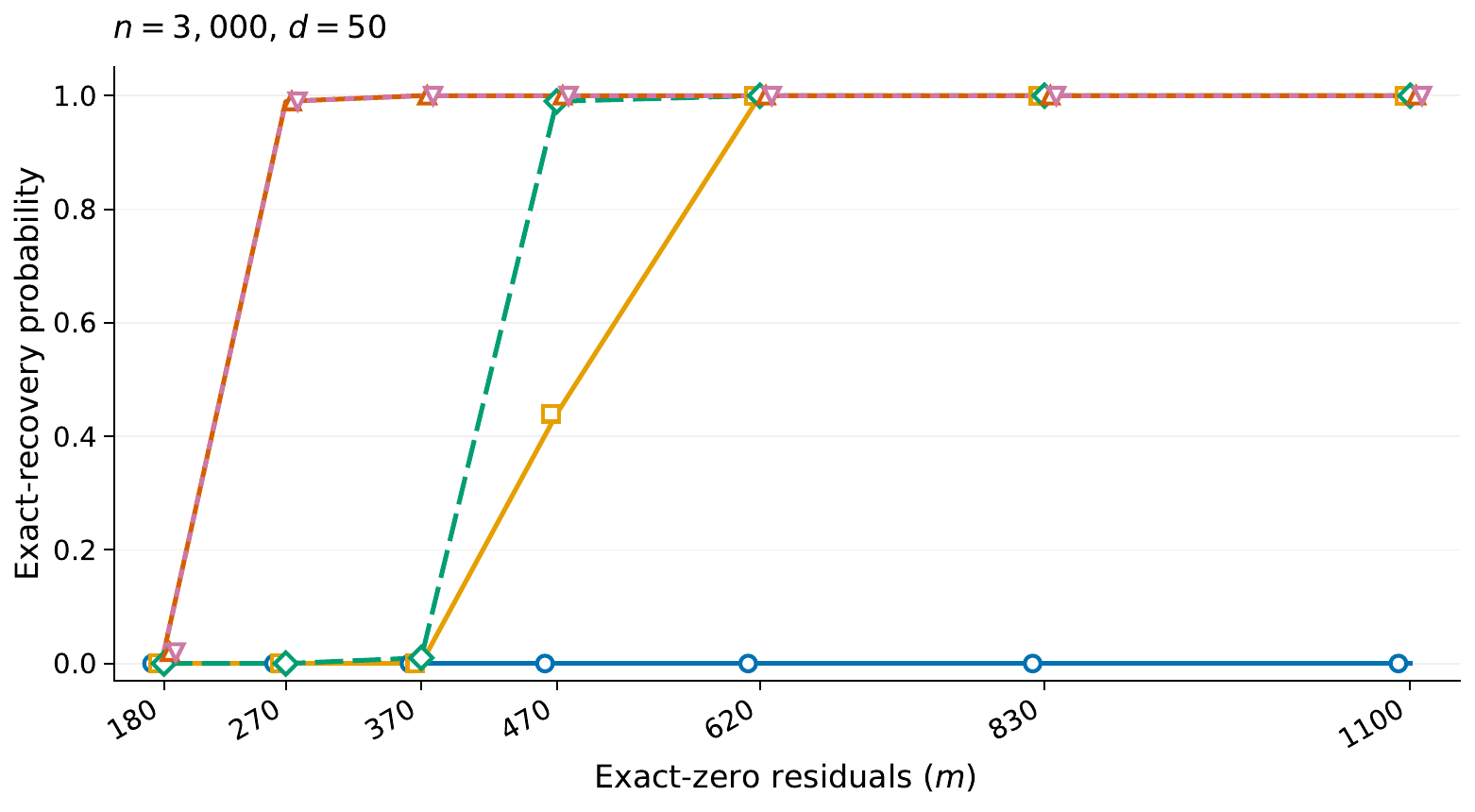}
    \caption{}
\end{subfigure}

{\caption{Empirical probability of recovery within $\norm{\wh\beta-\beta}_2 \le 10^{-5}$ (for 100 repetitions) for planted linear regression.}
\label{fig:planted-recovery}}
\end{figure}

The aggressive variant of $\operatorname{RB-Desc}$ performs even better than the standard variant. As discussed in our implementation details, this is because the aggressive variant considers more scales of $w$ violated (which may erroneously move in bad directions in general, but is helpful for stronger performance in the case of planted regression). While it is somewhat undesirable that these practical implementations of $\operatorname{RB-Desc}$ required substantial changes, we still believe this is an interesting demonstration of significant improvement relative to classical estimators. 

Lastly, we remark that LAD achieves exact recovery when $m$ is moderately large (consistent with the $m \gg \sqrt{nd}$ exact recovery guarantee given by \cite[Theorem 3.2]{gao2020model}), but OLS does not achieve exact recovery in these regimes. For some intuition, consider a toy 1-dimensional setting where $m$ samples are exactly $0$, and $n-m$ samples are $N(0,1)$. The empirical median (analogous to LAD) will output $0$ when $m \gg \sqrt{n}$, yet the empirical mean (analogous to OLS) will not output $0$ until $m = n$.

\subsection{Adaptive linear regression.}

We analyze the performance of OLS, LAD, $L_\infty$ regression, ASM, adaptive $L_q$, and $\operatorname{RB-Desc}$ on 5 symmetric noise distributions. The first coordinate of each $X_i$ is 1 (representing an intercept term), and the remaining coordinates are sampled from $N(0,I_{d-1})$. $\beta$ is sampled from the unit sphere. Each data point corresponds to 100 repetitions. 
For context, the average running time for Gaussian $N(0,1)$ errors with $d=10$ and $n=10{,}000$ was 0.00158 seconds for OLS, 0.526 for LAD, 0.455 for $L_\infty$,  10.3 for ASM, 3.67 for adaptive $L_q$, and 0.214 for $\operatorname{RB-Desc}$ (standard). 
The runtime is mostly similar for the other errors, except for in the location mixture simulation, the average runtime of $\operatorname{RB-Desc}$ is significantly larger (189.3 seconds).

We first explain why one should only expect $\operatorname{RB-Desc}$ and adaptive $L_q$ to have strong performance for some of these distributions. 

\begin{figure}[H]
\centering
\setcounter{subfigure}{0}

\begin{tabular}{@{}c@{\hspace{0.015\linewidth}}c@{}}

\makebox[0.48\linewidth][l]{%
\begin{subfigure}[t]{0.63\linewidth}
    \centering
    \includegraphics[width=\linewidth]{figs/adaptive/gaussian-error.pdf}
    \caption{}\label{fig:plot2-gaussian}
\end{subfigure}
}
&
\\[0.1cm]

\begin{subfigure}[t]{0.48\linewidth}
    \centering
    \includegraphics[width=\linewidth]{figs/adaptive/uniform-error.pdf}
    \caption{}\label{fig:plot2-unif}
\end{subfigure}
&
\setcounter{subfigure}{3}
\begin{subfigure}[t]{0.48\linewidth}
    \centering
    \includegraphics[width=\linewidth]{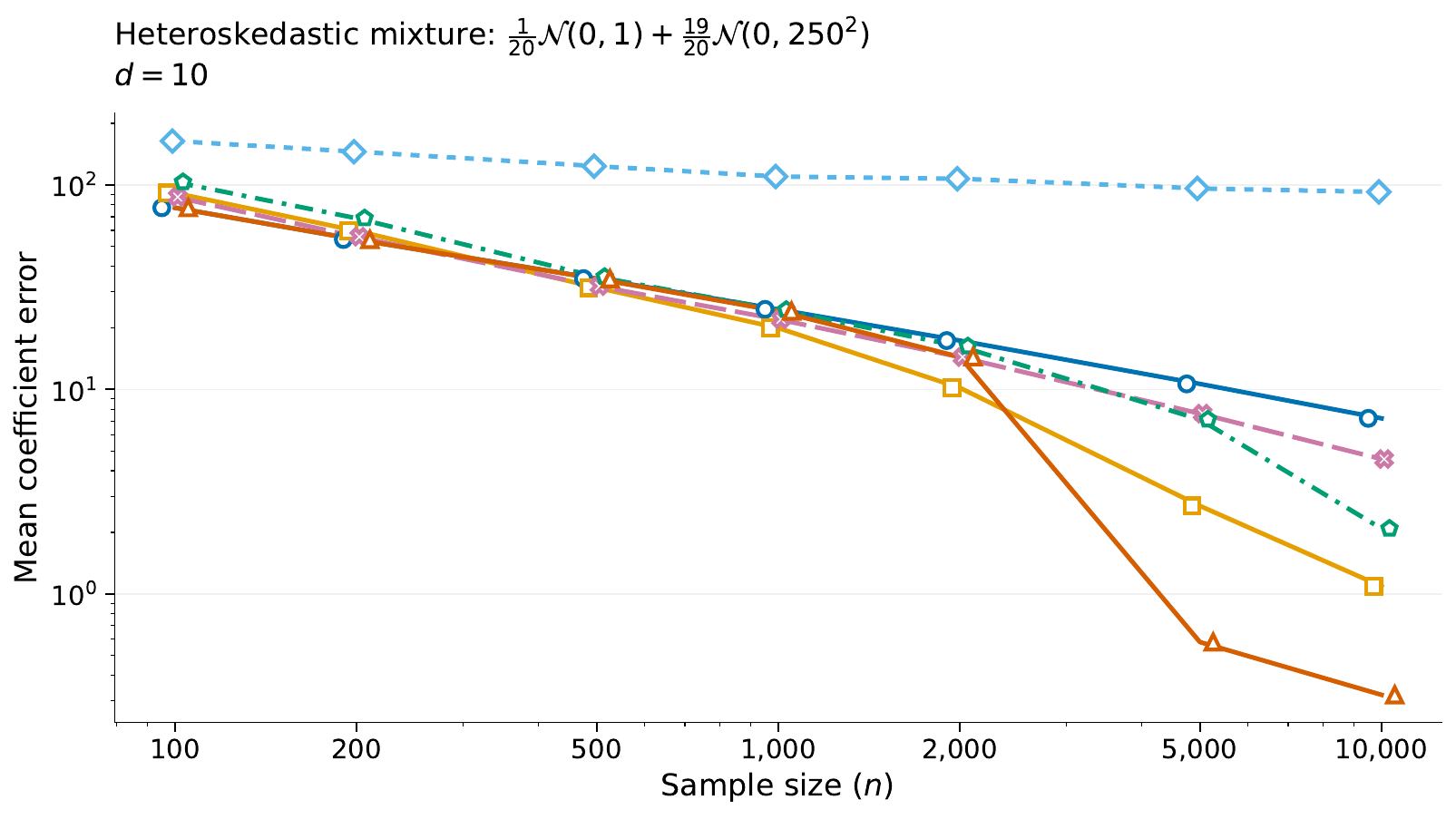}
    \caption{}\label{fig:plot2-heteroskedastic}
\end{subfigure}
\\[0.1cm]

\setcounter{subfigure}{2}
\begin{subfigure}[t]{0.48\linewidth}
    \centering
    \includegraphics[width=\linewidth]{figs/adaptive/smoothed-uniform-error.pdf}
    \caption{}\label{fig:plot2-smoothed-unif}
\end{subfigure}
&
\setcounter{subfigure}{4}
\begin{subfigure}[t]{0.48\linewidth}
    \centering
    \includegraphics[width=\linewidth]{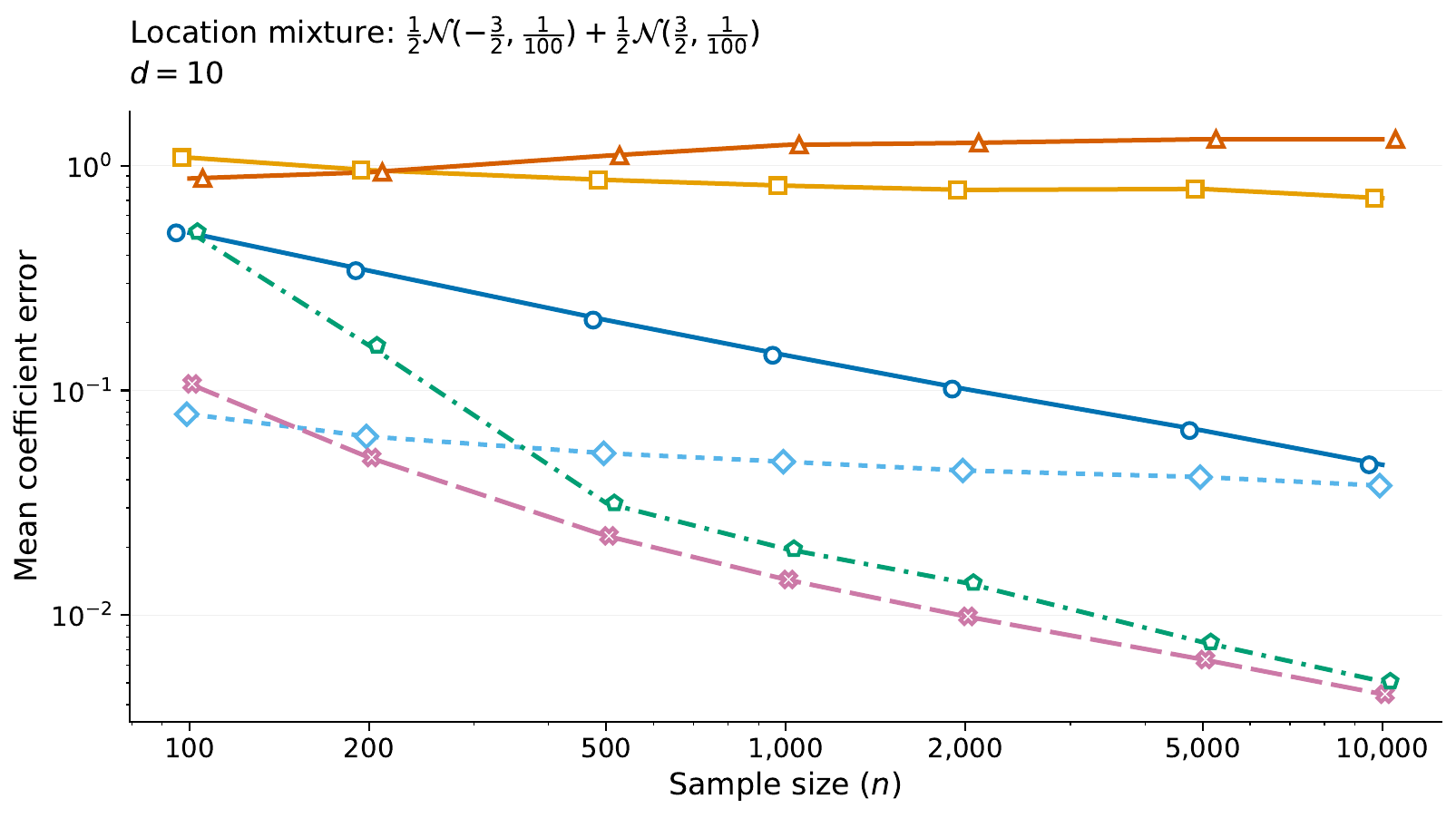}
    \caption{}\label{fig:plot2-location}
\end{subfigure}

\end{tabular}

\caption{Average $\norm{\wh\beta-\beta}_2$ error for adaptive linear regression.}
\label{fig:adaptive-plots-2}
\end{figure}

$\operatorname{RB-Desc}$ is designed for heteroskedastic estimation, and it mostly extracts signal by trying to leverage small ``bumps'' in the noise density near 0. It achieves nearly the best performance (like many estimators) for standard Gaussian noise (\cref{fig:plot2-gaussian}), and achieves the best performance for the two-component heteroskedastic mixture (\cref{fig:plot2-heteroskedastic}), where it has a big jump in performance once it has enough samples to detect the small-variance component. $\operatorname{RB-Desc}$ has suboptimal performance for uniform and smoothed uniform noise (\cref{fig:plot2-unif,fig:plot2-smoothed-unif}) where important signal is in the tails, since $\operatorname{RB-Desc}$ aims to leverage interior bumps and not tail information. Its performance for the location mixture (\cref{fig:plot2-location}) is predictably poor, since the noise distribution is not unimodal (and $\operatorname{RB-Desc}$ may unfortunately be trying to leverage bumps centered far away from 0).

On the other hand, adaptive $L_q$ is designed to leverage information in the tails. This is provably sufficient for near-optimal estimation with symmetric, log-concave noise (see \cref{thm:k1}), and this matches our empirical observations for the three such noise distributions (\cref{fig:plot2-gaussian,fig:plot2-unif,fig:plot2-smoothed-unif}). Adaptive $L_q$ achieves nearly the best performance (like many estimators) for standard Gaussian noise (\cref{fig:plot2-gaussian}); achieves strong performance for uniform noise (\cref{fig:plot2-unif}), where only $L_\infty$ has better performance (which is an MLE for uniform errors); and achieves the best performance for smoothed uniform noise (\cref{fig:plot2-smoothed-unif}), where our discussion in the introduction already indicated that OLS and $L_\infty$ should have suboptimal guarantees. Since adaptive $L_q$ leverages information in the tails, it is less able to leverage information in the interior of the noise distribution (this is also discussed in \cite{kao2024choosing}), and hence achieves middling performance on the heteroskedastic mixture (\cref{fig:plot2-heteroskedastic}). For the location mixture (\cref{fig:plot2-location}), the performance is only worse than the ASM estimator; we remark that the plot shows their errors approaching each other for large $n$, yet we expect that ASM might have superior performance in an even larger regime of $n$ if the variance in the mixture components was tuned differently.

In terms of general performance, the adaptive $L_q$ and ASM estimators seem to generally have the most versatile performance over these five simulations. We suspect the ASM estimator of \cite{feng2026optimal} may have comparatively stronger performance for asymmetric noise, and its strong performance in \cref{fig:plot2-location} gives an example of a type of signal that neither of our implemented estimators can leverage even with symmetric noise. More broadly, there is no single estimator that simultaneously achieved the best performance for all five noise distributions. We suspect that some version of our inefficient estimator (\cref{thm:main-adapt}) could leverage the signal in all of these distributions, yet we do not have a reasonable implementation of this estimator even for extremely small $d$ (the estimator requires a search over a prohibitively fine net). Designing a natural, practical estimator that can leverage the signal in all such distributions would be an interesting line of work.

Some of these noise distributions have been studied with similar simulation setups in previously-discussed related work. \cite{kao2024choosing} have linear regression simulations for Gaussian and uniform errors. \cite{feng2026optimal} have linear regression simulations corresponding to all five of these styles of noise distributions other than uniform noise (the location mixture we use is exactly their location mixture). \cite{compton2026attainability} have 1-dimensional mean estimation simulations corresponding to all five of these styles of noise distributions other than the location mixture; they also observe an analogous jump in improvement for heteroskedastic mixtures when $n$ becomes large enough to detect the small-variance component.

%% file: simple-gaussian.tex
\section{Simple heteroskedastic linear regression for standard Gaussian covariates}\label{sec:simple-gaussian}
The task of heteroskedastic regression is much simpler when the covariates are drawn according to $X_i \sim N(0,I_d)$ (and the errors are $\eps_i \sim N(0,\sigma_i^2)$). In this restricted setting, we observe that the reduction from linear regression to mean estimation used in the work of d’Orsi, Novikov, and Steurer   \cite[Section 4.2.1]{d2021consistent} is amenable to recent algorithms for heteroskedastic mean estimation. We will give brief intuition on how this works, but omit a rigorous analysis since this will be subsumed by our \cref{thm:main-het} for more general covariates; this is a standalone section and other parts of the paper do not use these techniques.

Consider a candidate $b \in \R^d$, and the difference $h = \beta - b$. Let us focus on the task of estimating $h_1$. In \cite{d2021consistent}, they observe how the following manipulation helps reduce the task to mean estimation:
\begin{equation}
    \frac{r_i(b)}{X_i(1)} = h_1 + \frac{\eps_i + \sum_{j=2}^d h_j \cdot X_i(j)}{X_i(1)}
\end{equation}
Moreover, conditioned on the value of $X_i(1)$, the distribution is exactly the Gaussian
\begin{equation}
    \frac{r_i(b)}{X_i(1)} \stackrel{d}= N\left(h_1, \frac{\sigma_i^2 + \norm{h_{-1}}_2^2}{X_i(1)^2}\right),
\end{equation}
where $h_{-1}$ denotes the vector without the first coordinate. Thus, if we condition on the values of $X_i(1)$, and only use the samples where $|X_i(1)| \ge 1$ (this is a constant fraction), then we immediately have a mean estimation problem with new variances bounded by $\sigma_i^2 + \norm{h_{-1}}_2^2$.

We now roughly outline how to get a Subset-of-Signals guarantee with this approach when  $m \gg d^{3/4}n^{1/4}$. Consider two regimes: (i) $\norm{h_{-1}}_2^2 > 1$ and (ii) $\norm{h_{-1}}_2^2 \le 1$.

For regime (i), using existing Subset-of-Signals heteroskedastic mean estimation guarantees (such as in \cite{compton2024near}) implies that we learn $h_1$ with error $\ll \norm{h_{-1}}_2 / \sqrt{d}$. If we simultaneously estimate every coordinate of $h$ in this manner, and adjust our estimate of $\beta$ accordingly by adding the estimated means, then the norm of the new $h$ will be at most, say, half of what it was originally. This is the same paradigm used in \cite{d2021consistent}, where all coordinates are estimated and improved iteratively. We continue doing so until we are in regime (ii).

In regime (ii), we now have $\approx m$ samples with variance bounded by a constant, and can run one last round of mean estimation with the existing Subset-of-Signals guarantees again; this will yield error $\approx (n/m^4)^{1/6}$ for each coordinate, implying total error
\[
\approx \sqrt{d} \left( \frac{n}{m^4} \right)^{1/6},
\]
which is the desired guarantee. Hence, the problem follows rather simply from mean estimation in the setting where covariates are sampled from a standard Gaussian. However, this approach does not seem conducive to handling more general covariates, and thus in \cref{sub:general-het} we design a new algorithm for heteroskedastic regression.

%% file: subset-signals-lb.tex
\section{Subset-of-Signals statistical lower bound}\label{sec:subset-lb}
\begin{theorem}[Subset-of-Signals lower bound]\label{thm:subset-lb}
There exist universal constants
$c,c_0>0$ such that the following holds.
Suppose $1 \le d,m \le n$. For every estimator $\widehat\beta=\widehat\beta\left((X_i,Y_i)_{i=1}^n\right)$ that is not given the standard-deviation vector, there exists a
$\boldsymbol\sigma=(\sigma_1,\ldots,\sigma_n)$ with at least $m$ values of $\sigma_i = 1$, and a $\beta \in \R^d$ such that, under 
\[
    X_i \sim N(0,I_d),
    \qquad
    Y_i=X_i^\top\beta+\varepsilon_i,
    \qquad
    \varepsilon_i\ \sim N(0,\sigma_i^2),
\]
it holds that
\[
\mathbb P
\left(
    \|\widehat\beta-\beta\|_2
    \ge c\,r_{n,d,m}
\right)
\ge \frac14,
\]
where
\[
r_{n,d,m}:=
\begin{cases}
\displaystyle
\frac{\sqrt n\,d^{3/2}}{m^2},
&
m<c_0 d^{3/4} n^{1/4},
\\[1.2ex]
\displaystyle
\sqrt{\frac dn}\left(\frac nm\right)^{2/3},
&
m\ge c_0 d^{3/4} n^{1/4}.
\end{cases}
\]
\end{theorem}
\begin{proof}
This proof will follow by invoking the general Hellinger modulus lower bound of \cref{thm:main-hel-lb}, and then analyzing the Hellinger modulus for a 1-dimensional mixture with two variances (the latter is mostly calculations, and is not conceptually novel beyond related analyses in \cite{liang2020learning,han2026empirical}).

We first analyze the Hellinger distance for translations of such mixtures. 
Let
\[
\varphi_\tau(x):=\frac1{\sqrt{2\pi}\tau}
\exp\!\left(-\frac{x^2}{2\tau^2}\right)
\]
be the density of $N(0,\tau^2)$, and define
\[
p_{q,\sigma}:=q\varphi_1+(1-q)\varphi_\sigma.
\]
This density is symmetric and unimodal. We bound the distance between translations:

\begin{lemma}[Translation bound]\label{lem:mixture}
For
\(0<q\le1/2\), \(\sigma\ge2\), and \(\abs{t}\le\sigma\),
\[
H_{p_{q,\sigma}}(t)
\lesssim
\left[
q^2\sigma\min\{t^2,1\}
+\frac{t^2}{\sigma^2}
\right].
\]
\end{lemma}

\begin{proof}
We use
\[
\hels(f,g)\lesssim \chi^2(g\|f)
=\int\frac{(g(x)-f(x))^2}{f(x)}.
\]
Since \(p_{q,\sigma}\ge(1-q)\varphi_\sigma\ge\varphi_\sigma/2\),
\[
H_{p_{q,\sigma}}(t)\lesssim \int \frac{\left(q\left(\varphi_1(x-t) - \varphi_1(x) \right) + \left(\varphi_\sigma(x-t) - \varphi_\sigma(x) \right)\right)^2}{\varphi_\sigma(x)} \lesssim q^2A_\sigma(t)+B_\sigma(t),
\]
where
\[
A_\sigma(t):=
\int\frac{(\varphi_1(x-t)-\varphi_1(x))^2}{\varphi_\sigma(x)}\,dx,
\qquad
B_\sigma(t):=
\int\frac{(\varphi_\sigma(x-t)-\varphi_\sigma(x))^2}{\varphi_\sigma(x)}\,dx.
\]
The value of $B_\sigma(t)$ has the closed-form
\[
B_\sigma(t)
=
\chi^2\!\left(N(t,\sigma^2)\,\middle\|\,N(0,\sigma^2)\right)
=e^{t^2/\sigma^2}-1
\lesssim \frac{t^2}{\sigma^2},
\]
using \(\abs t\le\sigma\).
For the $A_\sigma(t)$ term, it will be helpful to write \(z=x-u\) and bound
\begin{align*}
\frac{\varphi_1(x-u)^2}{\varphi_\sigma(x)}
&=
\frac{\sigma}{\sqrt{2\pi}}
\exp\!\left(-z^2+\frac{(z+u)^2}{2\sigma^2}\right) \\
&\le
\frac{\sigma}{\sqrt{2\pi}}
\exp\!\left(-z^2+\frac{z^2}{\sigma^2}
+\frac{u^2}{\sigma^2}\right) \\
& \lesssim \sigma e^{-3z^2/4}. &&\rtag{via \(\sigma\ge2\) and \(|u|\le\sigma\)} 
\end{align*}
Consequently, for \(|u|\le\sigma\),
\begin{equation}\label{eq:J-I-bounds}
\int\frac{\varphi_1(x-u)^2}{\varphi_\sigma(x)}\,dx\lesssim \sigma,
\qquad
\int\frac{\varphi_1'(x-u)^2}{\varphi_\sigma(x)}\,dx\lesssim \sigma,
\end{equation}
because \(\varphi_1'(x-u)=-(x-u)\varphi_1(x-u)=-z\varphi_1(z)\), and both
\(\int e^{-3z^2/4}dz\) and \(\int z^2e^{-3z^2/4}dz\) are finite
constants.

If \(|t|\le1\), using Cauchy-Schwarz gives,
for each \(x\),
\[
\bigl(\varphi_1(x-t)-\varphi_1(x)\bigr)^2
\le
|t|\int_0^{|t|}
\varphi_1'\!\left(x-\operatorname{sgn}(t)u\right)^2\,du.
\]
Divide by \(\varphi_\sigma(x)\), integrate in \(x\), and use
\eqref{eq:J-I-bounds}:
\[
A_\sigma(t)
\le
|t|\int_0^{|t|}
\left[
\int
\frac{\varphi_1'(x-\operatorname{sgn}(t)u)^2}{\varphi_\sigma(x)}\,dx
\right]du
\lesssim  \sigma t^2.
\]
Otherwise, when $1 < |t|\le\sigma$, then \((a-b)^2\le2a^2+2b^2\) and
\eqref{eq:J-I-bounds} give
\[
A_\sigma(t)
\le
2\int\frac{\varphi_1(x-t)^2}{\varphi_\sigma(x)}\,dx
+
2\int\frac{\varphi_1(x)^2}{\varphi_\sigma(x)}\,dx
\lesssim \sigma.
\]
Combining the two cases yields
\[
A_\sigma(t)\le C\sigma\min\{t^2,1\}.
\]
This gives the desired guarantee.
\end{proof}

This enables the Hellinger modulus bound:
\begin{lemma}[Piecewise Hellinger modulus bound]\label{lem:modulus}
There exists a universal constant $C>0$ such that, for every
\(0<a<1\) and \(0<q\le1/2\), one can choose \(\sigma\ge2\) for which
\[
\omega_{p_{q,\sigma}}(Ca)\ge 
\begin{cases}
2a^{3/2}q^{-2},&q\le a^{3/4},\\[0.5ex]
a^{1/2}q^{-2/3},&q\ge a^{3/4}.
\end{cases}
\]
\end{lemma}

\begin{proof}
\noindent \paragraph{Case (i) $q\ge a^{3/4}$.}  Choose
\[
\sigma=2q^{-2/3},
\qquad
t=a^{1/2}q^{-2/3}.
\]
This satisfies $\sigma \ge 2$ and $t \le 1$; invoking \cref{lem:mixture} yields
\[
H_{p_{q,\sigma}}(t)
\lesssim  \left(q^2\sigma t^2+t^2/\sigma^2\right)
\le Ca.
\]

\paragraph{Case (ii): $q\le a^{3/4}$.}  Choose
\[
\sigma=2aq^{-2},
\qquad
t=\sigma\sqrt a
= 2a^{3/2}q^{-2}.
\]
This satisfies $\sigma \ge 2$ and $t \le \sigma$; invoking \cref{lem:mixture} yields
\begin{equation*}
H_{p_{q,\sigma}}(t)
\lesssim \left(q^2\sigma\min\{t^2,1\}+t^2/\sigma^2\right)
\lesssim a.\qedhere
\end{equation*}
\end{proof}

\paragraph{Concluding the lower bound.} 
Let $\kappa \ge 1$ be a sufficiently large universal constant.
We first focus on the case where \(m\le n/(2\kappa)\), and set
\[
q:=\kappa\frac mn\le\frac12.
\]
Independently for each observation, draw a hidden scale label:
the noise standard deviation is \(1\) with probability \(q\) and is \(\sigma\) otherwise.
When the labels are hidden, the errors are i.i.d. \(p_{q,\sigma}\).

Let
\[
a:= \frac{c_1}{C_2}\frac dn,
\]
where \(c_1\) is the constant in \cref{thm:main-hel-lb} and $C_2$ is the constant in \cref{lem:modulus}. Choose \(\sigma\) according to \cref{lem:modulus} with parameter $a$. \cref{thm:main-hel-lb} then implies that with probability at least $3/8$, any estimator incurs error of order
\begin{equation}
\begin{cases}
a^{3/2}q^{-2},&q\le a^{3/4},\\
a^{1/2}q^{-2/3},&q\ge a^{3/4}.
\end{cases}
\label{eq:q-rate}
\end{equation}

We now relate this error in the i.i.d. setting to the Subset-of-Signals setting where at least $m$ scale labels are $1$.
If \(N_1\sim\mathrm{Bin}(n,q)\) denotes the number of observations with $\sigma_i = 1$,
then \(\E N_1=\kappa m\), and a Chernoff bound gives
\[
\Pp(N_1<m)\le e^{-c\kappa m}\le\frac{1}{8}.
\]
for sufficiently large $\kappa$.
For any estimator, the failure probability on i.i.d. samples from $p_{q,\sigma}$, is the average of its failure probabilities
over the hidden scale vectors. Since our use of \cref{thm:main-hel-lb} gave a failure probability of at least $\frac{3}{8}$, then subtracting the event \(\{N_1<m\}\) shows that there exists a fixed scale vector with
at least \(m\) entries having variance \(1\), where the estimator fails with probability at least $\frac{3}{8}-\frac{1}{8} = \frac{1}{4}$.

We now substitute our values into the error radius given by \eqref{eq:q-rate}. Since \(q\asymp m/n\) and \(a\asymp d/n\), the two quantities in
\eqref{eq:q-rate} are
\[
a^{3/2}q^{-2}
\asymp
\left(\frac dn\right)^{3/2}\left(\frac nm\right)^2
=
\frac{\sqrt n\,d^{3/2}}{m^2},
\]
and
\[
a^{1/2}q^{-2/3}
\asymp
\sqrt{\frac dn}\left(\frac nm\right)^{2/3}.
\]
The branching point \(q\asymp a^{3/4}\) is equivalent to
\[
m\asymp d^{3/4} n^{1/4}.
\]
Changing the branch point by a fixed multiplicative constant changes the displayed piecewise
rate by only a fixed multiplicative constant, so this proves the theorem for
\(m\le n/(2\kappa)\).

For the \(m>n/(2\kappa)\) case, we will simply consider the instance where all variances are equal to 1.  The classical lower bound in this setting is \(c\sqrt{d/n}\) (this would also follow from invoking \cref{thm:main-hel-lb}).  By choosing \(c_0\) in the theorem statement sufficiently
small, the transition \(c_0 d^{3/4} n^{1/4}\) lies below \(n/(2\kappa)\).  Hence, this case corresponds to the second branch, and
\[
\sqrt{\frac dn}\left(\frac nm\right)^{2/3}
\le (2\kappa)^{2/3}\sqrt{\frac dn}.
\]
Thus, this lower bound yields our desired result.
\end{proof}

%% file: additional-sims.tex
\section{Additional simulation figures}\label{sec:extra-sims}
In this section, we provide additional simulation figures for planted linear regression. We find these figures less insightful than \cref{fig:planted-recovery}, but we still include them for completeness.
The experiment setup is the same as \cref{subsec:planted-sims}, and each data point corresponds to 100 repetitions. 

\begin{figure}[H]
\centering
\begin{subfigure}{.32\linewidth}
    \centering
    \includegraphics[width=\linewidth]{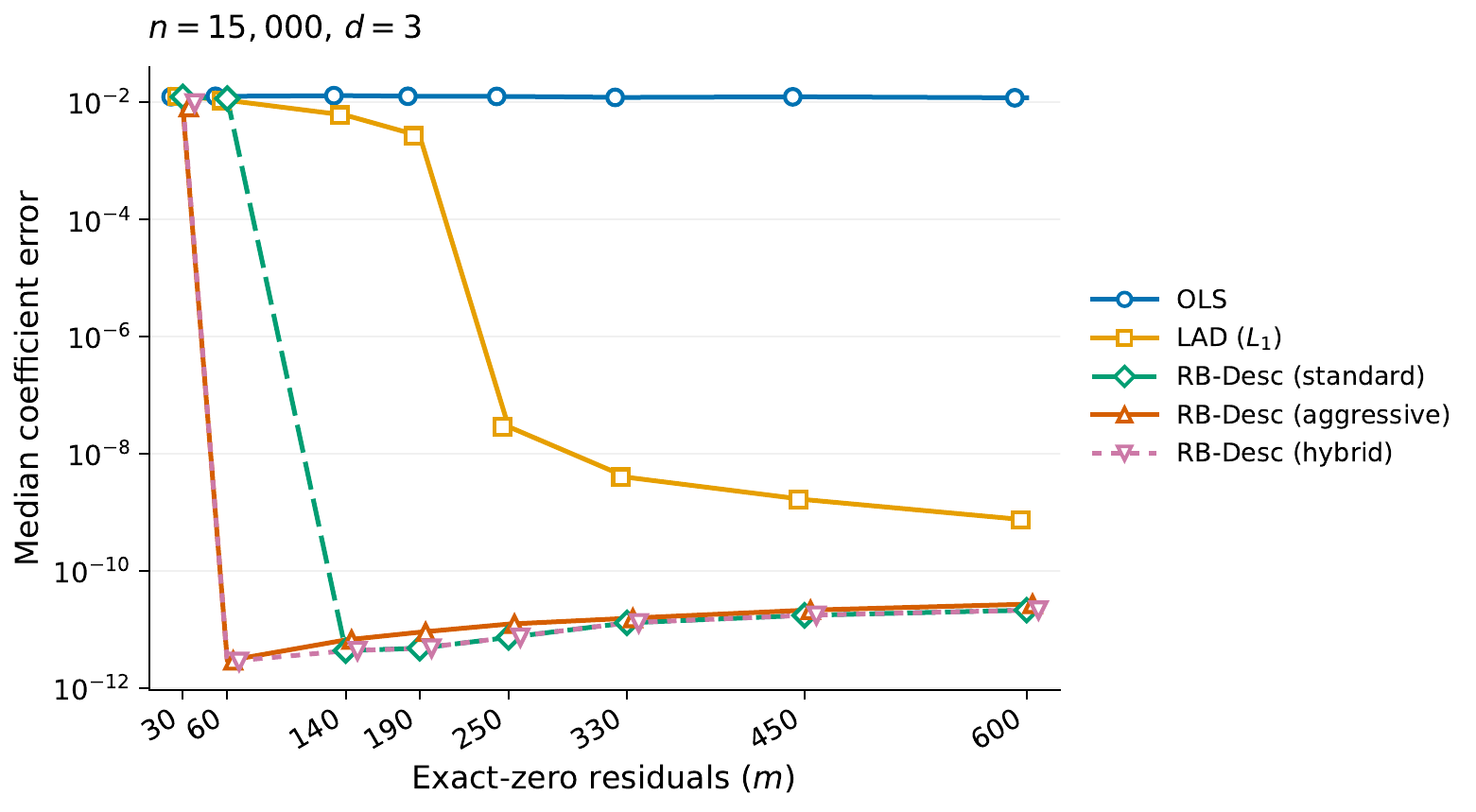}
    \caption{}
\end{subfigure}
    \hfill
\begin{subfigure}{.32\linewidth}
    \centering
    \includegraphics[width=\linewidth]{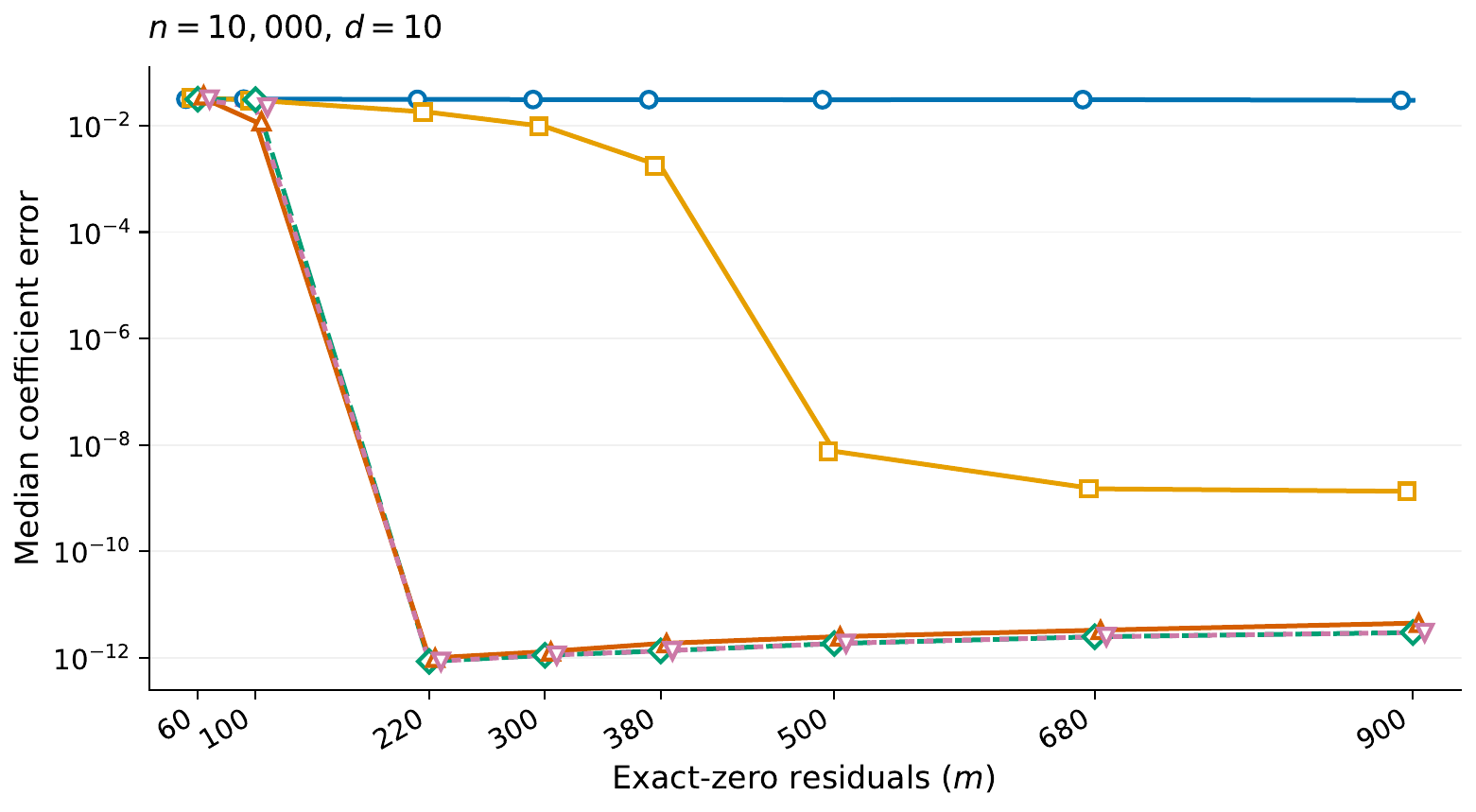}
    \caption{}
\end{subfigure}
   \hfill
\begin{subfigure}{.32\linewidth}
    \centering
    \includegraphics[width=\linewidth]{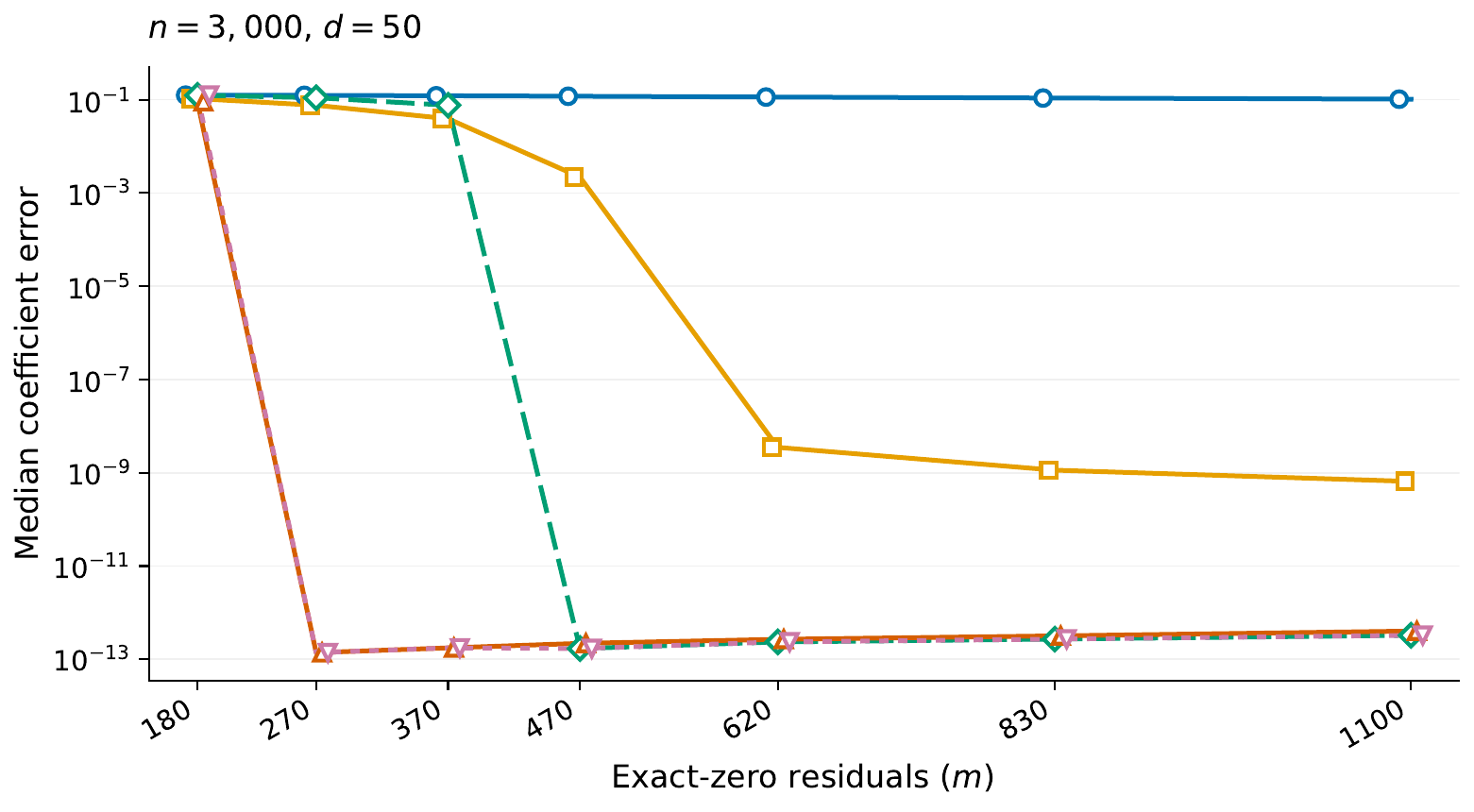}
    \caption{}
\end{subfigure}

{\caption{Median $\norm{\wh\beta-\beta}_2$ error for planted linear regression. }
\label{fig:planted-error}}
\end{figure}

\begin{figure}[H]
\centering
\begin{subfigure}{.32\linewidth}
    \centering
    \includegraphics[width=\linewidth]{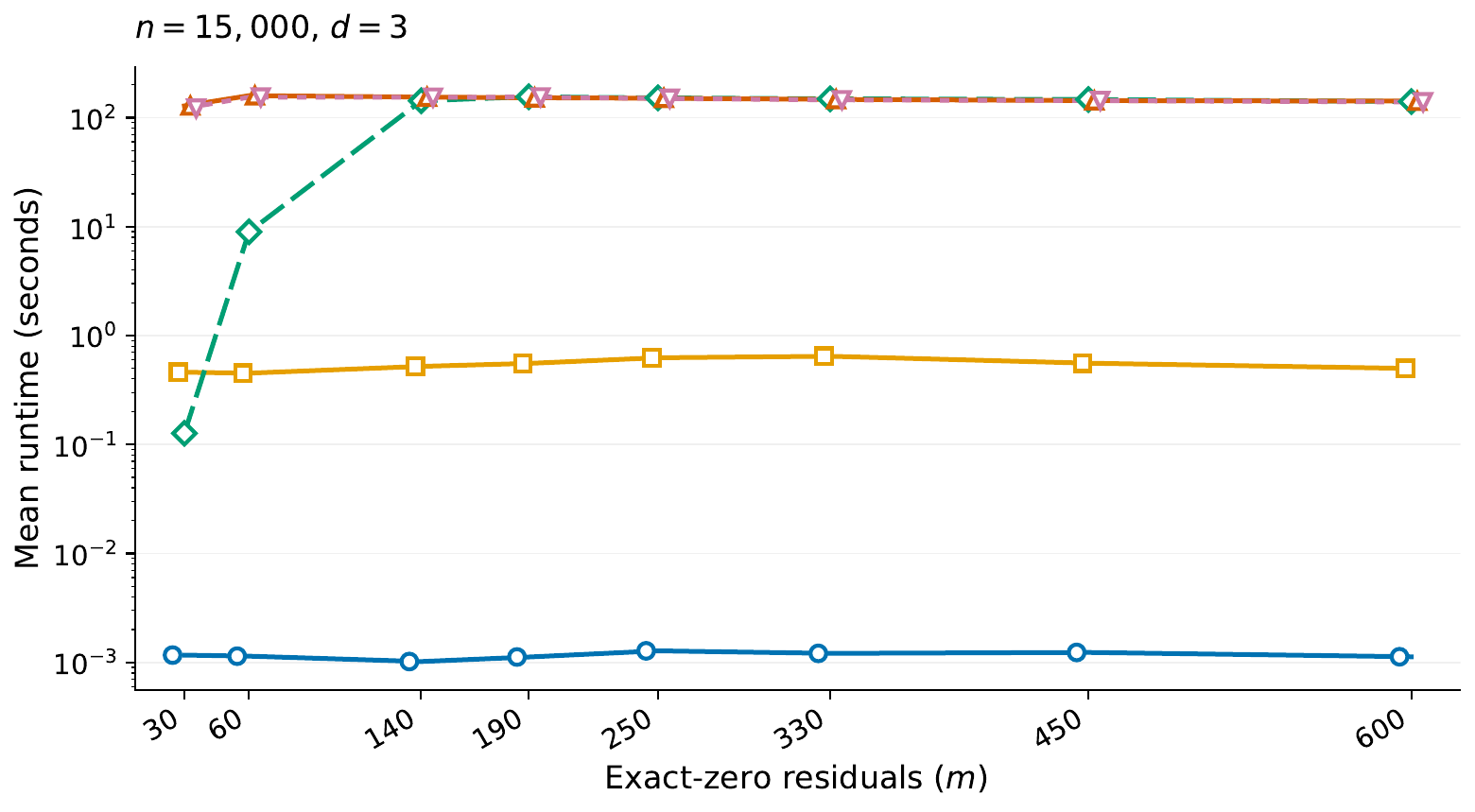}
    \caption{}
\end{subfigure}
    \hfill
\begin{subfigure}{.32\linewidth}
    \centering
    \includegraphics[width=\linewidth]{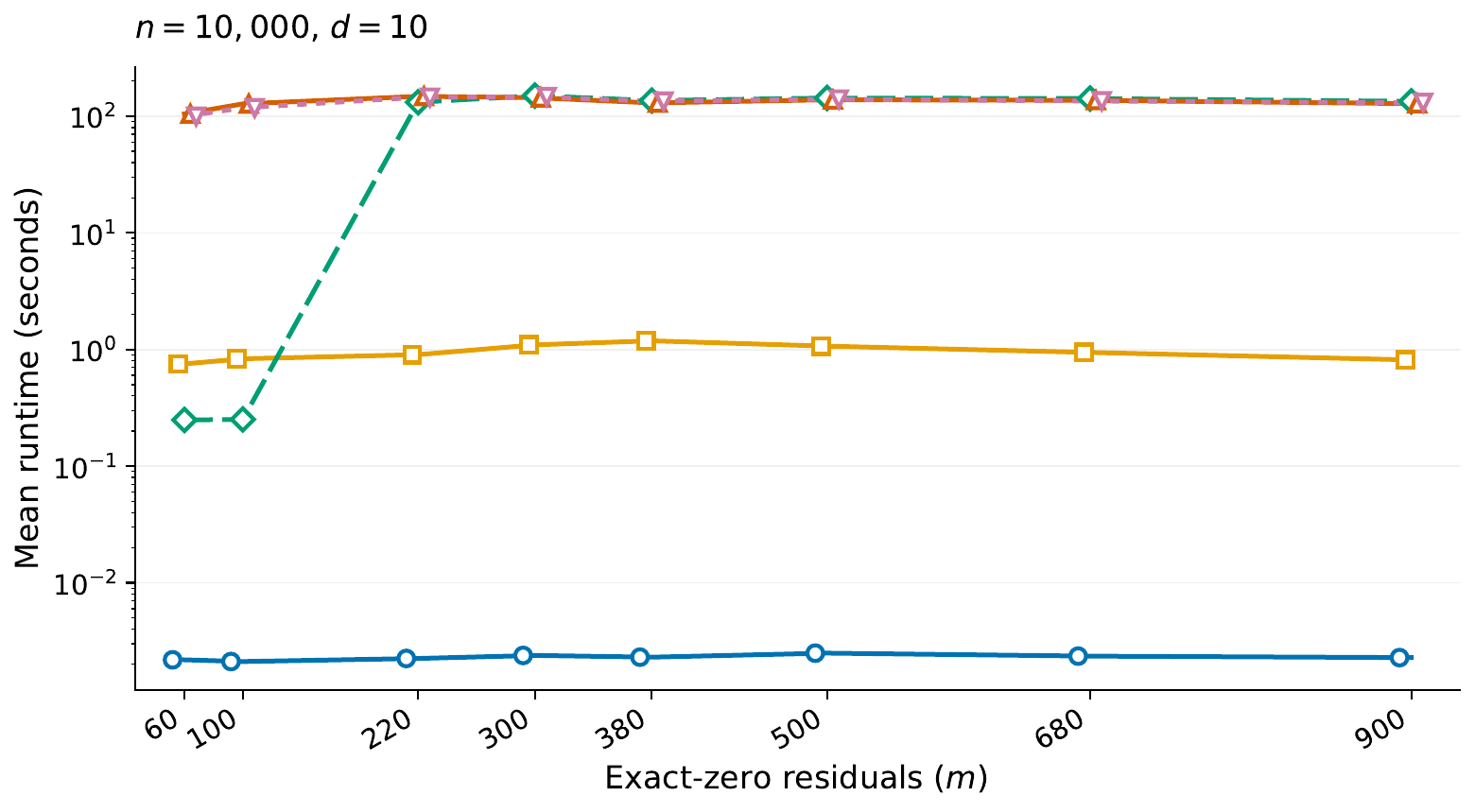}
    \caption{}
\end{subfigure}
   \hfill
\begin{subfigure}{.32\linewidth}
    \centering
    \includegraphics[width=\linewidth]{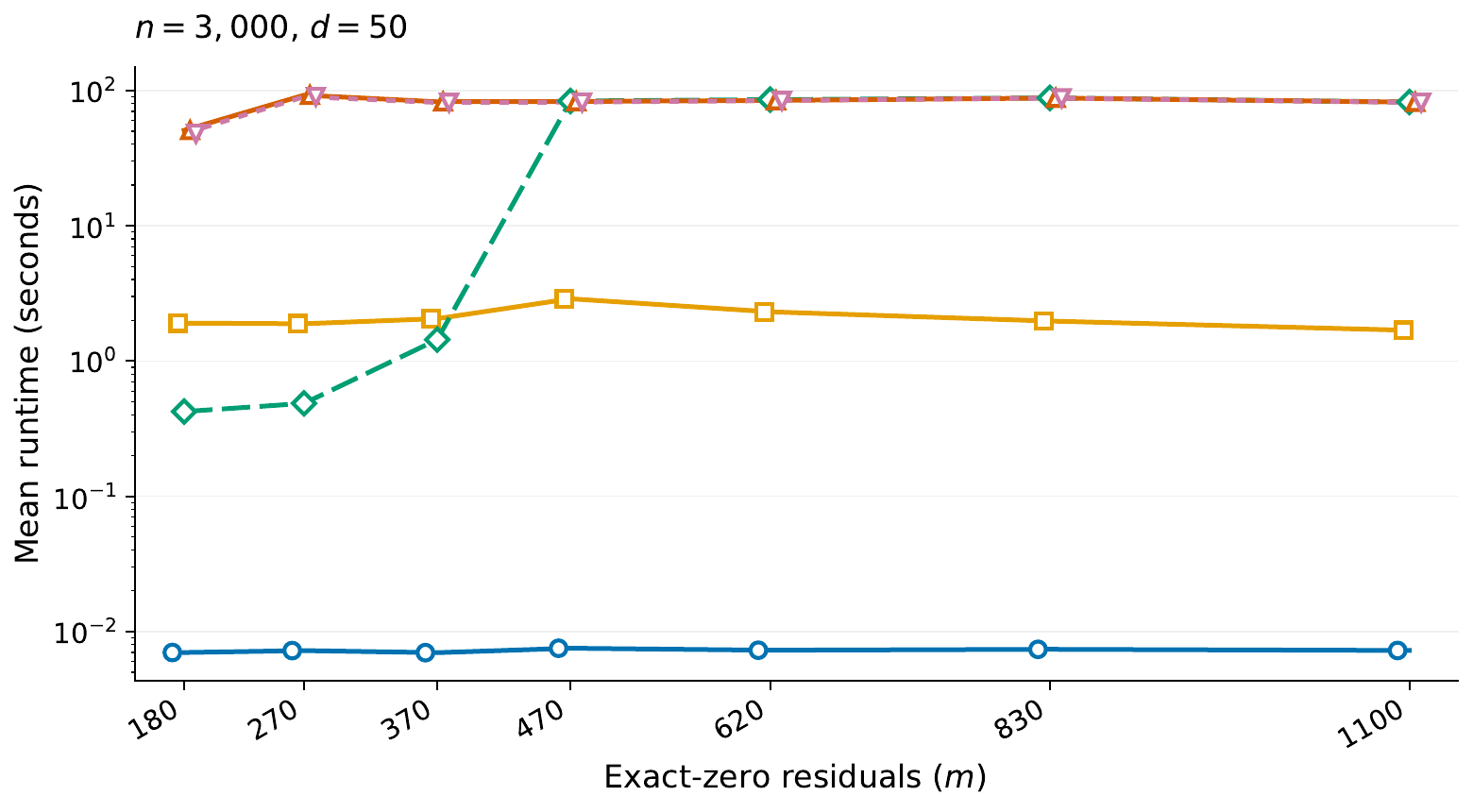}
    \caption{}
\end{subfigure}

{\caption{Average runtime for planted linear regression.}
\label{fig:planted-runtime}}
\end{figure}

\cref{fig:planted-error} illustrates the median error of the estimators as $m$ varies. 
The differences in median error between LAD and $\operatorname{RB-Desc}$ in the large-$m$ regime are likely only due to numerical precision differences. Plotting the median error seemed more informative than the mean error, as the latter was functionally a (harder to read) transformation of the likelihood of exact recovery (seen in \cref{fig:planted-recovery}). \cref{fig:planted-runtime} plots the average runtime of the estimators as $m$ varies.

%% file: unif-conv.tex
\section{Proof of uniform convergence results}
\subsection{Proof of \cref{prop:sample-event}}\label{app:empirical-process}
In this section, we use relatively standard techniques to prove the desired uniform convergence result. We begin by stating a VC-style result (\cref{lem:appendix-vc}), then in \cref{subsec:residual-specialization} we show how this implies our desired uniform convergence result (\cref{prop:sample-event}), and finally in \cref{subsec:selected-sum-proof} we prove \cref{lem:appendix-vc}.

\begin{lemma}[VC-selected counts and feature sums]\label{lem:appendix-vc}
Let $W_i=(X_i,Y_i)$ be independent random elements of $\R^d\times\R$.
Assume that $X_1,\ldots,X_n$ are i.i.d. copies of a covariate $X$; the conditional laws of $Y_i$
given $X_i$ may differ with $i$.  Let $\calA$ be a pointwise measurable class of
measurable subsets of $\R^{d+1}$ with VC dimension at most $V\ge1$.  For
$A\in\calA$, define
\[
 N(A)=\sum_{i=1}^n\one\{W_i\in A\},\qquad
 Q(A)=\sum_{i=1}^n\Pp(W_i\in A),
\]
and
\[
 Z(A)=\sum_{i=1}^n
 \bigl(X_i\one\{W_i\in A\}-\E[X_i\one\{W_i\in A\}]\bigr).
\]
For $\varepsilon\in(0,1/2)$, put
\[
 \ell_\varepsilon=\max\{1,\log(en/\varepsilon)\},
 \qquad D=d+V.
\]
There is a universal constant $C$ such that, with probability at least
$1-\varepsilon$, simultaneously for all $A\in\calA$,
\begin{equation}
 |N(A)-Q(A)| \le C\bigl(\sqrt{VQ(A)\ell_\varepsilon}+V\ell_\varepsilon\bigr). \label{eq:appendix-relative-count}
\end{equation}
Further, if $X$ satisfies 
\eqref{eq:well-conditioned} and \eqref{eq:subg}, then also simultaneously for all $A \in \calA$,
\begin{equation}
 \norm{Z(A)}_2 \le C_{K,\kappa}\ell_\varepsilon^2
 \bigl(\sqrt{D(N(A)+D)}+D\bigr).
 \label{eq:appendix-vector-sum}
\end{equation}
\end{lemma}

\subsubsection{Concluding \cref{prop:sample-event}}\label{subsec:residual-specialization}

\begin{proof}[Proof of \cref{prop:sample-event}]
Let $W_i=(X_i,y_i)$, and let $\mathcal A$ be the union of the following three
residual-set classes, with $b\in B_2(B)$ and $w>0$ varying over their full
ranges:
\[
\begin{split}
 &\{(x,y):\abs{y-x^\top b}\le w\},\\
 &\{(x,y):0<y-x^\top b\le w\},\qquad
   \{(x,y):-w\le y-x^\top b<0\}
\end{split}
\]
Since each set is the intersection of two halfspaces in $\R^{d+1}$, then $\operatorname{VC}(\mathcal A)\le Cd$. 
By applying \cref{lem:appendix-vc} to $\mathcal A$ with failure probability
$\delta/2$, then for every $A\in\mathcal A$,
\begin{align}
 |N(A)-Q(A)|
 &\le C\left(\sqrt{d\ell_{\delta/2} Q(A)}+d\ell_{\delta/2}\right),
 \label{eq:residual-count-input}\\
 \norm{Z(A)}_2
 &\lesssim \ell_{\delta/2}^2
 \left(\sqrt{d(N(A)+d)}+d\right).
 \label{eq:residual-vector-input}
\end{align}
We now focus separately on the three claims of \cref{prop:sample-event}.

\paragraph{Counts.}
For a two-sided residual slab, $N(A)=A_w(b)$ and $Q(A)=Q_w(b)$.  Using
\[
 |A-Q|\le C(\sqrt{aQ}+a)
 \quad\Longrightarrow\quad
 A\le2Q+Ca,\qquad Q\le2A+Ca
\]
with $a=d\ell_{\delta/2}$ proves \eqref{eq:count-comp}, after using $\ell_{\delta/2}\le L$ (where the constant in $L$ was chosen to be sufficiently large).

\paragraph{Scores.}
For a given $w$, set
\[
 A_+(b,w)=\{(x,y):0<y-x^\top b\le w\},\qquad
 A_-(b,w)=\{(x,y):-w\le y-x^\top b<0\}.
\]
In these terms, it is clear that
\[
 S_w(b)-\overline S_w(b)=Z(A_+(b,w))-Z(A_-(b,w)),
 \qquad
 N(A_+(b,w))+N(A_-(b,w))\le A_w(b).
\]
Using the triangle inequality, \eqref{eq:residual-vector-input}, and
$\sqrt{x}+\sqrt{y}\le\sqrt{2(x+y)}$ implies
\[
 \norm{S_w(b)-\overline S_w(b)}_2
 \lesssim \ell_{\delta/2}^2
 \left(\sqrt{d(A_w(b)+d)}+d\right),
\]
which immediately yields \eqref{eq:score-conc} (again using $\ell_{\delta/2} \le L$).

\paragraph{Feature norms.}
A $1/2$-net of $S^{d-1}$ has at
most $5^d$ points.  The sub-Gaussian assumption and
$\Sig\preceq\kappa I$ therefore give
\[
 \Pp\!\left(\norm{X_i}_2>C\sqrt{d+t}\right)\le e^{-t}
\]
for $t\ge 1$. Choosing $t=C\log(2n/\delta)$ and union-bounding over $i$, shows that with probability at least
$1-\delta/2$,
\[
 \max_i\norm{X_i}_2\le C\sqrt{d+\log(n/\delta)}\le\sqrt{Ld},
\]
where the constant in $L$ was chosen to be sufficiently large. 
Intersecting this event with the uniform convergence event over $\mathcal{A}$ yields \eqref{eq:max-X}.
\end{proof}

\subsubsection{Proof of \cref{lem:appendix-vc}}\label{subsec:selected-sum-proof}

The proof follows by combining three prior results that we now describe.

\begin{lemma}[VC covering bound; e.g. van der Vaart and Wellner Section~2.6 \cite{van1996weak}]\label{ext:vc-cover}
Let $\calA$ be a class of sets of VC dimension at most $V$.  There are universal
constants $A_0\ge e$ and $c_0\ge1$ such that, for every finite discrete
probability measure $Q$ and every $0<\eta\le1$,
\begin{equation}\label{eq:external-vc-cover}
 N\!\left(\{\one_A:A\in\calA\},L_2(Q),\eta\right)
 \le \left(\frac{A_0}{\eta}\right)^{c_0V}.
\end{equation}
\end{lemma}

We will leverage a variant of \eqref{eq:external-vc-cover} for weighted classes. An envelope for a class of functions $\mathcal H$ is a non-negative
measurable function $H$ satisfying $|h(u)|\le H(u)$ for every $h \in \mathcal H $ and every $u$.
The indicator class in \cref{ext:vc-cover} has envelope $H=1$.
If $r$ is a fixed measurable function and
\[
 \mathcal H_r=\{r\one_A:A\in\calA\},
\]
then $H=|r|$ is an envelope because
$|r(u)\one_A(u)|\le |r(u)|$ for every $A$.

The covering consequence used below is
\begin{equation}\label{eq:weighted-vc-cover}
 N\!\left(\mathcal H_r,L_2(Q),
           \eta\norm{r}_{L_2(Q)}\right)
 \le \left(\frac{A_0}{\eta}\right)^{c_0V}.
\end{equation}
Here the radius is scaled by the $L_2(Q)$ size of the envelope,
$\norm{|r|}_{L_2(Q)}=\norm r_{L_2(Q)}$.  To verify
\eqref{eq:weighted-vc-cover}, suppose first that
$\norm r_{L_2(Q)}>0$ and define a probability measure $Q_r$ by
\[
 dQ_r(u)=\frac{r(u)^2}{\norm r_{L_2(Q)}^2}\,dQ(u).
\]
\cref{ext:vc-cover} gives indicators
$\one_{A_1},\ldots,\one_{A_N}$, with
$N\le(A_0/\eta)^{c_0V}$, such that for every $A\in\calA$ some $A_j$ obeys
$\norm{\one_A-\one_{A_j}}_{L_2(Q_r)}\le\eta$.  For that same $A_j$,
\begin{align*}
 \norm{r\one_A-r\one_{A_j}}_{L_2(Q)}^2
 &=\int r(u)^2|\one_A(u)-\one_{A_j}(u)|^2\,dQ(u)\\
 &=\norm r_{L_2(Q)}^2
   \norm{\one_A-\one_{A_j}}_{L_2(Q_r)}^2\\
 &\le \eta^2\norm r_{L_2(Q)}^2.
\end{align*}
Thus the same indicator net becomes a weighted-function net after multiplication
by $r$.  If $\norm r_{L_2(Q)}=0$, every member of $\mathcal H_r$ is zero in
$L_2(Q)$, so the claim is immediate.

\begin{lemma}[i.n.i.d. VC maximal inequality; Lemma SA25 of Cattaneo, Feng, and Underwood \cite{cattaneo2024uniform}]\label{ext:cfu-maximal}
Let $U_1,\ldots,U_n$ be independent, with laws $P_1,\ldots,P_n$, and put
$\overline P=n^{-1}\sum_iP_i$.  Let $\mathcal F$ be a pointwise measurable
class with envelope $F$ such that
$0<\norm{F}_{L_2(\overline P)}<\infty$, and suppose that, for some
$A\ge e$ and $V\ge1$,
\begin{equation}\label{eq:external-vc-type}
 \sup_Q N\!\left(\mathcal F,L_2(Q),
                  \eta\norm{F}_{L_2(Q)}\right)
 \le \left(\frac A\eta\right)^V,
 \qquad 0<\eta\le1,
\end{equation}
where the supremum is over finite discrete probability measures.  Suppose
\[
 \sup_{f\in\mathcal F}\norm{f}_{L_2(\overline P)}\le\sigma
 \le\norm{F}_{L_2(\overline P)},
 \qquad
 M=\max_{1\le i\le n}F(U_i).
\]
Then
\begin{equation}\label{eq:external-cfu}
 \E\sup_{f\in\mathcal F}
 \left|\sum_{i=1}^n\{f(U_i)-P_if\}\right|
 \le C\left[
 \sqrt n\,\sigma
 \sqrt{V\log\!\left(\frac{A\norm{F}_{L_2(\overline P)}}{\sigma}\right)}
 +\norm{M}_2\,V
 \log\!\left(\frac{A\norm{F}_{L_2(\overline P)}}{\sigma}\right)
 \right].
\end{equation}
\end{lemma}

We leverage the following to yield a high-probability bound from the in-expectation bound:

\begin{theorem}[Deviation inequality; Theorem~4 of Adamczak~\cite{adamczak2008tail}]\label{ext:adamczak}
Let $U_1,\ldots,U_n$ be independent and let $\mathcal G$ be a countable class
such that $\E g(U_i)=0$ for every $g\in\mathcal G$ and every $i$.  Define
\[
 Z=\sup_{g\in\mathcal G}\left|\sum_{i=1}^ng(U_i)\right|,
 \qquad
 \sigma^2=\sup_{g\in\mathcal G}\sum_{i=1}^n\E g(U_i)^2,
\]
and
\[
 B=\left\|\max_{1\le i\le n}
          \sup_{g\in\mathcal G}|g(U_i)|\right\|_{\psi_1}.
\]
There is a universal constant $C$ such that, for every $u\ge1$, with
probability at least $1-4e^{-u}$,
\begin{equation}\label{eq:external-adamczak}
 Z\le2\E Z+C\bigl(\sigma\sqrt u+Bu\bigr).
\end{equation}
This is the $\psi_1$ case of Theorem~4 of \cite{adamczak2008tail} 
(fix the two auxiliary constants in that theorem and take
$t=C(\sigma\sqrt u+Bu)$).
\end{theorem}

\begin{proof}[Proof of \cref{lem:appendix-vc}]
Choose
\[
 \ell=\ell_\varepsilon,
 \qquad
 D=d+V,
 \qquad
 \ell_q=\max\{1,\log(en/q)\}
 \quad(1\le q\le n).
\]

\paragraph{1. Bounding $|N(A) - Q(A)|$.}
For a fixed $q$, let $\calA_q=\{A\in\calA:Q(A)\le q\}$ and apply
\cref{ext:cfu-maximal} to
$\mathcal F_q=\{\one_A:A\in\calA_q\}$.  By
\cref{ext:vc-cover}, its entropy exponent is $c_0V$; its envelope
is $1$; and
\[
 \sup_{A\in\calA_q}\norm{\one_A}_{L_2(\overline P)}^2
 =\frac1n\sup_{A\in\calA_q}Q(A)\le\frac qn.
\]
Thus \eqref{eq:external-cfu}, with $\sigma=\sqrt{q/n}$, gives
\begin{equation}\label{eq:count-mean-local}
 \E\sup_{A\in\calA_q}|N(A)-Q(A)|
 \le C\left(\sqrt{qV\ell_q}+V\ell_q\right).
\end{equation}
We will turn this into a high-probability guarantee by invoking \cref{ext:adamczak} on the centered indicator class; the variance proxy $\sigma^2 \le q$ (centering can only decrease the second moment) and $B \le 1$. 
Combining \eqref{eq:external-adamczak} and \eqref{eq:count-mean-local} therefore yields, for every $u\ge1$, with failure
probability at most $4e^{-u}$,
\begin{equation}\label{eq:count-tail-local}
 \sup_{A\in\calA_q}|N(A)-Q(A)|
 \le C\left[
 \sqrt{q(V\ell_q+u)}+V\ell_q+u
 \right].
\end{equation}

Apply \eqref{eq:count-tail-local} at dyadic levels
$q=1,2,4,\ldots$, with the last level capped at $n$, and take $u=c\ell$.
For a sufficiently large universal $c$, the union-bound gives a total probability of failure of at most $4(1+\log_2 n)e^{-c\ell}\le\varepsilon/2$. Since $\ell_q\le\ell$ and $V\ge1$, we obtain simultaneously at all dyadic levels
\begin{equation}\label{eq:count-dyadic}
 \sup_{A\in\calA_q}|N(A)-Q(A)|
 \le C\left(\sqrt{Vq\ell}+V\ell\right).
\end{equation}
If $Q(A)\ge1$, choose $q$ so that $Q(A)\le q<2Q(A)$; if $Q(A)<1$, use
$q=1$ and absorb $\sqrt{V\ell}$ into $V\ell$.  This proves
\eqref{eq:appendix-relative-count}.

\paragraph{2. Bounding $\norm{Z(A)}_2$.} Our plan is to bound $\norm{Z(A)}_2$ by invoking our VC-style bounds with the class of functions associated with 1-dimensional projections $r_v(x,y)=\ip{v}{x}$ for a fixed $v \in S^{d-1}$. 
Eventually, we will employ a net of such $v$ to yield our desired $\norm{Z(A)}_2$ bound, but first we focus on just a fixed $v$.
The sub-Gaussian assumption
and $\Sig\preceq\kappa I$ imply
\begin{equation}\label{eq:rv-subg}
 \norm{r_v(W_i)}_{\psi_2}\le C(K,\kappa),
 \qquad
 \norm{r_v}_{L_2(\overline P)}\le C(\kappa).
\end{equation}
As before, we will handle classes separately with $\calA_q$ for different $q$. 
Naturally, it is helpful to note that if an event $E$ has probability $p$, integration of the sub-Gaussian tail gives
\begin{equation}\label{eq:subg-event-moment}
 \E[r_v(W_i)^2\one_E]\le C p\log(e/p).
\end{equation}
To see this directly, choose $t_p=C\log(e/p)$ and use
$\Pp(r_v(W_i)^2>t)\le2e^{-ct}$, for the bound
\[
 \E[r_v(W_i)^2\one_E]
 \le \int_0^{t_p}p\,dt
     +\int_{t_p}^{\infty}2e^{-ct}\,dt
 \le Cp\log(e/p).
\]
Writing $p_i=P_i(A)$ and using concavity of $p\mapsto p\log(e/p)$, we obtain
for every $A\in\calA_q$
\begin{equation}\label{eq:local-multiplier-variance}
 \sum_{i=1}^n\E[r_v(W_i)^2\one_A(W_i)]
 \le C\sum_{i=1}^np_i\log(e/p_i)
 \le Cq\ell_q.
\end{equation}

We now apply \cref{ext:cfu-maximal} to
$\mathcal F_{q,v}=\{r_v\one_A:A\in\calA_q\}$.  We choose the strictly positive
envelope
\[
 F_v(x,y)=1+|r_v(x,y)|,
\]
satisfying $|r_v\one_A|\le F_v$.  The weighted covering bound
\eqref{eq:weighted-vc-cover} gives the required entropy estimate with
normalizing radius $\eta\norm{r_v}_{L_2(Q)}$.  Since
$\norm{r_v}_{L_2(Q)}\le\norm{F_v}_{L_2(Q)}$, the same net is also an
$\eta\norm{F_v}_{L_2(Q)}$-net.  Thus, the entropy condition in \cref{ext:cfu-maximal} holds with exponent $c_0V$ and envelope $F_v$.

Equation~\eqref{eq:local-multiplier-variance} controls the variance term for $\mathcal F_{q,v}$, since
\[
 \sup_{A\in\calA_q}\norm{r_v\one_A}_{L_2(\overline P)}^2
 =\frac1n\sup_{A\in\calA_q}
   \sum_{i=1}^n\E[r_v(W_i)^2\one_A(W_i)]
 \le \frac{Cq\ell_q}{n}
\]
lets us choose
\[
 \sigma_{q,v}
 =\min\left\{\norm{F_v}_{L_2(\overline P)},
               C\sqrt{\frac{q\ell_q}{n}}\right\}.
\]

We may also control the random variable $M$ in \cref{ext:cfu-maximal}:
\begin{equation}\label{eq:max-projection-L2}
 M_v:=\max_{1\le i\le n}F_v(W_i)
 =1+\max_{1\le i\le n}|r_v(W_i)|,
 \qquad
 \norm{M_v}_2\le C\sqrt{\log(en)}.
\end{equation}
The last inequality follows by integrating the union bound for the uniformly
sub-Gaussian variables $r_v(W_i)$.  We now invoke \cref{ext:cfu-maximal} and use that 
$1\le\norm{F_v}_{L_2(\overline P)}\le C$ and $\log\!\left(\frac{A\norm{F_v}_{L_2(\overline P)}}{\sigma_{q,v}}\right) \le C\ell_q$ to conclude
\begin{equation}\label{eq:vector-mean-local}
 \E\sup_{A\in\calA_q}
 \left|\sum_{i=1}^n
 \{r_v(W_i)\one_A(W_i)-P_i(r_v\one_A)\}\right|
 \le C\left[
 \ell_q\sqrt{qV}+V\ell_q\sqrt{\log(en)}
 \right].
\end{equation}

To apply \cref{ext:adamczak}, we define $U_i=(i,W_i)$ and $g_{A,v}(i,w)=r_v(w)\one_A(w)-P_i(r_v\one_A)$.
Note that $\E g_{A,v}(U_i)=0$.  The variance proxy $\sigma^2$ is at most $Cq\ell_q$ by
\eqref{eq:local-multiplier-variance}, because centering can only decrease the
second moment.  A convenient envelope for this centered class is
\[
 G_v(i,w)=|r_v(w)|+P_i|r_v|\le |r_v(w)|+C.
\]
Therefore the envelope term $B$ in \cref{ext:adamczak} satisfies
\begin{equation}\label{eq:max-projection-psi1}
 B\le
 \left\|\max_{1\le i\le n}G_v(i,W_i)\right\|_{\psi_1}
 \le C\sqrt{\log(en)}
\end{equation}
since the maximum of $n$ uniformly sub-Gaussian variables has $\psi_1$ norm
$O(\sqrt{\log(en)})$.  Using \cref{ext:adamczak} and
\eqref{eq:vector-mean-local} now proves,
that with failure probability at most $4e^{-u}$,
\begin{align}
 \sup_{A\in\calA_q}
 \left|\sum_{i=1}^n
 \{r_v(W_i)\one_A(W_i)-P_i(r_v\one_A)\}\right|
 \le C\bigl[
 &\ell_q\sqrt{qV}+V\ell_q\sqrt{\log(en)} \notag\\
 &+\sqrt{q\ell_q u}+\sqrt{\log(en)}\,u
 \bigr].
 \label{eq:vector-tail-local}
\end{align}

We now turn towards using a net argument to imply a bound for $\norm{Z(A)}_2$. 
Let $\calN$ be a $1/2$-net of $S^{d-1}$ with $|\calN|\le5^d$.  Apply
\eqref{eq:vector-tail-local} to every $v\in\calN$ and every dyadic $q$, taking
$u=c(d+\ell)$.  For sufficiently large universal $c$, a union-bound gives failure probability at most
\[
 4(1+\log_2 n)5^d e^{-c(d+\ell)}\le\varepsilon/2.
\]
Since $\ell_q,\log(en)\le\ell$ and $d+\ell\le D\ell$, the right side of
\eqref{eq:vector-tail-local} is at most
\begin{equation}\label{eq:vector-dyadic}
 C\left(\ell\sqrt{Dq}+D\ell^{3/2}\right).
\end{equation}
Using the dyadic $q$ and 
$\norm{z}_2\le2\max_{v\in\calN}|\ip{v}{z}|$ yields that, simultaneously for every
$A\in\calA$,
\begin{equation}\label{eq:appendix-vector-Q}
 \norm{Z(A)}_2
 \le C\left[
 \ell\sqrt{D(Q(A)+1)}+D\ell^{3/2}
 \right].
\end{equation}

\paragraph{3. Replace population mass by the observed count.}
Equation~\eqref{eq:appendix-relative-count} implies
\[
 Q(A)\le2N(A)+CV\ell.
\]
Substituting this into \eqref{eq:appendix-vector-Q}, using $V,d\le D$ and
$\ell\ge1$, yields \eqref{eq:appendix-vector-sum}:
\[
 \norm{Z(A)}_2
 \le C\ell^2\left(\sqrt{D(N(A)+D)}+D\right). \qedhere
\]
\end{proof}
\subsection{Proof of \cref{lem:sqrt-vc}}\label{app:sqrt-unif}
\begin{proof}
This follows nearly immediately from \cref{lem:appendix-vc} and typical algebraic manipulations.

Fix $A\in\calA$ and abbreviate $N=N(A)$, $Q=Q(A)$, and $B = V\ell_\delta$. 

If $Q\le B$, then~\eqref{eq:appendix-relative-count} implies
\[
N\le Q+C(\sqrt{QB}+B)\le C B.
\]
Consequently,
\[
|\sqrt N-\sqrt Q|
\le \sqrt N+\sqrt Q
\le (\sqrt{C}+1)\sqrt B.
\]

If $Q>B$, use the exact identity
\[
|\sqrt N-\sqrt Q|
=\frac{|N-Q|}{\sqrt N+\sqrt Q}
\le \frac{|N-Q|}{\sqrt Q}.
\]
Together with~\eqref{eq:appendix-relative-count}, this yields
\[
|\sqrt N-\sqrt Q|
\le C\left(\sqrt B+\frac{B}{\sqrt Q}\right)
\le 2C\sqrt B.
\]
Both cases prove~\eqref{eq:sqrt-vc} simultaneously for all $A\in\calA$ under the uniform convergence event of \cref{lem:appendix-vc}.
\end{proof}

%% file: misc-app.tex
\section{Additional deferred proofs}
\subsection{Proof of \cref{claim:move-constant}}
The following enables moving constant factors from outside the Hellinger modulus to inside the Hellinger modulus:
\begin{claim}\label{claim:move-constant}
For $t < \frac{1}{9}$ and any symmetric, unimodal $p$, it holds that
\[
\omega_p(t) \le \frac{1}{2} \omega_p(9t)
\]
\end{claim}
\begin{proof}
This follows immediately from \cref{lem:H-translation}.
\end{proof}

\subsection{Proof of \cref{prop:fixed-q}}\label{sec:fixed-q}
\begin{proof}
For $u\in\R^d$, let
\[
L_m(u)=\frac1m\sum_{i=1}^m
\abs{Y_i-X_i^\top(\beta+u)}^q
=\frac1m\sum_{i=1}^m
\abs{\varepsilon_i-X_i^\top u}^q.
\]
Our plan is to show that for all $u$ satisfying $\norm{u}_2 \ge R$ for a later chosen $R$, it holds that
\begin{equation}\label{eq:grad-goal}
\ip{\nabla L_m(u)}{u} > 0.
\end{equation}
By convexity of $L_m$, this would prove $\norm{\wh\beta_q - \beta}_2 \le R$. 

Fix $u\in\R^d$ and let $v_i=X_i^\top u$.  For each observation, define the one-dimensional function
\[
\phi_i(t)=\abs{\varepsilon_i-tv_i}^q,
\qquad 0\le t\le1.
\]
This will be a relevant function, since
\[
\ip{\nabla L_m(u)}{u} = \frac1m \sum_{i=1}^m \phi_i'(1).
\]
We compute the first and second derivatives of $\phi_i$ as
\[
\phi_i'(t)
=-q\,v_i\,\sgn(\varepsilon_i-tv_i)
\abs{\varepsilon_i-tv_i}^{q-1},
\qquad
\phi_i''(t)
=q(q-1)v_i^2\abs{\varepsilon_i-tv_i}^{q-2}.
\]
Consequently, we may bound
\begin{align*}
    \phi_i'(1) &= \phi_i'(0) + \left(\phi_i'(1) - \phi_i'(0) \right) = \phi_i'(0) + q(q-1)v_i^2\int_0^1
\abs{\varepsilon_i-tv_i}^{q-2}\,dt\\
& \ge \phi_i'(0) + q(q-1)\abs{\varepsilon_i}^{q-2}v_i^2
\one\{\varepsilon_i v_i\le0\} \rtag{via $\varepsilon_i v_i\le0 \implies \abs{\varepsilon_i-tv_i} \ge\abs{\eps_i}$}\\
& = -q v_i \sgn(\eps_i) \abs{\eps_i}^{q-1} + q(q-1)\abs{\varepsilon_i}^{q-2}v_i^2
\one\{\varepsilon_i v_i\le0\} \\
\implies &\ip{\nabla L_m(u)}{u} \ge -qG^\top u + q(q-1)C_n(u),
\end{align*}
where
\[
G=\frac1m\sum_{i=1}^m
X_i\sgn(\varepsilon_i)\abs{\varepsilon_i}^{q-1}
\]
relates to the gradient of the objective at the truth ($\nabla L_m(0)$), and 
\[
C_n(u)=\frac1m\sum_{i=1}^m
\abs{\varepsilon_i}^{q-2}(X_i^\top u)^2
\one\{\varepsilon_iX_i^\top u\le0\}
\]
is a curvature-related term. We will bound both terms separately.

\textit{Bounding the gradient term.}
Let
\[
\xi_i=X_i\sgn(\varepsilon_i)\abs{\varepsilon_i}^{q-1}.
\]
Since $p$ is symmetric and $\eps_i$ and $X_i$ are independent, it holds that $\E[\xi_i] = 0$. We may then bound the second moment by
\begin{align*}
\E\norm{G}_2^2
&=\frac1{m^2}\sum_{i=1}^m
\E\norm{\xi_i}_2^2\\
&=\frac1m\E\left[
\abs{\varepsilon}^{2q-2}X^\top X
\right]\\
&=\frac{M_{2q-2}}m
\operatorname{tr}(\Sigma)\\
&\le \frac{\kappa dM_{2q-2}}m.
\end{align*}
Applying Markov's inequality to
$\norm{G}_2^2$ yields
\begin{equation}\label{eq:grad-final}
\Pp\left\{
\norm{G}_2
>6 \sqrt{\frac{\kappa dM_{2q-2}}m}
\right\}
\le\frac1{36}.
\end{equation}

\textit{Bounding the curvature term.} For ease of reading, we will separately prove the required result for this term in \cref{lemma:curve}. We invoke this result with error probability $\delta=1/36$; its sample-size requirement is enforced by our condition \eqref{eq:prop-n-cond} with large enough $C_{K,\kappa}$. Hence, with probability at least $35/36$, it holds that
\begin{equation}\label{eq:curv-final}
C_n(u)
\ge c_{K,\kappa}M_{q-2}\norm u_\Sigma^2 \ge c_{K,\kappa} M_{q-2} \norm{u}_2^2
\qquad\text{for every }u\in\R^d.
\end{equation}
This also implies the rank condition, as otherwise a vector in the null space of $\mathbf{X}$ would violate \eqref{eq:curv-final}.

\textit{Concluding $\ip{\nabla L_m(u)}{u} > 0$ for $\norm{u}_2 \ge R$.} We finally prove the desired bound as in \eqref{eq:grad-goal} for $\norm{u}_2 \ge R$. With probability at least $17/18$, by using the events of \eqref{eq:grad-final} and \eqref{eq:curv-final}, we bound
\begin{align*}
    \ip{\nabla L_m(u)}{u} &\ge -qG^\top u + q(q-1)C_n(u) \\
    &\ge -q \norm{G}_2 \norm{u}_2 + q(q-1)C_n(u) \\
    &\ge -6q \sqrt{\frac{\kappa dM_{2q-2}}m} \norm{u}_2 + q(q-1)c_{K,\kappa} M_{q-2} \norm{u}_2^2\\
    &= \norm{u}_2 \cdot \left(q(q-1)c_{K,\kappa} M_{q-2} \norm{u}_2 - 6q \sqrt{\frac{\kappa dM_{2q-2}}m} \right). 
\end{align*}
Hence, we may conclude that $\ip{\nabla L_m(u)}{u} > 0$ whenever 
\[
\norm{u}_2 \ge C_{K,\kappa} \frac{6q \sqrt{\frac{\kappa dM_{2q-2}}m}}{q(q-1)c_{K,\kappa} M_{q-2}} \impliedby \norm{u}_2 \ge C_{K,\kappa} \sqrt{\frac{d V_p(q)}{m}},
\]
for sufficiently large $C_{K,\kappa}$, proving our desired result.
\end{proof}

We now prove the desired curvature term bound:
\begin{lemma} \label{lemma:curve}
There exists a $C_{K,\kappa} \ge 1$ such that, if
\begin{equation}\label{eq:curve-cond-n}
m\ge C_{K,\kappa}Q_p(q)  d\log(em/\delta),
\end{equation}
then for $\delta \in (0,1/2)$, with probability at least $1-\delta$,
\begin{equation}
\frac1m\sum_{i=1}^m
\abs{\varepsilon_i}^{q-2}(X_i^\top u)^2
\one\{\varepsilon_iX_i^\top u\le0\}
\ge c_{K,\kappa}M_{q-2}u^\top\Sigma u
\end{equation}
for every $u\in\R^d$.
\end{lemma}
\begin{proof}
Observe that the claim is scale-invariant, so it suffices to focus on the case where $\norm{u}_2 = 1$.
Let
\[
\xi = |\eps|^{q-2},
\qquad
\widetilde{X} = -\sgn(\eps)X.
\]
In this notation, recall that
\[
C_n(u) = \frac1m \sum_{i=1}^m \xi_i (\wt{X}_i^\top u)^2 \one\{ \wt{X}_i^\top u \ge 0\}
\]
By symmetry we may use
\[
\E\big[(\widetilde X^\top u)^2\one\{\widetilde X^\top u\ge0\}\big]
=\frac12\E(\widetilde X^\top u)^2 \ge c_{K,\kappa}.
\]
Using the same logic as \cref{lem:annulus} implies there exist constants $a_{K,\kappa},b_{K,\kappa} > 0$ where
\[
\Pp(\wt{X}^\top u \ge a_{K,\kappa}) \ge b_{K,\kappa}
\]
for all $\norm{u}_2=1$. Our proof will only require lower bounding the contributions from terms where $\wt{X}_i^\top u \ge a_{K,\kappa}$. In particular, we will prove
\begin{equation}\label{eq:curve-goal}
    \frac1m \sum_{i=1}^m \xi_i \one\{ \wt{X}_i^\top u \ge a_{K,\kappa} \} \ge c_{K,\kappa} M_{q-2},
\end{equation}
which directly implies our desired bound.
We will prove such a lower bound by invoking uniform convergence guarantees for the counts of certain sets, and showing how we may represent this quantity in a way that corresponds to an integral over certain counts. 

It will be helpful to work with a modified version of this quantity where the values are restricted below a maximum value. Consider
\[
U = 8M_{q-2}Q_p(q),
\qquad
\overline{\xi} = \min(\xi, U).
\]
By definition, $\E[\xi^2] = M_{q-2}^2 Q_p(q)$. We show that $\Ex[\overline{\xi}]$ is not much smaller than $\Ex[\xi]$,
\begin{align*}
\E[\overline\xi]
&=M_{q-2}-\E[(\xi-U)_+] \\
&\ge M_{q-2}-\E\big[\xi\one\{\xi>U\}\big] \\
&\ge M_{q-2}-\frac{\E\xi^2}{U}
=M_{q-2}-\frac{M_{q-2}^2Q_p(q)}{8M_{q-2}Q_p(q)}
=\frac78M_{q-2}.\numberthis\label{eq:overxi-lb}
\end{align*}
We now define the relevant sets on which we will leverage uniform convergence. For $s \in (0,U]$ and $\norm{v}_2=1$, let
\[
E_{v,s}
=
\{\xi\ge s,\ \widetilde X^\top v\ge a_{K,\kappa}\}.
\]
The probabilities of these events satisfy
\begin{equation}
b_{K,\kappa}\Pp(\xi\ge s)
\le
\Pp(E_{v,s})
\le
\Pp(\xi\ge s).
\end{equation}
In terms of the observations $(X,Y)$, the concept class has VC dimension $\lesssim d$. For $q > 2$, $E_{v,s}$ can be written as the union of two intersections of two halfspaces
\[
\{Y-X^\top\beta\ge s^{1/(q-2)},\ X^\top v\le-a_{K,\kappa}\},
\qquad
\{Y-X^\top\beta\le-s^{1/(q-2)},\ X^\top v \ge a_{K,\kappa}\}.
\]
For $q=2$, $E_{v,s}$ is empty when $s > 1$, and for $s \le 1$ is similarly
\[
\{Y-X^\top\beta > 0,\ X^\top v\le-a_{K,\kappa}\},
\qquad
\{Y-X^\top\beta < 0,\ X^\top v \ge a_{K,\kappa}\}.
\]

We now employ the uniform convergence guarantee \eqref{eq:appendix-relative-count} in \cref{lem:appendix-vc}, with failure probability $\delta$. Under this event, simultaneously for every $v$ and $s$,
\begin{equation}
N(E_{v,s}) = \sum_{i=1}^m \one\{ (X_i,Y_i) \in E_{v,s} \}
\ge
mb_{K,\kappa}\Pp(\xi\ge s)
-C_{K,\kappa}
\left(
\sqrt{dm\Pp(\xi\ge s)\ell_\delta}
+d\ell_\delta
\right),
\end{equation}
where
\[
\ell_\delta = \max\{ 1, \log(em/\delta)\}.
\]

We use this to bound
\begin{align*}
&\frac1m\sum_{i=1}^m
\overline\xi_i\one\{\widetilde X_i^\top v\ge a_{K,\kappa}\}
= \frac1m\sum_{i=1}^m
\left( \int_0^U \one\{ \overline{\xi}_i \ge s\} \, ds \right) \one\{\widetilde X_i^\top v\ge a_{K,\kappa}\}
=\frac1m\int_0^U N(E_{v,s})\,ds \\
&\ge \frac1m \int_0^U mb_{K,\kappa}\Pp(\xi\ge s)
-C_{K,\kappa}
\left(
\sqrt{dm\Pp(\xi\ge s)\ell_\delta}
+d\ell_\delta
\right) \, ds \\
&= b_{K,\kappa} \E[\overline\xi] - C_{K,\kappa} \sqrt{\frac{d \ell_\delta}{m}} \left(\int_0^U \sqrt{\Pp(\xi\ge s)} \, ds \right) - \frac{C_{K,\kappa} d\ell_\delta U}{m}
\rtag{via $\Ex[\overline\xi] = \int_0^U \Pp(\xi \ge s) \, ds$} \intertext{Using Cauchy-Schwarz gives $\int_0^U\sqrt{\Pp(\xi\ge s)}\,ds \le \left(U\int_0^U\Pp(\xi\ge s)\,ds\right)^{1/2} = \sqrt{U\,\E\overline\xi}
\le \sqrt{8Q_p(q)} M_{q-2}$, and hence we may continue}
& \ge b_{K,\kappa} \E[\overline\xi] - C_{K,\kappa} M_{q-2} \sqrt{\frac{d \ell_\delta Q_p(q)}{m}}  - \frac{C_{K,\kappa} d\ell_\delta M_{q-2}Q_p(q)}{m} \\
& \ge  \frac{7 b_{K,\kappa}}{8}M_{q-2} - C_{K,\kappa} M_{q-2} \sqrt{\frac{d \ell_\delta Q_p(q)}{m}}  - \frac{C_{K,\kappa} d\ell_\delta M_{q-2}Q_p(q)}{m} \rtag{via \eqref{eq:overxi-lb}}\\
& \ge \frac{ b_{K,\kappa}}{2} M_{q-2}, \rtag{via \eqref{eq:curve-cond-n}}
\end{align*}
when the constant in our condition \eqref{eq:curve-cond-n} is chosen sufficiently large. This proves \eqref{eq:curve-goal}, and hence completes our proof.
\end{proof}

\subsection{Proof of \cref{cor:poly-lq}}\label{sec:poly-lq}
\begin{proof}
The condition of \eqref{eq:X-singular} immediately holds with probability at least $17/18$ by invoking \cref{lemma:curve} with $q=2$ and $\delta = 1/18$, since
\[
\sigma_{\text{min}}^2(\mathbf{X}) = \min_{\norm{u}_2 = 1} \sum_{i=1}^m (X_i^\top u)^2 \ge \min_{\norm{u}_2 = 1} \sum_{i=1}^m \abs{\eps_i}^{q-2} (X_i^\top u)^2  \one\{\eps_i X_i^\top u \le 0\} \ge c_{K,\kappa} m.
\]

We now turn towards the implication for the algorithm given by Adil, Kyng, Peng, and Sachdeva \cite[Theorem 1.1]{adil2024fast}. Recall 
\[
\wh\beta_q = \argmin_{b \in \R^d} \norm{\mathbf{X}b - \mathbf{Y}}_q^q.
\]
We invoke their Theorem 1.1 to yield a $\wh A_q$ where
\[
\norm{\mathbf{X} \wh A_q - \mathbf{Y}}_q^q \le \norm{\mathbf{X} \wh\beta_q - \mathbf{Y}}_q^q + \eps,
\]
for a later-chosen $\eps > 0$. We now briefly set up our invocation in their notation. Consider the equivalent optimization problem
\[
\min_{r \in \R^m, \, \beta \in \R^d : r+ \mathbf{X}\beta = \mathbf{Y}} \norm{r}_q^q,
\]
where $r$ represents the residuals. In their notation, we represent $x = (r \quad \beta)$ and let
\[
\mathbf{A} = [I_m \quad \mathbf{X}],\qquad
b = \mathbf{Y},\qquad
\mathbf{d} = \mathbf{M} = 0,\qquad
\mathbf{N} = \operatorname{diag}(I_m, 0_d),\qquad
\mathbf{x}^{(0)} = (\mathbf{Y} \quad 0_d).
\]
Their algorithm uses $O\left(q m^{1/3} \max\left\{1,  \log\left(\frac{\norm{\mathbf{Y}}_q^q}{\eps}\right) \right\}\right)$ calls to a linear system solver.

We must relate the objective error to the coefficient error $\norm{\wh A_q - \wh\beta_q}_2$, so we may choose the appropriate value of $\eps$.
We use the following form of Clarkson's inequalities
\[
\norm{\frac{f+g}{2}}_q^q + \norm{\frac{f-g}{2}}_q^q \le \frac{1}{2} \left( \norm{f}_q^q + \norm{g}_q^q \right),
\]
with $f = \mathbf{Y}-\mathbf{X}\wh\beta_q$ and $g = \mathbf{Y} - \mathbf{X}\wh A_q$, which yields,
\begin{align*}
    & \norm{\mathbf{Y} - \mathbf{X} \left(\frac{\wh\beta_q + \wh A_q}{2}\right)}_q^q + \norm{\mathbf{X} \left(\frac{\wh\beta_q - \wh A_q}{2}\right)}_q^q \le \frac{1}{2} \left( \norm{\mathbf{Y} - \mathbf{X} \wh\beta_q}_q^q + \norm{\mathbf{Y} - \mathbf{X} \wh A_q}_q^q \right) \\
    \implies &\norm{\mathbf{Y} - \mathbf{X} \wh\beta_q }_q^q + \norm{\mathbf{X} \left(\frac{\wh\beta_q - \wh A_q}{2}\right)}_q^q \le \frac{1}{2} \left( \norm{\mathbf{Y} - \mathbf{X} \wh\beta_q}_q^q + \norm{\mathbf{Y} - \mathbf{X} \wh A_q}_q^q \right) \\
    \implies & \norm{\mathbf{X} \left(\frac{\wh\beta_q - \wh A_q}{2}\right)}_q^q \le \frac{1}{2} \left( \norm{\mathbf{Y} - \mathbf{X} \wh A_q}_q^q  - \norm{\mathbf{Y} - \mathbf{X} \wh\beta_q}_q^q\right) \\
    \implies & \norm{\mathbf{X} \left(\frac{\wh\beta_q - \wh A_q}{2}\right)}_q^q \le \eps/2 \\
    \implies & \norm{\mathbf{X} \left(\frac{\wh\beta_q - \wh A_q}{2}\right)}_2^q m^{1-q/2} \le \eps/2 \rtag{via $\norm{u}_q \ge m^{1/q - 1/2} \norm{u}_2$ for $u \in \R^m$}\\
    \implies & \left(c_{K,\kappa} m \norm{\wh\beta_q - \wh A_q}_2^2\right)^{q/2} m^{1-q/2} \le \eps/2 \rtag{via \eqref{eq:X-singular}}\\
    \implies & 2m \left(c_{K,\kappa} \norm{\wh\beta_q - \wh A_q}_2^2\right)^{q/2}  \le \eps \\
\end{align*}

Hence, $\norm{\wh\beta_q - \wh A_q}_2 \le \gamma$ if we choose
\[
\eps = m \left(c_{K,\kappa} \gamma^2 \right)^{q/2},
\]
for a small enough $c_{K,\kappa}$. Thus, the number of calls to a linear system solver is at most 
\[
O\left(q m^{1/3} \max\left\{1,  \log\left(\frac{\norm{\mathbf{Y}}_q^q}{\eps}\right) \right\}\right) \le O\left(q^2 m^{1/3} \max\left\{1, \log\left(\frac{C_{K,\kappa} m \norm{\mathbf{Y}}_\infty}{\gamma}\right)\right\}\right) \qedhere
\]

\end{proof}

\subsection{Smoothed uniform errors example}\label{sec:smoothed-uniform}
In this section, we discuss the example where the errors are distributed like $\operatorname{Unif}[-1,1] * N(0,n^{-2a})$. We first control the Hellinger modulus:

\begin{lemma}
There exist constants $c,C,\eps_0 > 0$ such that the following holds. Let $\eps \in(0,\eps_0]$, $a \in [0,1]$, and the distribution $p$ is the convolution $\operatorname{Unif}[-1,1] * N(0,n^{-2a})$. Then, the Hellinger modulus satisfies
\[
c \left(\sqrt{n^{-a} \eps} + \eps\right) \le \omega_p(\eps) \le C \left(\sqrt{n^{-a} \eps} + \eps\right).
\]
\end{lemma}
\begin{proof}
Consider $t \in [0,1]$. We will control $H_p(t)$, and then invert this to control the Hellinger modulus. Let $u(x)$ be the density of $\operatorname{Unif}[-1,1]$, $\sigma = n^{-a}$, $\phi_\sigma(x)$ is the density of $N(0,\sigma^2)$, and $\overline \Phi_\sigma(x) = \Pp_{z \sim N(0,\sigma^2)}[z \ge x]$.

\textit{Upper bounding $H_p(t)$.} First, by data processing inequality, we may bound
\[
\hels(p,p_t) \lesssim \tv(u,u_t) = \frac{t}{2}.
\]

For our second upper bound, let $f(x) = \sqrt{p(x)}$. Our bound will require control over $f'(x)$, and hence control over $p'(x)$. We may write $p'(x)$ as
\[
p'(x) = - \int_{-1}^1 \frac{x-y}{\sigma^2} \cdot \frac{1}{2}\phi_\sigma(x-y) \, dy.
\]
Using Cauchy-Schwarz, we may bound,
\begin{align*}
p'(x)^2 &\le \left(\int_{-1}^1 \frac{(x-y)^2}{\sigma^4}  \cdot \frac{1}{2}\phi_\sigma(x-y) \, dy \right) \left(\int_{-1}^1 \frac{1}{2}\phi_\sigma(x-y) \, dy \right)\\
& \le  \left(\int_{-\infty}^\infty \frac{z^2}{\sigma^4}  \cdot \frac{1}{2}\phi_\sigma(z) \, dz \right) p(x) = \frac{p(x)}{2 \sigma^2} \\
\implies & \abs{p'(x)} \lesssim \sqrt{p(x)}/\sigma.
\end{align*}
This further implies $\norm{f'}_\infty \lesssim 1/\sigma$. We may thus bound the squared Hellinger distance by
\begin{align*}
H_p(t) &\lesssim \norm{f - f(\cdot - t)}_2^2 \le \norm{f - f(\cdot - t)}_\infty \norm{f - f(\cdot - t)}_1 \le \left( t \norm{f'}_\infty \right) \cdot \left( t \norm{f'}_1 \right) \\
& \lesssim \frac{t^2}{\sigma} \norm{f'}_1 = \frac{t^2}{\sigma} \cdot 2f(0) \lesssim \frac{t^2}{\sigma}.
\end{align*}
Together, both upper bounds yield
\[
H_p(t) \lesssim \min\left\{ t, \frac{t^2}{\sigma} \right\}.
\]

\textit{Lower bounding $H_p(t)$.}
Consider the tail halfline $A=(1,\infty)$, and let $\alpha = \Pp_{x \sim p}(x \in A)$, $\beta = \Pp_{x \sim p_t}(x \in A)$. Using data processing inequality,
\[
H_p(t) \ge \hels(\operatorname{Bern}(\alpha),\operatorname{Bern}(\beta)) \ge \frac{1}{2} \left(\sqrt{\alpha} - \sqrt{\beta} \right)^2 = \frac{(\beta - \alpha)^2}{2 (\sqrt{\alpha} + \sqrt{\beta})^2} \gtrsim \frac{(\beta - \alpha)^2}{\alpha + \beta}.
\]
The relevant quantities are
\[
\alpha = \frac{1}{2} \int_0^2 \overline \Phi_\sigma(x) \, dx = \frac{1}{2} \int_0^2 \overline \Phi_1(x/\sigma) \, dx \lesssim \sigma + \int_\sigma^2 \phi_1(x/\sigma) \, dx \lesssim \sigma,
\qquad
\beta - \alpha = \int_{1-t}^1 p(x) \, dx \asymp t.
\]
This immediately yields the bound
\[
H_p(t) \gtrsim \frac{t^2}{\sigma + t} \asymp \min \left\{ t, \frac{t^2}{\sigma} \right\}.
\]

\textit{Concluding the Hellinger modulus control.} We have shown that for any $t \in [0,1]$ it holds that
\[
H_p(t) \asymp \frac{t^2}{\sigma + t}.
\]
Thus, for $\eps \le \eps_0$ for a sufficiently small constant $\eps_0$, the largest possible $t$ is
\[
t \asymp \sqrt{\sigma \eps} + \eps. \qedhere
\]
\end{proof}

Now that we have control over the Hellinger modulus, we may apply our results to this setting. 

Our lower bound (\cref{thm:main-hel-lb}) implies that any estimator must incur error
\[
\gtrsim \max\left\{ \sqrt{\frac{d}{n^{1+a}}}, \frac{d}{n} \right\},
\]
with probability at least $\frac{3}{8}$, for standard Gaussian covariates.

Under the conditions of our upper bound (\cref{thm:k1}) and $n \ge C_{K,\kappa} d \log^4(64n/\delta)$ for a sufficiently large $C_{K,\kappa} \ge 1$, then an estimator may incur error at most
\[
C_{K,\kappa} \polylog(n,1/\delta) \cdot \max\left\{ \sqrt{\frac{d}{n^{1+a}}}, \frac{d}{n} \right\},
\]
with probability at least $1-\delta$.